%% file: wg_calderon.tex
\documentclass[a4paper]{article} 
\usepackage[pdfauthor={%
Martin Halla%
},pdftitle={%
The Calder\'{o}n operator for time-harmonic Maxwell's equations in closed waveguides%
}, unicode,colorlinks=true,linkcolor=black,citecolor=black,urlcolor=black,pagebackref=false,breaklinks]{hyperref}
\input{header-paper.tex}
\input{macros-general.tex}

\input{macros-emwg.tex}
\title{%
The \Calderon operator for time-harmonic Maxwell's equations in closed waveguides%
\thanks{%
The author acknowledges funding from Deutsche Forschungsgemeinschaft (DFG, German Research Foundation), projects 541433971 and 258734477 - SFB 1173.
}
}
\author{%
Martin Halla%
\thanks{%
Karlsruhe Institute of Technology,
Institute for Applied and Numerical Mathematics,
Englerstra\ss e 2,
76131 Karlsruhe, Germany
(martin.halla@kit.edu)
}
}

\begin{document}

\date{}
\maketitle
\begin{abstract}
\noindent
We consider the time-harmonic Maxwell's equations in a closed, cylindrical waveguide with longitudinal invariant but heterogeneous permittivity and permeability.
We consider minimal smoothness assumptions, i.e., the material parameters have to be in $L^\infty$ and the cross section has to be Lipschitz with a possibly non-trivial topology.
First, we show that the associated \Calderon operator is well-defined and bounded in the natural trace spaces $\Hspace_0^{-1/2}(\curlg;\geocs) \to \Hspace_0^{-1/2}(\divg;\geocs)$.
Second, we consider interior Maxwell problems obtained from using the \Calderon operator and analyze their well-posedness.
Furthermore, we report some additional results on the location and group velocities of the waveguide wavenumbers.
\\

\\

\end{abstract}


\section{Introduction}\label{sec:introduction}


Electromagnetic waveguides serve as a mathematical model for a wide range of applications such as antennas, power cables, telecommunication, and optical fiber amplifiers \cite{JaureguiLimpertTunnermann13,ImperialeJoly14,cmacs2003baseline,Pozar11,Marcuse13,SheSmirSmo22}.
The analysis of such models constitutes a very fertile field of mathematical research with several open problems, in particular for non-smooth media as suchch usually appear in engineering applications.

\subsection{Homogeneous waveguides}

When one studies electromagnetic wave propagation in a cylindrical waveguide with homogeneous permittivity and permeability one soon realizes that the modal equations can be reduced to two decoupled scalar equations, and many authors have done so independently.
In particular, the modal eigenvalue problem (EVP) reduces to shifted Laplace EVPs with Dirichlet and Neumann boundary conditions respectively.
We can find this result, e.g., in \cite[Cha.~4]{Cessenat} (see also \cite{IlinskiiKravtsovSveshnikov91,GalishnikovaIlinskii13}), where in addition the \Calderon operator is analyzed and the exterior problem is studied as an evolution problem.
Concerning the well-posedness of the interior problem formulated by means of the \Calderon operator \cite[Cha.~4]{Cessenat} only hints that for smooth enough data a $\div$-stabilized formulation is possible, which satisfies the Fredholm alternative.
A detailed analysis of the latter can be found in \cite{Delitsyn05}.
The Fredholmness of the natural formulation (i.e., without $\div$-stabilization) is proven in \cite{Kim17}.
The case of a truncated semi-cylinder without any geometry and material perturbations is considered in \cite{MelenkDemkowiczHenneking25}, where the stability constant is explicitly estimated in terms of the waveguide length.

\subsection{Heterogeneous waveguides}

In the case of heterogeneous materials the modal EVP does not decouple into two separate problems and the EVP does not admit the structure of the EVP for a selfadjoint operator in a Hilbert space.
Indeed, the original modal EVP involves both the full electric field $(\Et,\En)$ and the full magnetic field $(\Ht,\Hn)$ and is linear in the wavenumber $\beta$.
Several authors have applied distinct reformulations of the EVP for their analysis.
In \cite{Delitsyn00} the material is assumed to be piece-wise $C^\infty$ and the completeness of the fields $(\mathbf{B}_T:=\mu\Ht,\En)$ in a subspace of $\Hspace_0(\divg;\geocs)\times H^1_0(\geocs)$ was established.
Later, this result was leveraged to the completeness of the various components in energy spaces \cite{Delitsyn11b} and $L^2$ spaces \cite{Delitsyn11a}.
A different approach is taken in \cite{SheSmi14a,SheSmi14b} where it is assumed that the permeability is constant and the permittivity is piece-wise constant (with two layers) and a quartic (w.r.t.\ the wavenumber) operator stencil is derived.
The authors prove the completeness of the fields $(\Et,\Ht)$ in $(\Lspace^2(\geocs))^2$, while showing that in general they do not form a Schauder basis.
%
The approach of \cite{HallaMonk24} is to derive a linear EVP (with the squared wavenumber as spectral parameter) for the field $\Et$ and to obtain the completeness in $\Hspace_0(\curlg;\geocs)$.
The common technique of \cite{Delitsyn00,SheSmi14a,SheSmi14b,HallaMonk24} is to apply the Keldysh perturbation theorem, see, e.g., \cite{Markus88}, which provides the asymptotic location of eigenvalues and the completeness of the root space.
For the case of a coaxial step index fiber with small perturbation from the homogeneous case the $\Et$ $(\Ht)$ components of the modes form indeed a Riesz basis in $\Hspace_0(\curlg;\geocs)$ $(\Hspace(\curlg;\geocs))$ and $\Hspace_0^{-1/2}(\curlg;\geocs)$ $(\Hspace^{-1/2}(\curlg;\geocs))$ \cite{Halla26a}.



Concerning the well-posedness/Fredholmness of Maxwell problems we mention the following articles.
In \cite{Delitsyn07,Delitsyn19} and earlier works \cite{BogolyubovDelitsynSveshnikov99,BogolyubovDelitsynSveshnikov00,BogolyubovDelitsynSveshnikov03} the material is assumed to be piece-wise smooth and a solution $(\mathbf{B}_T=\mu \Ht, \En)$ is proved to exist in $L^2_\mathrm{loc}(\mathbb{R},\Hspace_0(\divg,\geocs)\times H^1_0(\geocs))$ for smooth enough data.
%
A different approach is taken in \cite{PlamenevskiPoretski18}, which assumes $W^{1,\infty}$-materials, $\div$-stabilizes the Maxwell system to an elliptic system (which assumes smooth enough data) and applies the framework of elliptic systems to obtain stability in weighted Sobolev spaces.

\subsection{Trace spaces of $\Hspace(\curlF;\geo)$}

In order to construct a \Calderon (or equivalently, the electric-to-magnetic) operator we have to handle the occurring trace spaces of $\Hspace(\curlF;\geo)$.
A complete analysis of those for Lipschitz domains is conducted in \cite{Tartar97} and \cite{BuffaCostabelSheen02}.
For our waveguide problems at hand we have to deal with traces on parts of the boundary of function spaces with partial boundary condition.
An analysis of such for Lipschitz polyhedra can be found in \cite{BuffaCiarlet01a,BuffaCiarlet01b} and \cite{Buffa01} (see also, \cite[Cha.~3.1]{AssousCiarletLabrunie18}).
Additional results on mixed boundary conditions for Hilbert complexes are reported in \cite{BauerPaulySchomburg16,PaulySchomburg22} (see also, \cite{HiptmairPaulySchulz23}).

\subsection{Results and structure of the article}

In the present article we consider very minimal assumptions on the waveguide parameters:
The permittivity and permeability are only required to have $L^\infty$ regularity, the frequency may be complex, and the cross section is only required to be Lipschitz with arbitrary topology.
Under these assumptions we establish in \Cref{thm:sol-ext} the existence and boundedness of the \Calderon operator and prove additional properties in \Cref{thm:calop-decomp}.
This allows us to prove in \Cref{thm:wTc} the Fredholmness of interior problems.
Complementory results on the wavenumbers and modes are reported in \Cref{thm:ev-strip,thm:ev-groupvel,thm:completeness-Et}.

The remainder of this article is structured as follows.
In \Cref{sec:setting} we introduce the considered setting, assumptions and problems to investigate.
In \Cref{sec:modalEVP} we study the eigenvalue problem for the waveguide modes.
In \Cref{sec:Calderon} we derive and analyze the \Calderon operator.
Finally, in \Cref{sec:interior} we study the well-posedness of interior Maxwell problems.
In \Cref{sec:ttrspaces} we review the concept of (partial) tangential traces and carry out certain technical details which we could not find in the literature.
In \Cref{sec:app_operatortheory} we review certain aspects of the functional calculus for sectorial operators \cite{Haase06book} and the relation between unbounded operators and unbounded sesquilinear forms \cite{Kato95}.

\section{Setting}\label{sec:setting}

The notation for Sobolev spaces and traces is specified in \Cref{sec:ttrspaces}.
The notation concerning operator theory is specified in \Cref{sec:app_operatortheory}.

\subsection{Geometry}\label{subsec:geometry}

Let $\geocs\subset\bbR^2$ be a bounded Lipschitz domain.
We allow $\geocs$ to have a non-trivial topology, i.e., it's boundary $\geocsbd:=\partial\geocs$ consists of $\Ngeocsbd+1$, $\Ngeocsbd\in\Nz$ connected components with $\geocsbd_0$ being the outer component.
We set $\geobdtrans:=\geocs\times\{0\}$ and $\geoext:=\geocs\times\Rp$.
Further let $\geoint\subset\bbR^3$ be a bounded Lipschitz domain such that $\geobdtrans=\ol{\geoint}\cap\ol{\geoext}$ and $\geo:=\geoint\cup\geobdtrans\cup\geoext$ is Lipschitz.
We set $\geobddir:=\partial\geoint\setminus\ol{\geobdtrans}$.
We use the notation $\xf=(x,y,z)=(\xt,z)\in\mathbb{R}^3$, $\xt\in\mathbb{R}^2$, $z\in\bbR$.
Let $\nvf$ and $\nvt$ be the outward unit normal vectors of $\partial\geo$ and $\partial\geocs$ respectively, which are related by $\nvf=(\nvt,0)^\top$ on $\partial\geocs\times\Rp$.

\subsection{Parameters}\label{subsec:parameters}

We consider the permittivity and permeability $\eps,\mu\in L^\infty(\geo,\Rp)$ with $\epsmin:=\inf_{\xf\in\geo}\eps(\xf)>0$, $\epsmax:=\sup_{\xf\in\geo}\eps(\xf)$, $\mumin:=\inf_{\xf\in\geo}\mu(\xf)>0$, $\mumax:=\sup_{\xf\in\geo}\mu(\xf)$ and assume that inside $\geoext$ they are independent of the longitudinal variable $z$, i.e., $\eps|_{\geoext}(\xf)=:\epscs(\xt)$, $\epscs\in L^\infty(\geocs)$ and $\mu|_{\geoext}(\xf)=:\mucs(\xt)$, $\mucs\in L^\infty(\geocs)$.
Further let $\ome\in\Cx:=\bbC\setminus\{0\}$ be the angular frequency.

\subsection{Maxwell problems}

First we will study the following exterior problem, in order to justify that the forthcoming radiation condition to be introduced in \Cref{def:RC} leads to a mathematically meaningful problem.
\begin{problem}[exterior]\label{prob:MWext}
Given $\bfE_{T,0}\in\Hspace^{-1/2}_0(\curl_{\geobdtrans};\geobdtrans)$,
find $\Eff\in\Hspace_{\tloc}(\curlF;\geoext)$ such that
\begin{subequations}\label{eq:MWext}
\begin{align}
\label{eq:MWextPDE}
\hspace{30mm}&&\curlF\muinv\curlF\Eff-\ome^2\eps\Eff&=0 &&\text{in }\geoext,&&\hspace{30mm}\\
\label{eq:MWextBC}
&&\Eff_T &= 0 &&\text{on }\partial\geocs\times\Rp,\\
\label{eq:MWextIC}
&&\Eff_T &= \bfE_{T,0} &&\text{on }\geobdtrans,\\
\label{eq:MWextRC}
&&\Eff &&&\text{is outgoing}.
\end{align}
\end{subequations}
\end{problem}
Our initial problem of interest is the following.
\begin{problem}[global]\label{prob:MWradiation}
Given $\Fff\in\Lspace^2(\geo)$ with $\supp\Fff\subset\geoint$,
find $\Eff\in\Hspace_{\tloc}(\curlF;\geo)$
such that
\begin{subequations}\label{eq:MWradiation}
\begin{align}
\label{eq:MWradiationPDE}
\hspace{30mm}&&\curlF\muinv\curlF\Eff-\ome^2\eps\Eff&=\Fff &&\text{in }\geo,&&\hspace{30mm}\\
\label{eq:MWradiationBC}
&&\Eff_T &= 0 &&\text{on }\partial\geo,\\
\label{eq:MWradiationRC}
&&\Eff|_{\geoext} &&&\text{is outgoing}.
\end{align}
\end{subequations}
\end{problem}
By means of a \Calderon operator $\calop$ studied in \Cref{sec:Calderon} we reduce \Cref{prob:MWradiation} to a bounded domain.
\begin{problem}[interior]\label{prob:MWint}
Given $\Fff\in\Lspace^2(\geoint)$,
find $\Eff\in\Hspace(\curlF;\geoint)$ such that
\begin{subequations}\label{eq:MWint}
\begin{align}
\label{eq:MWintPDE}
\hspace{30mm}&&\curlF\muinv\curlF\Eff-\ome^2\eps\Eff&=\Fff &&\text{in }\geoint,&&\hspace{30mm}\\
\label{eq:MWintHBC}
&&\Eff_T&= 0 &&\text{on }\geobddir,\\ 
\label{eq:MWintTBC}
&&\nvf\times\muinv\curlF\Eff&=\calop\Eff_T &&\text{on } \geobdtrans.
\end{align}
\end{subequations}
\end{problem}
The weak form of the previous is the following.
\begin{problem}[interior, variational]\label{prob:MWintweak}
Given $\Fff\in\Lspace^2(\geoint)$,
find $\Eff\in\Hspace_{\bci}(\curlF;\geoint)$ such that
\begin{align}\label{eq:MaxwellVar}
\spl \muinv\curlF\Eff,\curlF\Eff\tf \spr_{\geoint}
-\ome^2\spl \eps\Eff,\Eff\tf \spr_{\geoint}
+\spl \calop \Eff_T,\Eff_T\tf \spr_{\geobdtrans}
=\spl \Fff,\Eff\tf \spr_{\geoint}
\end{align}
for all $\Eff\tf\in\Hspace_{\bci}(\curlF;\geoint)$.
\end{problem}
Note that for simplicity in the above problems the right-hand side $\Fff$ is assumed to be in $\Lspace^2$, but of course our analysis will hold for general $\Fff$ contained in the respective dual space.

\subsection{The waveguide cross section}

Bold face spaces over $\geocs$ are defined to consist of two-dimensional vectors, e.g., $\Lspace^2(\geocs)=(L^2(\geocs))^2$, $\Hspace^1(\geocs)=(H^1(\geocs))^2$, etc..
Since $\geobdtrans$ is flat we use the identification $\geobdtrans\sim\geocs$ and identify the space of tangential vectors of $\geobdtrans$ with $\bbC^2$.
Thus we have the identification of tangential spaces, e.g., $\Lspace^2_T(\geobdtrans)\sim\Lspace^2(\geocs)$, $\Hspace^{1/2}_{T}(\geobdtrans)=\Hspace^{1/2}_{RT}(\geobdtrans)~\sim\Hspace^{1/2}(\geocs)$, etc..
On $\geocs$ we introduce the rotation operator
\[
\rotm:=\left(\begin{array}{cc} 0&-1\\1&0\end{array}\right),
\]
and the surface differential operators
\begin{align*}
\curlg(u_x,u_y)^\top&:=\partial_x u_{y}-\partial_y u_{x},&
\Curlg u&:=(\partial_y u,-\partial_x u)^\top,\\
\divg(u_x,u_y)^\top&:=\partial_x u_{x}+\partial_y u_{y},&
\gradg u&:=(\partial_x u,\partial_y u)^\top,
\end{align*}
for which we note that
\begin{align*}
\rotm^2=-I,\hspace{2mm}
\gradg u=\rotm\Curlg u,\hspace{2mm}
\rotm\gradg u=-\Curlg u,\hspace{2mm}
\curlg (\rotm\bfu)=\divg \bfu,\quad
\divg (\rotm\bfu)=-\curlg \bfu
\end{align*}
for $\bfu=(u_x,u_y)^\top$.
In addition, let $\tv:=\rotm\nvt$ be the unit tangential vector on $\partial\geocs$.
\begin{definition}\label{def:cut-off}
The values $\omega\in\bbC$ for which one of the following equations
\begin{subequations}
\begin{align}
\label{eq:cut-off-dir}
-\divg\mucsinv\gradg u-\omega^2\epscs u&=0 \quad\text{in }\geocs, &u&=0 \quad\text{on }\partial\geocs,\\
\label{eq:cut-off-neu}
-\divg\epscs^{-1}\gradg u-\omega^2 \mucs u&=0 \quad\text{in }\geocs, &\partial_{\nvt}u&=0 \quad\text{on }\partial\geocs,
\end{align}
\end{subequations}
admits a non-trivial solution $u\in H^1(\geocs)$ are called cut-off frequencies of the first and second kind respectively.
\end{definition}
To justify the forthcoming definition \eqref{eq:DefSop} and the remainder of the manuscript subsequent to \eqref{eq:DefSop}, we make the following assumption from \eqref{eq:DefSop} onward.
\begin{assumption}\label{ass:cut-off}
The frequency $\omega$ is not a cut-off frequency of the first kind.
\end{assumption}

\section{The modal eigenvalue problem}\label{sec:modalEVP}

Let $\Eff, \Hff$ denote the three-dimensional electric and magnetic fields, solving the time-harmonic Maxwell's equations (with time dependency convention $e^{-i\ome t}$)
\begin{subequations}\label{eq:MW-3D}
\begin{align}
\hspace{30mm}&&-i\ome \epsilon \Eff &= \phantom{-}\curlF \Hff &\text{in }\geoext,&&\hspace{30mm}\\
&&-i\ome \mu \Hff &= -\curlF \Eff &\text{in }\geoext,
\end{align}
together with the boundary conditions on $\partial\geocs\times\Rp$
\begin{align}
\hspace{30mm}&&\nvf\times\Eff &= 0 &\text{on }\partial\geocs\times\Rp,&&\hspace{30mm}\\
&&\nvf\cdot\mu\Hff &= 0 &\text{on }\partial\geocs\times\Rp.
\end{align}
\end{subequations}
As solution to \eqref{eq:MW-3D} of the form
\begin{align*}
\Eff(\xf)=\Ef(\xt)\exp(i\beta z)\sum_{n=0}^N e_nz^n,
\qquad
\Hff(\xf)=\Hf(\xt)\exp(i\beta z)\sum_{n=0}^N h_nz^n,
\end{align*}
with $N\in\Nz$, $(e_n,h_n)\in\bbC^2$, $n=0,\dots,N$, $(e_N,h_N)\neq(0,0)$, $\beta\in\bbC$ is called a mode of length $N$ with wavenumber $\beta$.

\subsection{The modal eigenvalue problem for $(\Et,\En,\Ht,\Hn)$}

Note that if
$\Uff(\xf)=\Uf(\xt)\exp(i\beta z)$
with $\Uf=(\Ux,\Uy,\Uz)^\top$ and $\Ut:=(\Ux,\Uy)^\top$, $\Un:=\Uz$, then
\begin{align}\label{eq:curlofTPfunction}
\curlF\Uff=\bpm \partial_y \Uz-i\beta \Uy\\-\partial_x \Uz+i\beta \Ux\\ \partial_x \Uy-\partial_y \Ux\epm \exp(i\beta z)
=\bpm \Curlg \Un+i\beta R\Ut  \\ \curlg\Ut\epm \exp(i\beta z).
\end{align}
So with $\Ef=(\Et,\En)^\top$, $\Hf=(\Ht,\Hn)^\top$ the problem to find modes of length zero is equivalent to
\begin{subequations}\label{eq:modal-full}
\begin{align}
\label{eq:modal-full-Et}
\hspace{30mm}&&-i\omega\epscs \Et &= \phantom{-}\Curlg \Hn+i\beta\rotm\Ht &\text{in }\geocs,&&\hspace{30mm}\\
\label{eq:modal-full-Ez}
&&-i\omega\epscs \En &= \phantom{-}\curlg\Ht &\text{in }\geocs,\\
\label{eq:modal-full-Ht}
&&-i\omega\mucs \Ht &=-\Curlg \En-i\beta \rotm\Et &\text{in }\geocs,\\
\label{eq:modal-full-Hz}
&&-i\omega\mucs \Hn &= -\curlg\Et &\text{in }\geocs,\\
\label{eqmodal-full-Ez-bc}
&&\En &=0 &\text{on }\partial\geocs,\\
&&\tv\cdot\Et &=0 &\text{on }\partial\geocs,\\
&&\nvt\cdot\mucs\Ht &=0 &\text{on }\partial\geocs.
\end{align}
\end{subequations}
The modes of non-trivial length arise as the elements of the Jordan chains of the corresponding eigenvalue problem, which reads
\begin{evproblem}\label{evprob:evp-full}
Find
$(\beta,(\Et,\En,\Ht,\Hn)) \in\bbC\times(\Hspace(\curlg;\geocs)\times H^1(\geocs)\times\Hspace(\curlg;\geocs)\times H^1(\geocs))\setminus\{0\}$
which solves \eqref{eq:modal-full}.
\end{evproblem}

\subsection{The modal eigenvalue problem for $(\Et,\Ht)$}

To obtain an observation concerning the location of the wavenumbers we derive a problem for the $(\Et,\Ht)$-components.
Combining \eqref{eq:modal-full-Ez} with \eqref{eqmodal-full-Ez-bc} yields $\curlg\Ht=0$ on $\partial\geocs$.
Now, we eliminate $\En$ and $\Hn$ from \eqref{eq:modal-full} to obtain
\begin{subequations}\label{eq:EH-curl}
\begin{align}
\label{eq:Et}
\hspace{10mm}&&\Curlg(\mucs^{-1} \curlg \Et) -\omega^2\epscs \Et &= \phantom{-}\omega\beta \rotm \Ht &\text{in }\geocs,&&\hspace{10mm}\\
\label{eq:Ht}
&&\Curlg(\epscs^{-1} \curlg \Ht) -\omega^2\mucs \Ht &= -\omega\beta \rotm \Et &\text{in }\geocs,\\
\label{eq:Et-bc}
&&\tv\cdot \Et &=0 &\text{on }\partial\geocs,\\
&&\curlg\Ht &=0 &\text{on }\partial\geocs,\\
\label{eq:Ht-bc}
&&\nvt\cdot\mucs\Ht &=0 &\text{on }\partial\geocs.
\end{align}
\end{subequations}
Now, applying $\divg$ to \eqref{eq:Et} and \eqref{eq:Ht} we compute that
\begin{subequations}\label{eq:curlgs}
\begin{align}
\label{eq:curlgHt}
\hspace{10mm}&&
-\omega^2 \divg(\epscs\Et)&=-\omega \beta\curlg\Ht &\text{in }\geocs,&&\hspace{10mm}\\
\label{eq:curlgEt}
&&-\omega^2 \divg(\mucs\Ht)&=\phantom{-}\omega \beta\curlg\Et &\text{in }\geocs.
\end{align}
\end{subequations}
We multiply \eqref{eq:Ht} and \eqref{eq:Et} with $-\frac{\beta}{\omega}\epscs R$ and $\frac{\beta}{\omega}\mucs\rotm$ respectively and insert \eqref{eq:curlgs} to obtain
\begin{subequations}\label{eq:EH-divg}
\begin{align}
\label{eq:Et-divg}
\hspace{00mm}&&-\epscs\nabla \epscs^{-1} \divg \epscs\Et+\omega \beta\epscs\mucs\rotm\Ht+\beta^2\epscs\Et&=0 &\text{in }\geocs,&&\hspace{00mm}\\
\label{eq:Ht-divg}
&&-\mucs\nabla \mucs^{-1} \divg \mucs\Ht-\omega \beta\epscs\mucs\rotm\Et+\beta^2\mucs\Ht&=0 &\text{in }\geocs.
\end{align}
Furthermore, \eqref{eq:Ht-bc} and \eqref{eq:curlgHt} yield the boundary condition
\begin{align}\label{eq:E-divg-bc}
\divg(\epscs\Et) &=0 \quad\text{on }\partial\geocs.
\end{align}
\end{subequations}
Note that, provided $\beta,\ome\neq0$, the previous steps can all be reversed, whereas the forthcoming relation is in general only one way.
Consider arbitrary constants $\eps_0,\mu_0\geq0$,
Multiplying \eqref{eq:Et}-\eqref{eq:Ht} with $\eps_0\mu_0$ and adding \eqref{eq:Et-divg}-\eqref{eq:Ht-divg} gives that a solution $(\beta,(\Et,\Ht))$ to \eqref{eq:EH-curl} satisfies
\begin{subequations}\label{eq:EH-curl-div}
\begin{align}
\begin{split}
\eps_0\mu_0\Curlg(\mucs^{-1} \curlg \Et) -\epscs\gradg \epscs^{-1} \divg \epscs\Et
+(\beta^2-\omega^2\eps_0\mu_0)\epscs\Et \\
+\omega \beta(\epscs\mucs-\eps_0\mu_0) \rotm\Ht &= 0
\quad\text{in }\geocs,
\end{split}
\\
\begin{split}
\eps_0\mu_0\Curlg(\epscs^{-1} \curlg \Ht) -\mucs\gradg \mucs^{-1} \divg \mucs\Ht
+(\beta^2-\omega^2\eps_0\mu_0)\mucs\Ht \\
-\omega \beta(\epscs\mucs-\eps_0\mu_0)\rotm\Et &= 0
\quad\text{in }\geocs,
\end{split}
\\
\divg(\epscs\Et) = \tv\cdot \Et &=0\quad\text{on }\partial\geocs,\\
\curlg\Ht = \nvt\cdot\Ht &=0\quad\text{on }\partial\geocs.
\end{align}
\end{subequations}
The allure of \eqref{eq:EH-curl-div} is that $\beta^2$ and $\omega^2$ only appear as the combined term $(\beta^2-\omega^2\eps_0\mu_0)$.
For (quasi) homogeneous materials $\epscs\mucs=\eps_0\mu_0$ the equations \eqref{eq:EH-curl-div} for $\Et$ and $\Ht$ decouple, and thus we recognize that solutions to \eqref{eq:EH-curl-div} do not necessarily satisfy \eqref{eq:EH-curl}.
Thus the relation of \eqref{eq:EH-curl} to \eqref{eq:EH-curl-div} is only one-way.
Nevertheless, this suffices to establish the following results.
\begin{theorem}\label{thm:ev-strip}
Let the assumptions of \Cref{subsec:geometry,subsec:parameters} be satisfied.
If $\ome\in\bbR$, then the wavenumbers $\beta$ are symmetric with respect to the real axis and satisfy
\begin{align*}
\hspace{20mm}&&|\Re\beta| &\leq |\omega|\|(\epscs\mucs-\eps_0\mu_0)/\sqrt{\epscs\mucs}\|_{L^\infty(\geocs)}, &\text{if }\Im\beta\neq0,&&\hspace{20mm}\\
&&|\Re\beta| &\leq |\omega|\|\sqrt{\epscs\mucs}\|_{L^\infty(\geocs)}, &\text{if }\Im\beta=0,
\end{align*}
with arbitrary constants $\eps_0,\mu_0\geq0$.
\end{theorem}
\begin{proof}
First, note that if $(\beta,(\Et,\En,\Ht,\Hn))$ is an eigenpair, then $(\overline{\beta},(\overline{\Et},\overline{\En},\overline{\Ht},\overline{\Hn}))$ is so too.
Now, testing \eqref{eq:EH-curl-div} with $(\Et,\Ht)$, taking the imaginary part thereof and dividing by $2\Im(\beta)$ yields
\begin{align*}
0&=|\omega\Re(\spl (\epscs\mucs-\eps_0\mu_0)\rotm\Ht,\Et\spr_{\geocs})
+\Re(\beta)(\spl \epscs\Et,\Et\spr_{\geocs}+\spl\mucs\Ht,\Ht\spr_{\geocs})|\\
&\geq (|\Re(\beta)|-|\omega| \|(\epscs\mucs-\eps_0\mu_0)/\sqrt{\epscs\mucs}\|_{L^\infty(\geocs)})(\spl \epscs\Et,\Et\spr_{\geocs}+\spl\mucs\Ht,\Ht\spr_{\geocs}),
\end{align*}
from which the second claim follows.
For $\Im(\beta)=0$ we test \eqref{eq:EH-divg} with $(\Et,\Ht)$, which yields
\begin{align*}
0&\geq\spl\epscs^{-1}\div(\epscs\Et),\div(\epscs\Et)\spr_\geocs
+\spl\mucs^{-1}\div(\mucs\Ht),\div(\mucs\Ht)\spr_\geocs\\
&-|\beta\omega2\Re(\spl\epscs\mucs\rotm\Ht,\Et\spr_\geocs)|
+\beta^2(\spl\epscs\Et,\Et\spr_\geocs+\spl\mucs\Ht,\Ht\spr_\geocs)\\
&\geq |\beta|(|\beta|-|\omega|\|\sqrt{\epscs\mucs}\|_{L^\infty(\geocs)})
(\spl \epscs\Et,\Et\spr_{\geocs}+\spl\mucs\Ht,\Ht\spr_{\geocs})
\end{align*}
from which the last claim follows.
\end{proof}
In addition, we have the following.
\begin{theorem}\label{thm:ev-groupvel}
Let the assumptions of \Cref{subsec:geometry,subsec:parameters} be satisfied.
Then, for each $\ome\in\Rx$, $\eps_0,\mu_0\geq0$ and simple wavenumber $\beta\in\bbR$ there exists $\delta>0$ such that if $\|\eps_0-\eps\|_{L^\infty(\geocs)}+\|\mu_0-\mu\|_{L^\infty(\geocs)}<\delta$, then $\beta$ is differentiable with respect to $\ome$ and satisfies
\begin{align}\label{eq:group-phase-vel}
\frac{\beta}{\omega\epsilon_0\mu_0} \frac{d\beta(\omega)}{d\omega}
&=\frac{\spl\epsilon\Et,\Et\spr_\geocs+\spl\mu\Ht,\Ht\spr_\geocs - \frac{\beta}{\omega\epsilon_0\mu_0} \Re(\spl(\epsilon\mu-\epsilon_0\mu_0)\rotm\Ht,\Et\spr_\geocs)}%
{\spl\epsilon\Et,\Et\spr_\geocs+\spl\mu\Ht,\Ht\spr_\geocs +\frac{\omega}{\beta} \Re(\spl(\epsilon\mu-\epsilon_0\mu_0)\rotm\Ht,\Et\spr_\geocs)} >0,
\end{align}
where $(\Et,\Ht)$ is the associated eigenelement to $\beta$.
Also, if $\beta$ is only semi-simple, but admits locally differentiable branches $\ome\mapsto (\beta_n(\ome),(\Ef_{T,n}(\ome),\Hf_{T,n}(\ome))$ of eigenpairs, then \eqref{eq:group-phase-vel} holds respectively.
\end{theorem}
\begin{proof}
Note that if $\ome,\beta\in\bbR$, then the eigenelement $(\Et,\Ht)$ can be chosen such that $(\Et,\Ht)=(\Re\Et,\Re\Ht)$.
Deriving \eqref{eq:EH-curl} and \eqref{eq:EH-divg}, testing \eqref{eq:EH-curl} and \eqref{eq:EH-divg} each with $(\Et,\Ht)$, and adding both quantities yields \eqref{eq:group-phase-vel} where our assumptions ensure that the nominator and denominator therein are positive.
In the semi-simple case, the proof proceed identically.
\end{proof}
%
%
\begin{remark}\label{rem:beta-real-ome-compl}
If $\ome\in\bbC\setminus\bbR$, then it follows from \eqref{eq:modal-EtEz} that each wavenumber $\beta$ satisfies $\beta\notin\bbR$.
\end{remark}

\subsection{The modal eigenvalue problem for $\Et$}

We plug \eqref{eq:modal-full-Hz}, \eqref{eq:modal-full-Ht} into \eqref{eq:modal-full-Et}; and apply $\curlg\mu^{-1}$ to \eqref{eq:modal-full-Ht} and plug in \eqref{eq:modal-full-Ez}.
Thus we obtain
\begin{subequations}
\label{eq:modal-EtEz}
\begin{align}
\Curlg\mucs^{-1}\curlg\Et-\omega^2\epscs\Et+\beta^2\mucs^{-1}\Et+i\beta\mucs^{-1}\gradg\En &= 0 \quad\text{in }\geocs,\\
-\divg\mucs^{-1}\gradg\En-\omega^2\epscs\En+i\beta\divg(\mucs^{-1}\Et) &= 0 \quad\text{in }\geocs,\\
\tv\cdot\Et &=0 \quad\text{on }\partial\geocs,\\
\En &=0 \quad\text{on }\partial\geocs.
\end{align}
\end{subequations}
We formulate \eqref{eq:modal-EtEz} in weak form:
\begin{align}\label{eq:modal-EtEz-weak}
\begin{aligned}
&(\Et,\En)\in \Hspace_0(\curlg;\geocs)\times H^1_0(\geocs) \quad\text{solves}\\
&\qquad\spl\mucs^{-1}\curlg\Et,\curlg\Et\tf\spr_{\geocs}
-\omega^2\spl\epscs\Et,\Et\tf\spr_{\geocs}
+\beta^2\spl\mucs^{-1}\Et,\Et\tf\spr_{\geocs}\\
&\qquad+\spl\mucs^{-1}\gradg\En,\gradg\En\tf\spr_{\geocs}
-\omega^2\spl\epscs\En,\En\tf\spr_{\geocs}
-i\beta\spl\mucs^{-1}\Et,\gradg\En\tf\spr_{\geocs}
+i\beta\spl\mucs^{-1}\gradg\En,\Et\tf\spr_{\geocs} = 0\\
&\text{for all } (\Et\tf,\En\tf)\in \Hspace_0(\curlg;\geocs)\times H^1_0(\geocs).
\end{aligned}
\end{align}
We introduce
\begin{align*}
\Vspace^0&:=\{\bfv\in\Lspace^2(\geocs)\colon \divg\mucsinv\bfv=0 \text{ in }\geocs\},\\
\Vspace^{1/2}&:=\{\bfv\in\Lspace^2(\geocs)\colon \divg\mucsinv\bfv=0 \text{ in }\geocs,\quad \curlg\bfv\in H^{-1/2}(\geocs),\quad \tv\cdot\bfv=0\text{ on }\partial\geocs\},\\
\Vspace^1&:=\{\bfv\in\Lspace^2(\geocs)\colon \divg\mucsinv\bfv=0 \text{ in }\geocs,\quad \curlg\bfv\in L^2(\geocs),\quad\hspace{6mm} \tv\cdot\bfv=0\text{ on }\partial\geocs\},\\
\Wspace^0&:=L^2(\geocs),\qquad\hspace{5mm}
\Wspace^{1/2}:=H^{1/2}_{00}(\geocs),\qquad\hspace{6.6mm}
\Wspace^1:=H^1_0(\geocs),\\
\Xspace^0&:=\Vspace^0\times \Wspace^0,\qquad
\Xspace^{1/2}:=\Vspace^{1/2}\times \Wspace^{1/2},\qquad
\Xspace^1:=\Vspace^1\times \Wspace^1.
\end{align*}
Next, due to \Cref{lem:tracespace-decomp-orth-Hone} we have the regular decomposition $\Hspace_0(\curlg;\geocs)=\Vspace^1\oplus\gradg\Wspace^1$ and employ henceforth the notation $\Et=\bfv+\ome^{-1}\gradg w$, $\Et\tf=\bfv\tf+\ome^{-1}\gradg w\tf$, $\bfv,\bfv\tf\in\Vspace^1, w,w\tf\in\Wspace^1$.
Thence \eqref{eq:modal-EtEz-weak} becomes
\begin{align}\label{eq:modal-EtwEz-weak}
\begin{aligned}
&(\bfv,w,\En)\in \Vspace^1 \times \Wspace^1 \times \Wspace^1 \quad\text{solves}\\
&\hspace{10mm}\spl\mucs^{-1}\curlg\bfv,\curlg\bfv\tf\spr_{\geocs}
-\omega^2\spl\epscs(\bfv+\ome^{-1}\gradg w),(\bfv\tf+\ome^{-1}\gradg w\tf)\spr_{\geocs}\\
&\hspace{10mm}+\beta^2\spl\mucs^{-1}\bfv,\bfv\tf\spr_{\geocs}
+\beta^2\spl\mucs^{-1}\ome^{-1}\gradg w,\ome^{-1}\gradg w\tf\spr_{\geocs}\\
&\hspace{10mm}+\spl\mucs^{-1}\gradg\En,\gradg\En\tf\spr_{\geocs}
-\omega^2\spl\epscs\En,\En\tf\spr_{\geocs}\\
&\hspace{10mm}-i\beta\spl\mucs^{-1}\ome^{-1}\gradg w,\gradg\En\tf\spr_{\geocs}
+i\beta\spl\mucs^{-1}\gradg\En,\ome^{-1}\gradg w\tf\spr_{\geocs} = 0\\
&\text{for all } (\bfv\tf,w\tf,\En\tf)\in \Vspace^1\times\Wspace^1\times\Wspace^1.
\end{aligned}
\end{align}
Note that \eqref{eq:modal-EtwEz-weak} is quadratic in the eigenvalue parameter $\beta$.
Hence the next step is to regain a problem, which is affine in $\beta^2$.
To this end let
\begin{align*}
\hspace{5mm}&&\Dcurl &\in\BLO(\Vspace^1,L^2(\geocs)),& \Dcurl\Et&:=\curlg\Et,&&\hspace{5mm}\\
&&\mathbf{M}_\sigma&\in\BLO(\Lspace^2(\geocs)),& \mathbf{M}_\sigma\Et&:=\sigma\Et,&\text{ for }\sigma\in L^\infty(\geocs),\\
&&\embLVec&\in\BLO(\Vspace^1,\Lspace^2(\geocs)),& \embLVec\Et&:=\Et,\\
&&\DGrad&\in\BLO(\Wspace^1,\Lspace^2(\geocs)),& \DGrad\En&:=\gradg\En,\\
&&\embL&\in\BLO(\Wspace^1,L^2(\geocs)),& \embL\En&:=\En,\\
&&M_\sigma&\in\BLO(L^2(\geocs)),& M_\sigma\En&:=\sigma\En,&\text{ for }\sigma\in L^\infty(\geocs).
\end{align*}
Then \eqref{eq:modal-EtwEz-weak} reads in operator form:
\begin{align}\label{eq:modal-EtwEz-operator}
\begin{aligned}
\bpm \mathbf{0}\\0\\0 \epm &=
\left[\bpm
\Dcurl^* \Mmui \Dcurl-\ome^2 \embLVec^* \MepsVec \embLVec
& -\ome \embLVec^*\MepsVec\DGrad
& 0\\
-\ome \DGrad^*\MepsVec\embLVec
& -\DGrad^* \MepsVec \DGrad
& 0\\
0 & 0
& \DGrad^* \MmuiVec \DGrad -\omega^2 \embL^* \Meps \embL
\epm \right.\\
&\hspace{-5mm} \left.-\frac{i\beta}{\ome}
\bpm
0
& 0
& 0\\
0
& 0
& -\DGrad^* \MmuiVec \DGrad\\
0 & \DGrad^* \MmuiVec \DGrad
& 0
\epm 
+\beta^2
\bpm
\embLVec^* \MmuiVec \embLVec
& 0
& 0\\
& 
\DGrad^* \MmuiVec \DGrad
& 0\\
0 & 0
& 0
\epm \right]
\bpm \bv\\ w\\\En \epm.
\end{aligned}
\end{align}
Invoking \Cref{ass:cut-off} we can build in \eqref{eq:modal-EtwEz-operator} the Schur complement with respect to $\En$ and \eqref{eq:modal-EtwEz-operator} becomes
\begin{align}\label{eq:modal-Etw-operator}
\begin{aligned}
&\bpm
\Dcurl^* \Mmui \Dcurl-\omega^2 \embLVec^* \MepsVec \embLVec
& -\ome \embLVec^*\MepsVec\DGrad\\
-\ome \DGrad^*\MepsVec\embLVec
& -\DGrad^* \MepsVec \DGrad
\epm
\bpm \bv\\ w \epm=
-\beta^2 \bpm
\embLVec^* \MmuiVec \embLVec &\\
& \Sop
\epm
\bpm \bv\\ w \epm,
\end{aligned}
\end{align}
where
\begin{align}\label{eq:DefSop}
\Sop:=\frac{1}{\ome^2}\Big(\DGrad^* \MmuiVec \DGrad-\DGrad^* \MmuiVec \DGrad (\DGrad^* \MmuiVec \DGrad -\omega^2 \embL^* \Meps \embL)^{-1}
\DGrad^* \MmuiVec \DGrad\Big).
\end{align}
Since the operator $\Sop$ is not quite handsome we aim to get rid of it.
To this end we abbreviate $\Aoptmp:=\DGrad^* \MmuiVec \DGrad$
and $\Koptmp=\ome^2 \embL^* \Meps \embL$.
Applying the formula
\begin{align*}
\Aop-\Aop(\Aop-\Kop)^{-1}\Aop=\Aop\big(I-(\Aop-\Kop)^{-1}\Aop\big)=-\Aop(\Aop-\Kop)^{-1}\Kop
\end{align*}
to $\Aop=\Aoptmp$ and $\Kop=\Koptmp$ the operator $\Sop$ simplifies to
\begin{align*}
\Sop=\frac{-1}{\ome^2}\Aoptmp (\Aoptmp-\Koptmp)^{-1} \Koptmp=
-\DGrad^* \MmuiVec \DGrad
(\DGrad^* \MmuiVec \DGrad-\omega^2 \embL^* \Meps \embL)^{-1}
\embL^* \Meps \embL.
\end{align*}
Note that also
\begin{align*}
\Aop-\Aop(\Aop-\Kop)^{-1}\Aop=-\Kop-\Kop(\Aop-\Kop)^{-1}\Kop,
\end{align*}
which is the expression used in \cite{HallaMonk24}.
However, at this point we deviate from \cite{HallaMonk24}.
Now we multiply \eqref{eq:modal-Etw-operator} with 
\begin{align*}
\diag\Big(I,-(\Aoptmp (\Aoptmp-\Koptmp)^{-1})^{-1}\Big)
=\diag\Big(I,-(\Aoptmp-\Koptmp) \Aoptmp^{-1}\Big)
=\diag\Big(I,-I +\Koptmp \Aoptmp^{-1}\Big)
\end{align*}
to obtain
\begin{align}
\label{eq:evp-tmp}
\begin{aligned}
\bpm
\Dcurl^*\Mmui\Dcurl-\omega^2 \embLVec^* \MepsVec \embLVec
& -\omega \embLVec^*\MepsVec\DGrad\\
\omega (I-
\ome^2 \embL^*\Meps\embL \DGrad^*\Mmui\DGrad) \DGrad^*\MepsVec\embLVec
& (I-
\ome^2 \embL^*\Meps\embL \DGrad^*\Mmui\DGrad) \DGrad^* \MepsVec \DGrad
\epm
&\bpm \bv\\ w \epm \\
=-\beta^2
\bpm
\embLVec^* \MmuiVec \embLVec &\\
& \embL^* \Meps \embL
\epm
&\bpm \bv\\ w \epm.
\end{aligned}
\end{align}
With
\begin{subequations}
\label{eq:def_operators}
\begin{align}
\Aopb&:=\bpm
\Dcurl^* \Mmui \Dcurl +\embLVec^* \embLVec
& \\
& \DGrad^* \MepsVec \DGrad
\epm,\\
\Bopb&:=\bpm
-\ome^2 \embLVec^* \MepsVec \embLVec +\embLVec^* \embLVec
& -\ome \embLVec^*\MepsVec\DGrad\\
\ome \DGrad^*\MepsVec\embLVec
-\ome^3 \embL^* \Meps \embL (\DGrad^* \MmuiVec \DGrad)^{-1} \DGrad^*\MepsVec\embLVec
& \Bopb_{22}
\epm,\\
\Bopb_{22}&:=-\ome^2 \embL^* \Meps \embL (\DGrad^* \MmuiVec \DGrad)^{-1} \DGrad^* \MepsVec \DGrad,\\
\Kopb&:=\bpm
\embLVec^* \MmuiVec \embLVec &\\
& \embL^* \Meps \embL\epm,\\
\lam&:=-\beta^2,
\end{align}
\end{subequations}
\eqref{eq:evp-tmp} reads
\begin{align}\label{eq:modal-Etw-operator-II}
(\Aopb+\Bopb) \bpm \bv\\  w \epm &= \lam \Kopb \bpm \bv\\  w \epm.
\end{align}
Hence we end up with the following eigenvalue problem.
\begin{evproblem}\label{evprob:evp-bounded}
Find $(\lam,(\bv,w)^\top)\in\bbC\times\Xspace^1\setminus\{0\}$ which solves \eqref{eq:modal-Etw-operator-II}.
\end{evproblem}
We collect our equivalence transformations of this section in the following lemma.
\begin{lemma}\label{lem:evp-equi}
If $(\beta,(\Et,\En,\Ht,\Hn))$ 
is a solution to \Cref{evprob:evp-full}, then $(\lam:=-\beta^2,(\bv,w)^\top)$ 
with $\bv+\ome^{-1}\gradg w=\Et$ is a solution to \Cref{evprob:evp-bounded}.
Vice-versa, if $(\lam,(\bv,w)^\top)$ 
is a solution to \Cref{evprob:evp-bounded}, then there exist two linear independent solutions $(\beta:=\pm\sqrt{\lam},(\Et,\En^{(\pm)},\Ht^{(\pm)},\Hn^{(\pm)}))$ 
to \Cref{evprob:evp-full} with $\Et=\bv+\ome^{-1}\gradg w$, and $(\En^{(\pm)},\Ht^{(\pm)},\Hn^{(\pm)})$ being uniquely determined by $\Et$ and the sign of $\beta$.
In particular,
\begin{align}
\En=\frac{\pm i\beta}{\ome} \Big(I +(\DGrad^*\MmuiVec\DGrad-\ome^2\embL^*\Meps\embL)^{-1} \ome^2\embL^*\Meps\embL \Big) w.
\end{align}
\end{lemma}
\begin{proof}
Follows with basic algebraic manipulations.
\end{proof}
The forthcoming statements are a refinement of \cite[Thm.~4]{HallaMonk24}.
\begin{theorem}\label{thm:completeness-X}
Let the assumptions of \Cref{subsec:geometry,subsec:parameters} and \Cref{ass:cut-off} be satisfied.
Then the spectrum of the \Cref{evprob:evp-bounded} consists of an infinite sequence of eigenvalues $(\lambda_n)_{n\in\bbN}$, each with finite algebraic multiplicity, which do not accumulate in $\bbC$, for each $\delta>0$ only finitely many eigenvalues lie outside the sector $\{z\in\bbC\colon |\arg z|<\delta\}$, and the space spanned by the generalized eigenspaces is dense in each space $\Xspace^s$, $s=0,1/2,1$.
%
\end{theorem}
\begin{proof}
We apply the Keldysh perturbation theorem, see, e.g., \cite[Appendix]{Markus88}, in the form stated in \cite[Thm.~2.1]{Halla23MSteklov} to the operator stencil $I-\Aopb^{-1/2}\Bopb\Aopb^{-1/2}-\lambda\Aopb^{-1/2}\Kopb\Aopb^{-1/2}$.
The dimension of the space $X$ is infinite.
It follows from it's definition that $\Aopb^{-1/2}\Kopb\Aopb^{-1/2}$ is selfadjoint, positive and injective.
Due to \Cref{lem:HoneLtwoFO}, \Cref{lem:Vfiniteorder} and \cite[Lem.~2.3]{Halla23MSteklov} the operators $\Aopb^{-1/2}\Bopb\Aopb^{-1/2}$ and $\Aopb^{-1/2}\Kopb\Aopb^{-1/2}$ are of finite order.
Thus the first part of the claims follow.
The remaining part, i.e., the completeness in $\Xspace^s$, $s=0,1/2$ follows from the completeness in $X^1$ and the denseness of $C_{00}^\infty(\geocs)$ and $\bfC_{00}^\infty(\geocs)$ respectively.
\end{proof}
\begin{theorem}\label{thm:completeness-Et}
Let the assumptions of \Cref{subsec:geometry,subsec:parameters} and \Cref{ass:cut-off} be satisfied.
Then the spectrum of the \Cref{evprob:evp-full} consists of an infinite sequence of eigenvalues $(\beta_n)_{n\in\bbN}$, each with finite algebraic multiplicity, which do not accumulate in $\bbC$, for each $\delta>0$ only finitely many eigenvalues lie outside the sectors $\{z\in\bbC\colon |\arg (z\pm\pi/2)|<\delta\}$.
The eigenvalues are symmetric with respect to the origin.
The span of the $\Et$-components of the generalized eigenspaces corresponding to one half\footnote{The eigenspaces corresponding to $\pm\beta_n$ have the same dimension. So \emph{selecting one half} means selecting one half for each pair $\pm\beta_n$.} of the eigenvalues is dense in $\Hspace_0(\curlg;\geocs)$ and $\Hspace_0^{-1/2}(\curlg;\geocs)$.
%
\end{theorem}
\begin{proof}
Follows from \Cref{thm:completeness-X}, \Cref{lem:evp-equi}, \Cref{lem:tracespace-decomp-orth}, \Cref{lem:tracespace-decomp-orth-Hone} and \Cref{rem:ass-top-flat}.
\end{proof}
Next, let us discuss the radiation condition for \Cref{prob:MWext}.
To this end we have to split the wavenumbers in one \emph{outgoing} half and one complementary \emph{incoming} half.
For non-real wavenumbers $\beta$ the ones with $\Im\beta>0$ are defined to be outgoing.
For real wavenumbers $\beta$ the group velocity is the criterion of choice.
To avoid technical discussions occurring for non-trivial Jordan chains, we make the following assumption.
To treat the case of non-trivial Jordan chains a more refined analysis is necessary, and we refer to the discussions in \cite{KostyuchenkoOrazov81,KostyuchenkoShkalikov83} on related quadratic stencils.
\begin{assumption}\label{ass:groupvel}
Let $\ome_0:=\ome$.
We assume that each \emph{real} wavenumber $\beta(\ome_0)$ is semi simple.
Let $N\in\bbN$ be it's multiplicity.
We further assume that there exists $\delta>0$ such that for $\ome\in(\ome_0-\delta,\ome_0+\delta)\not\ni\{0\}$ there exist exactly $N$ (respecting multiplicity) eigenvalues $\beta_n(\ome)$, $n=1,\dots,N$ in $(\ome_0-\delta,\ome_0+\delta)$ and the functions $\ome\to\beta_n(\ome)$, $n=1,\dots,N$ are differentiable.
Finally, we assume that $\frac{d\beta_n(\ome)}{d\ome}\neq0$ for $\ome\in(\ome_0-\delta,\ome_0+\delta)$ and that either $\ome\frac{d\beta_n(\ome)}{d\ome}>0$ for all $n=1,\dots,N$ or $\ome\frac{d\beta_n(\ome)}{d\ome}<0$ for all $n=1,\dots,N$.
\end{assumption}
\begin{definition}[Radiation Condition]\label{def:RC}
Let \Cref{ass:cut-off} and \Cref{ass:groupvel} be satisfied.
Let $2\Npr$, $\Npr\in\Nz$ be the number of real wavenumbers $\pm\beta_n$, $n=1,\dots,\Npr$ with corresponding eigenelements $\mathbf{E}_{T,n}$, where $\ome\frac{d\beta_n(\ome)}{d\ome}>0$.
The wavenumbers $+\beta_n$, $n=1,\dots,\Npr$ are called outgoing.
A function $\Eff\in\Hspace_{\tloc}(\curlF;\geoext)$ is called outgoing (or equivalently, to satisfy the radiation condition), if its tangential part admits a decomposition $\Eff_T(\xf)=\sum_{n=1}^{\Npr} c_n\mathbf{E}_{T,n}(\xt)e^{i\beta_n z}+\Eff_{T,\mathrm{ev}}(\xf)$ with $c_n\in\bbC$, $n=1,\dots,\Npr$ and $\Eff_{T,\mathrm{ev}}\in \Hspace(\curlF;\geoext)$.
\end{definition}
\begin{remark}
Wavenumbers $\beta\in\bbR$ for which the phase velocity $\frac{\beta_n(\ome)}{\ome}$ and group velocity $\frac{d\beta_n(\ome)}{d\ome}$ have the same/opposite sign, i.e., $\pm\frac{\beta_n(\ome)}{\ome}\frac{d\beta_n(\ome)}{d\ome}>0$, are called forward/backward propagating.
\end{remark}

\section{The \Calderon operator}\label{sec:Calderon}

In this section we construct a solution $\Eff$ to \Cref{prob:MWext}.
This result then enables us to define the \Calderon operator
\begin{align*}
\calop\Ef_{T,0}:=-\nvf_{\geoext}\times\mucs^{-1}\curlF\Eff, \qquad \calop\in\BLO(\Hspace_0^{-1/2}(\curl_{\geobdtrans};\geobdtrans), \Hspace_0^{-1/2}(\div_{\geobdtrans};\geobdtrans)),
\end{align*}
appearing in Problems \ref{prob:MWint} and \ref{prob:MWintweak}.
Note that another convenient notion is to work with the Electric-to-Magnetic operator $\EtM$, for which we have the relationship
\begin{align*}
\calop=i\ome\EtM.
\end{align*}
Our approach is to construct in \Cref{subsec:sqroot} a square root of the unbounded version of the operator $K^{-1}(\Aopb+\Bopb)$, and to construct in \Cref{subsec:semigroup} the solution $\Eff$ to \Cref{prob:MWext} by means of a holomorphic semi-group.
In \Cref{subsec:calderon} we establish additional properties of $\calop$, which we will use in \Cref{sec:interior} to analyze \Cref{prob:MWintweak}.
We refer to \Cref{sec:app_operatortheory} for details on the notation and a details on the applied operator theoretical frameworks.

\subsection{Construction of an operator square root}\label{subsec:sqroot}

In order to distinguish between the bounded and unbounded versions of operators we equip the latter with underlines.
Further, we use the abbreviated notation $\cf=(\bfv,w)^\top\in\Xspace^0$ and let $\Xspaceu^0:=\Xspace^0$ be equipped with the equivalent scalar product
\begin{align*}
\spl \cf,\cf\tf \spr_{\Xspaceu^0}:=\spl \mu^{-1}\bfv,\bfv\tf \spr_{\Lspace^2(\geocs)}
+\spl \eps w,w\tf \spr_{L^2(\geocs)},
\qquad \cf=(\bfv,w)^\top,\cf\tf=(\bfv\tf,w\tf)^\top \in \Xspaceu^0,
\end{align*}
which is motivated by $\spl \cf,\cf\tf \spr_{\Xspaceu^0}=\spl K\cf,\cf\tf\spr_{\Xspace^1}$ for $\cf,\cf\tf\in\Xspace^1$.
We consider the elliptic, symmetric, positive sesquilinear form $\ases(\plh,\plh)$ with domain $\Xspace^1\subset\Xspaceu^0$
defined by
\begin{align*}
\ases(\cf,\cf\tf):=\spl \Aopb \cf,\cf\tf \spr_{\Xspace^1},\qquad \cf,\cf\tf \in\Xspace^1.
\end{align*}
By means of \Cref{prop:repr_thm_second} let $\Aopu\colon \opdom(\Aopu)\subset\Xspaceu^0\to\Xspaceu^0$ be the invertible selfadjoint operator associated to $\ases(\plh,\plh)$, i.e.,
\begin{align*}
\ases(\cf,\cf\tf)&=\spl \Aopu \cf,\cf\tf \spr_{\Xspaceu^0}, \qquad \forall \cf\in\opdom(\Aopu),\forall\cf\tf \in\Xspace^1,\\
\ases(\cf,\cf\tf)&=\spl \Aopu^{1/2} \cf,\Aopu^{1/2}\cf\tf \spr_{\Xspaceu^0}, \qquad \forall \cf,\cf\tf \in\opdom(\Aopu^{1/2})=\Xspace^1.
\end{align*}
It follows from the compact embedding $\Xspace^1\hookrightarrow\Xspaceu^0$ that $\Aopu$ has a compact resolvent.
Further, we consider the sesquilinear form $\bses(\plh,\plh)$ with domain $\Xspace^1$ defined by
\begin{align*}
\bses(\cf,\cf\tf):=\spl \Bopb \cf,\cf\tf \spr_{\Xspace^1},\qquad \cf,\cf\tf \in\Xspace^1.
\end{align*}
It follows from the definition \eqref{eq:def_operators} of $\Bop$ that the form $\bses(\plh,\plh)$ is $1/2$-form subordinate to $\ases(\plh,\plh)$, i.e., there exists a constant $C>0$ such that
\begin{align*}
|\bses(\cf)| \leq C \, \ases(\cf)^{p} \, \|\cf\|_{\Xspaceu^0}^{2(1-p)} \quad\forall \cf \in\Xspace^1, \quad p=1/2.
\end{align*}
Thus the form $\qses(\plh,\plh):=\ases(\plh,\plh)+\bses(\plh,\plh)$ is elliptic, and by means of \Cref{prop:repr_thm_first} let $\Qopu\colon\opdom(\Qopu)\subset\Xspaceu^0\to\Xspaceu^0$ be the quasi m-$\theta$-accretive operator ($\theta\in[0,\pi/2)$) with compact resolvent satisfying
\begin{align*}
\qses(\cf,\cf\tf)=\spl \Qopu\cf,\cf\tf \spr_{\Xspaceu^0} \qquad \forall\cf\in\opdom(\Qopu), \forall\cf\tf \in\Xspace^1.
\end{align*}
We would like to work with fractional powers of $\Qopu$.
However, $\Qopu$ is in general not sectorial.
To employ the functional calculus for sectorial operators nonetheless, we work with a modification $\tilde\Qopu$ of $\Qopu$ as follows.
Let $\tilde\beta_n$, $n=1,\dots,\tNpr$, $0\leq\tNpr\leq\Npr$ be the \emph{distinct} real, outgoing wavenumbers with respective Riesz projections $P_n:=\frac{1}{2\pi i} \int_{\Upsilon_n} (\lam-\Qopu)^{-1} \dd\lam$,
where $\Upsilon_n$ is a positively oriented, sufficiently small circle with center $-\tilde\beta_n^2$.
Let $\alpha_n>0$ be pair-wise distinct and $\alpha_n^2\notin\spec(\Qopu)$, $n=1,\dots,\tNpr$.
Then, we define
\begin{align*}
\tilde\Qopu:=\Qopu+\sum_{n=1}^{\Npr} (\tilde\beta_n^2+\alpha_n^2) P_n.
\end{align*}
To estimate the resolvent of $\tilde\Qopu$ we introduce $F(\lam)\in\BLO(\Xspaceu^0)$ defined by
\begin{align*}
F(\lam):=(\lam-\Aopu)^{1/2} \sum_{n=1}^{\Npr} (\tilde\beta_n^2+\alpha_n^2) P_n (\lam-\Aopu)^{1/2}
\end{align*}
and $B(\lam)\in\BLO(\Xspaceu^0)$ as in \Cref{subsubsec:perturbed_form}
\begin{align*}
\spl B(\lam)\cf,\cf \spr_{\Xspaceu^0} = \bses((\lam-\Aopu)^{-1/2}\cf,(\lam-\Aopu)^{-*/2}\cf\tf) \quad\forall\cf,\cf\tf\in\Xspaceu^0.
\end{align*}
We have the following auxiliary estimate.
\begin{lemma}\label{lem:FB-estimate}
For each $\theta\in(0,\pi)$ there exists $C_\theta>0$ such that $\|B(\lam)\|_{\BLO(\Xspaceu^0)}$, $\|F(\lam)\|_{\BLO(\Xspaceu^0)} \leq C_\theta \|(\lam-\Aopu)^{-1/2}\|_{\BLO(\Xspaceu^0)}$ for all $\lam\in\bbC\setminus\overline{\Sector_\theta}$.
\end{lemma}
\begin{proof}
Note that $0\notin\spec(\Aopu)$.
The claim for $F(\lam)$ follows from its definition of the boundedness of the Riesz projections.
For the second claim we note that the definition of $B(\lam)$ goes back to \eqref{eq:def_operators}.
Since each term in the definition of $B$ involves at least one factor $\embLVec$ on the right or $\embLVec^*$ on the left, the claim follows.
\end{proof}
\begin{lemma}\label{lem:tQ-sectorial}
The operator $\tilde\Qopu\colon\opdom(\tilde\Qopu)=\opdom(\Qopu)\subset\Xspaceu^0\to\Xspaceu^0$ is sectorial.
\end{lemma}
\begin{proof}
For $\lam\notin\spec(\tilde\Qopu)\cup\Rp$ we have due to \Cref{subsubsec:perturbed_form} that
\begin{align*}
(\lam-\tilde\Qopu)^{-1}=(\lam-\Aopu)^{-1/2} (1-B(\lam)-F(\lam))^{-1} (\lam-\Aopu)^{-1/2}.
\end{align*}
Due to \Cref{lem:FB-estimate} for each $\theta\in(0,\pi)$ there exists $r_\theta>0$ such that $\|B(\lam)\|_{\BLO(\Xspaceu^0)}+\|F(\lam)\|_{\BLO(\Xspaceu^0)}<1/2$ for all $\lam\in\bbC\setminus(\Sector_\theta\cup B_{r_\theta}(0))$ and hence $\|\lam(\lam-\tilde\Qopu)\|_{\BLO(\Xspaceu^0)} \leq 2\|\lam(\lam-\Aopu)\|_{\BLO(\Xspaceu^0)}$ for such $\lam$.
Since $\spec(\tilde\Qopu)$ is discrete, we can find $\theta\in(0,\pi)$ such that $\spec(\tilde\Qopu)\subset\Sector_\theta$.
Hence the claim follows.
\end{proof}
Now, due to \Cref{lem:tQ-sectorial} the fractal powers of $\tilde\Qopu$ are well-defined (see \Cref{subsubsec:functcalc}) and subsequently we set
\begin{align*}
\Qopu^\gamma:=\tilde\Qopu^\gamma+\sum_{n=1}^{\Npr} (e^{-\sgn(\tilde\beta_n)\gamma\pi i}|\tilde\beta_n|^{2\gamma}-\alpha_n^{2\gamma}) P_n, \quad \gamma=[-1,0\text{$)$}\cup\text{$($}0,1].
\end{align*}
Note that the minus in the exponent of $e^{-\sgn(\tilde\beta_n)\gamma\pi i}$ will later cancel out by considering the operator $-\Qopu^{\gam}$.
\begin{lemma}\label{lem:powerproduct}
It holds that $\Qopu^{\gam_1}\Qopu^{\gam_2}=\Qopu^{\gam_1+\gam_2}$ for all $\gam_1,\gam_2>0$.
\end{lemma}
\begin{proof}
We compute that
\begin{align*}
\Qopu^{\gam_1}\Qopu^{\gam_2}
&=\bigg(\tilde\Qopu^{\gam_1}+\sum_{n=1}^{\Npr} (e^{-\sgn(\tilde\beta_n)\gam_1\pi i}|\tilde\beta_n|^{2\gam_1}-\alpha_n^{2\gam_1}) P_n\bigg)
\bigg(\tilde\Qopu^{\gam_2}+\sum_{n=1}^{\Npr} (e^{-\sgn(\tilde\beta_n)\gam_2\pi i}|\tilde\beta_n|^{2\gam_2}-\alpha_n^{2\gam_2}) P_n\bigg)\\
&=\tilde\Qopu^{\gam_1}\tilde\Qopu^{\gam_2}
+\sum_{n=1}^{\Npr} (e^{-\sgn(\tilde\beta_n)\gam_1\pi i}|\tilde\beta_n|^{2\gam_1}-\alpha_n^{2\gam_1}) (e^{-\sgn(\tilde\beta_n)\gam_2\pi i}|\tilde\beta_n|^{2\gam_2}-\alpha_n^{2\gam_2}) P_n\\
&+\sum_{n=1}^{\Npr} (e^{-\sgn(\tilde\beta_n)\gam_1\pi i}|\tilde\beta_n|^{2\gam_1}-\alpha_n^{2\gam_1}) \alpha_n^{2\gam_2} P_n
+\sum_{n=1}^{\Npr} (e^{-\sgn(\tilde\beta_n)\gam_2\pi i}|\tilde\beta_n|^{2\gam_2}-\alpha_n^{2\gam_2}) \alpha_n^{2\gam_1} P_n\\
&=\tilde\Qopu^{\gam_1+\gam_2}
+\sum_{n=1}^{\Npr} (e^{-\sgn(\tilde\beta_n)(\gam_1+\gam_2)\pi i}|\tilde\beta_n|^{2(\gam_1+\gam_2)}-\alpha_n^{2(\gam_1+\gam_2)}) P_n
=\Qopu^{\gam_1+\gam_2}.
\end{align*}
\end{proof}
\begin{lemma}\label{lem:Q-sqrt-diff}
For each $\gam\in(0,1)$ the operator $\Gopu_\gam:=\Qopu^{-\gam}-\Aopu^{-\gam}$ satisfies $\Aopu^{s_1}\Gopu_{\gam}\Aopu^{s_2}\in\BLO(\Xspace^0)$ for all $s_1,s_2\in[0,1/2]$, $s_1+s_2<1/2+\gam$.
\end{lemma}
\begin{proof}
We compute that
\begin{align*}
\Qopu^{-\gam}-\Aopu^{-\gam}=
\tilde\Qopu^{-\gam}-\Aopu^{-\gam}
+\sum_{n=1}^{\tNpr} (e^{-\sgn(\tilde\beta_n)(-\gam\pi i)}|\tilde\beta_n|^{-2\gam}-\alpha_n^{-2\gam}) P_n
\end{align*}
and by means of \eqref{eq:frac-power-neg} that
\begin{align*}
\tilde\Qopu^{-\gam}-\Aopu^{-\gam}
&=\frac{\sin\gam\pi}{\pi} \int_0^\infty t^{-\gam} \big((t+\Aopu)^{-1}-(t+\tilde\Qopu)^{-1}\big)\,\dd t\\
&=\frac{\sin\gam\pi}{\pi} \int_0^\infty t^{-\gam} (t+\Aopu)^{-1/2} \big(1-(1+B(-t)+F(-t))^{-1}\big) (t+\Aopu)^{-1/2}\,\dd t\\
&=\frac{\sin\gam\pi}{\pi} \int_0^\infty t^{-\gam} (t+\Aopu)^{-1/2} \big(B(-t)+F(-t)\big) (1+B(-t)+F(-t))^{-1} (t+\Aopu)^{-1/2}\,\dd t.
\end{align*}
For each $s\in[0,1/2]$ there exists $C_s>0$ such that $\|\Aopu^{s}(t+\Aopu)^{-1/2}\|_{\BLO(\Xspace^0)} \leq C_s (t+1)^{s-1/2}$.
Hence the claim follows by means of \Cref{lem:FB-estimate}.
\end{proof}
\begin{lemma}\label{lem:Q-sqrt-Ainv-sqrt}
It holds that $\Qopu^{-\gam}\Aopu^{\gam}$, $\Aopu^{\gam}\Qopu^{-\gam}$, $\Qopu^{\gam}\Aopu^{-\gam}$, $\Aopu^{-\gam}\Qopu^{\gam}$ $\in\BLO(\Xspaceu^0)$ for each $\gam\in(0,1/2]$.
\end{lemma}
\begin{proof}
It follows from \Cref{lem:Q-sqrt-diff} that $\Qopu^{-\gam}\Aopu^{\gam}=1+\Gopu_\gam\Aopu^{\gam} \in \BLO(\Xspaceu^0)$ and $\Aopu^{\gam}\Qopu^{\gam}=1+\Aopu^{\gam}\Gopu_\gam \in \BLO(\Xspaceu^0)$.
Since $\Aopu^\del$, $\del>0$ has a compact resolvent it holds that $\Qopu^{-\gam}\Aopu^{\gam}$, $\Aopu^{\gam}\Qopu^{-\gam}$ are Fredholm with index zero.
On the other hand $\Qopu^{-\gam}\Aopu^{\gam}$, $\Aopu^{\gam}\Qopu^{-\gam}$ are injective and hence they have bounded inverses $\Aopu^{-\gam}\Qopu^{\gam}$, $\Qopu^{\gam}\Aopu^{-\gam}$.
\end{proof}
\begin{lemma}\label{lem:Q-sqrt-resolvent-estimate}
For each $\theta\in(0,\pi)$ there exist constants $r_\theta, C_\theta>0$ such that
$\|(\lambda-\Qopu^{1/2})^{-1}\|_{\BLO(\Xspaceu^0)} \leq C_\theta\|(\lambda-\Aopu^{1/2})^{-1}\|_{\BLO(\Xspaceu^0)}$ for all $\lam\in\overline{\Sector_\theta\cup B_{r_\theta}(0)}$.
\end{lemma}
\begin{proof}
Follows along the lines of the proof of \Cref{lem:tQ-sectorial}.
\end{proof}
\begin{corollary}\label{cor:Q-sqrt-quasisectorial}
For each $\lam_0>0$ there exits $\theta_0\in(0,\pi/2)$ such that the operator $\Qopu^{1/2}+\lam_0$ is sectorial with angle $\theta_0$.
\end{corollary}
\begin{proof}
Follows from the location of the spectrum of $\Qopu^{1/2}$ and \Cref{lem:Q-sqrt-resolvent-estimate}.
\end{proof}
\begin{corollary}\label{cor:Q-sqrt-estforboundedsemigroup}
There exists $\theta\in(0,\pi/2)$ such that $\spec(\Qopu^{1/2})\setminus\Sector_\theta\subset i\bbR$ and that
\begin{align*}
\sup_{\lam\in e^{\pm i\theta}\Rpz} \|\lam(\lam-\Qopu^{1/2})^{-1}\|_{\BLO(\Xspaceu^0)}<\infty.
\end{align*}
\end{corollary}
\begin{proof}
Follows from the location of the spectrum of $\Qopu^{1/2}$ and \Cref{lem:Q-sqrt-resolvent-estimate}.
\end{proof}

\subsection{Analysis of the exterior problem}\label{subsec:semigroup}

Due to \Cref{cor:Q-sqrt-quasisectorial} the operator $-\Qopu^{1/2}$ generates an exponentially bounded holomorphic semigroup $T\colon \Sector_\theta\to\Xspaceu^0$ (for some angle $\theta\in(0,\pi/2)$), and due to \Cref{cor:Q-sqrt-estforboundedsemigroup} the restricted group $T\colon (0,+\infty)\to\Xspaceu^0$ is bounded.
In particular,
\begin{align}\label{eq:solution_semigroup}
T(z)\cf_0=\sum_{n=1}^{\tNpr} e^{i\tilde\beta_n z} P_n\cf_0
+\int_{-\partial\Sector_{\theta}} e^{\lam z} (\lam+\Qopu^{1/2})^{-1} \cf_0 \,\dd\lam.
\end{align}
with $\theta\in(0,\pi/2)$ as in \Cref{cor:Q-sqrt-estforboundedsemigroup}, $\cf_0\in\Xspaceu^0$, and $-\Sector_\theta$ is oriented going from $-e^{i\theta}\infty$ to $-e^{-i\theta}\infty$.
\begin{lemma}\label{lem:solution_semigroup}
For each $\cf_0\in\Xspaceu^0$ the function $\cf(z):=T(z)\cf_0\colon (0,\infty)\to \Xspaceu^0$ satisfies
\begin{enumerate}[I)]
\item\label{item:continous} $\Qopu\cf, \Aopu^{1/2}\cf, \cf, \cf', \cf'' \in C\big( (z_1,z_2); \Xspaceu^0 \big)$ for all $0<z_1<z_2<\infty$;
\item\label{item:ODE} $\Qopu\cf(z)-\cf''(z)=0$ for all $z>0$;
\item\label{item:IV} $\lim_{z\to0+} \|\cf(z)-\cf_0\|_{\Xspaceu^0}=0$;
\item\label{item:RC} $\cf(\plh)$ satisfies the radiation condition
\begin{align}
\cf(z)=\sum_{n=1}^{\tNpr} e^{i\tilde\beta_n z} P_n\cf_0+ \cf_\mathrm{ev}(z) \qquad\text{with}\qquad
\|\cf_\mathrm{ev}(z)\|_{\Xspace^1}+\|\cf_\mathrm{ev}'(z)\|_{\Xspace^0}\leq Ce^{-\del z}
\end{align}
for some constants $C,\del>0$.
\end{enumerate}
If $\cf_0\in\Xspace^{1/2}$, then it further holds that
\begin{enumerate}[I)]
\setcounter{enumi}{4}
\item\label{item:Sobolev} $\cf\in L^2\big( (0,z_1); \Xspace^1 \big)$ and $\cf\in H^1\big( (0,z_1); \Xspaceu^0 \big)$ for all $z_1>0$;
\item\label{item:IVplus} $\lim_{z\to0+} \|\cf(z)-\cf_0\|_{\Xspace^{1/2}}=0$.
\end{enumerate}
\end{lemma}
\begin{proof}
\emph{Proof of \ref{item:continous}).}\quad
Since $T(\cdot)$ forms a holomorphic semigroup we have that $\cf\in C\big((z_1,z_2);\Xspaceu^0\big)$ and $\phi'=\Qopu^{1/2}\phi$ in $(z_1,z_2)$.
Exploiting \eqref{eq:solution_semigroup} it follows that $\cf',\cf''\in C\big((z_1,z_2);\Xspaceu^0\big)$ too.
Thence, due to $\Qopu^{1/2}\cf=\cf'$ and $\Qopu\cf=\cf''$ also $\Qopu^{1/2}\cf,\Qopu\cf\in C\big((z_1,z_2);\Xspaceu^0\big)$.
At last $\Aopu^{1/2}\cf \in C\big((z_1,z_2);\Xspaceu^0\big)$ due to $\Aopu^{1/2}\Qopu^{-1/2}\in\BLO(\Xspaceu^0)$ by means of \Cref{lem:Q-sqrt-Ainv-sqrt}.\\
\emph{Proof of \ref{item:ODE}).}\quad
Follows from \ref{item:continous}) and $\cf''=\Qopu^{1/2}\cf'=\Qopu^{1/2}\Qopu^{1/2}\cf=\Qopu\cf$.\\
\emph{Proof of \ref{item:IV}).}\quad
Follows from the semigroup property of $T(\cdot)$.\\
\emph{Proof of \ref{item:RC}).}\quad
There exists $\del_0>0$ such that $\|\lam(\lam-\del+\Qopu^{1/2})^{-1}\|_{\Xspace^0}$ is bounded on $-\partial\Sector_\theta$ for all $\del\in[0,\del_0)$.
Hence $z\mapsto e^{\del z} \int_{-\partial\Sector_{\theta}} e^{\lam z} (\lam+\Qopu^{1/2})^{-1} \cf_0 \,\dd\lam
=\int_{-\partial\Sector_{\theta}+\del} e^{\lam z} (\lam-\del+\Qopu^{1/2})^{-1} \cf_0 \,\dd\lam
=\int_{-\partial\Sector_{\theta}} e^{\lam z} (\lam-\del+\Qopu^{1/2})^{-1} \cf_0 \,\dd\lam
$ is bounded on $z\in(0,\infty)$.
This yields $\|\cf_\mathrm{ev}(z)\|_{\Xspace^0}\leq Ce^{-\del z}$.
The claim for $\|\cf_\mathrm{ev}(z)\|_{\Xspace^1}$ and $\|\cf'_\mathrm{ev}(z)\|_{\Xspace^0}$ follows the same way.\\
\emph{Proof of \ref{item:Sobolev}).}\quad
It suffices to prove the claim for the second term in \eqref{eq:solution_semigroup}.
Since $\Aopu$ is selfadjoint and positive it holds that $\psi_1(z):=\int_{-\partial\Sector_{\theta}} e^{\lam z} (\lam+\Aopu^{1/2})^{-1} \cf_0 \,\dd\lam \in L^2\big( (0,z_1); \Xspace^1 \big) \cap H^1\big( (0,z_1); \Xspaceu^0 \big)$.
On the other hand
\begin{align*}
\psi_2(z)&:=\int_{-\partial\Sector_{\theta}} e^{\lam z} \big( (\lam+\Qopu^{1/2})^{-1} -(\lam+\Aopu^{1/2})^{-1} \big) \cf_0 \,\dd\lam\\
&=\int_{-\partial\Sector_{\theta}} e^{\lam z} (\lam+\Qopu^{1/2})^{-1} (\Aopu^{1/2}-\Qopu^{1/2}) (\lam+\Aopu^{1/2})^{-1} \cf_0 \,\dd\lam\\
&=\int_{-\partial\Sector_{\theta}} e^{\lam z} (\lam+\Qopu^{1/2})^{-1}\Qopu^{1/2}\Aopu^{-1/2}\Aopu^{1/2} (\Qopu^{-1/2}-\Aopu^{-1/2}) \Aopu^{1/2} (\lam+\Aopu^{1/2})^{-1} \Aopu^{-1/2}\Aopu^{1/2}\cf_0 \,\dd\lam,
\end{align*}
where $\Qopu^{1/2}\Aopu^{-1/2}\Aopu^{1/2} (\Qopu^{-1/2}-\Aopu^{-1/2}) \Aopu^{1/2}\in \BLO(\Xspace^0)$ due to \Cref{lem:Q-sqrt-diff,lem:Q-sqrt-Ainv-sqrt}.
Since $\Aopu^{1/2}\cf_0\in\Xspaceu^0$ we estimate that
$\|\psi_2'(z)\|_{\Xspaceu^0}+\|\Aopu^{1/2}\psi_2(z)\|_{\Xspaceu^0} \leq C_1 \int_{-\partial\Sector_{\theta}} (1+|\lam|)^{-5/4}\,\dd\lam<C_2<\infty$ with some $z\in[0,z_1]$-uniform constants $C_1,C_2>0$.
We deduce that $\psi_2',\Aopu^{1/2}\psi_2\in C\big([0,z_1];\Xspaceu^0\big)$ and the claim follows.\\
\emph{Proof of \ref{item:IVplus}).}\quad
We exploit \eqref{eq:solution_semigroup} and note that it suffices to show
\begin{align*}
\lim_{z\to0+} \Big\|\Aopu^{1/4} \int_{-\partial\Sector_{\theta}} e^{\lam z} (\lam+\Qopu^{1/2})^{-1} \cf_0 \,\dd\lam - \Aopu^{1/4} (1-P)\cf_0\Big\|_{\Xspaceu^0}=0
\end{align*}
where $P:=\sum_{n=1}^{\tNpr} P_n$.
We compute that
\begin{align*}
\Aopu^{1/4} \int_{-\partial\Sector_{\theta}} e^{\lam z} (\lam+\Qopu^{1/2})^{-1} \cf_0 \,\dd\lam
=\Aopu^{1/4}\Qopu^{-1/4} \int_{-\partial\Sector_{\theta}} e^{\lam z} (\lam+\Qopu^{1/2})^{-1} \Qopu^{1/4} \cf_0 \,\dd\lam
\end{align*}
and $\Aopu^{1/4} (I-P)\cf_0=\Aopu^{1/4}\Qopu^{-1/4} (1-P)\Qopu^{1/4}\cf_0$.
Due to \Cref{lem:Q-sqrt-Ainv-sqrt} we have $\Aopu^{1/4}\Qopu^{-1/4}\in\BLO(\Xspaceu^0)$.
Thus the claim can be deduced from \ref{item:IV}) applied to $\Qopu^{1/4}\cf_0\in\Xspaceu^0$.
\end{proof}
For $\bfE_{T,0}\in\Hspace^{-1/2}_0(\curl_{\geobdtrans};\geobdtrans)\sim\Hspace^{-1/2}_0(\curl_{\geocs};\geocs)$ let $\cf_0=(\bv_0,w_0)^\top$ be such that $\bfE_{T,0}=\bv_0+\frac{1}{\ome}\gradg w_0$ and $(\bv(z),w(z))^\top:=\cf(z)$ be the solution provided by and \eqref{eq:solution_semigroup} and \Cref{lem:solution_semigroup}, and
\begin{subequations}\label{eq:solution_hcurlloc}
\begin{align}
	\Eff:=(\Et,\En)^\top, \qquad \Et:=\bv+\frac{1}{\ome}\gradg w,
\end{align}
\begin{align}\label{eq:def-En-sol}
\begin{aligned}
\En&:=\sum_{n=1}^{\tNpr} \frac{i\tilde\beta_n}{\ome} e^{i\tilde\beta_n z} (\Aopu_W-\ome^2)^{-1}\Aopu_W P_W P_n \cf_0\\
&+\int_{-\partial\Sector_{\theta}} \frac{\lam}{\ome} e^{\lam z} (\Aopu_W-\ome^2)^{-1}\Aopu_W P_W (\lam+\Qopu^{1/2})^{-1} \cf_0 \,\dd\lam,
\end{aligned}
\end{align}
\end{subequations}
where $P_W$ takes the $\Wspace$-component of a $\Xspace$-vector, i.e., $P_W\cf=w$ for $\cf=(\bv,w)^\top$ and $\Aopu_W$ is the second diagonal block component of $\Aopu$.
\begin{theorem}\label{thm:sol-ext}
Let the assumptions of \Cref{subsec:geometry,subsec:parameters}, \Cref{ass:cut-off} and \Cref{ass:groupvel} be satisfied.
Then, the function $\Eff$ defined by \eqref{eq:solution_hcurlloc} satisfies $\Eff\in\Hspace_{\tloc}(\curlF;\geoext)$ and solves \Cref{prob:MWext}.
\end{theorem}
\begin{proof}
The proof rests on exploiting \Cref{lem:solution_semigroup}.
We proceed by proving that $\Eff\in\Lspace^2_{\tloc}(\geoext)$, $\curlF\Eff\in\Lspace^2_{\tloc}(\geoext)$ and that $\Eff$ satisfies \eqref{eq:MWextPDE}, \eqref{eq:MWextBC}, \eqref{eq:MWextIC}, \eqref{eq:MWextRC}.\\
\emph{Step 1$\colon$ $\Eff\in\Lspace^2_{\tloc}(\geoext)$.}\quad
Since $\cf\in L^2\big((0,z_1);\Xspace^1\big)$ it follows that $\Et\in\Lspace^2\big(\geocs\times(0,z_1)\big)$ for each $z_1>0$.
In addition, because $(\Aopu-\ome^2)^{-1}\Aopu\in\BLO(\Xspaceu^0)$ it follows that $\En\in L^2\big(\geocs\times(0,z_1)\big)$ for each $z_1>0$ due to $\cf\in H^1\big((0,z_1);\Xspaceu^0\big)$.\\
\emph{Step 2$\colon$ $\curlF\Eff\in\Lspace^2_{\tloc}(\geoext)$.}\quad
Note that $\curlF\Eff=\bpm \Curlg\En+\partial_z\rotm\Et \\ \curlg\Et \epm$ (compare with \eqref{eq:curlofTPfunction}).
Since $\curlg\Et=\curlg\bv$, $\curlg\in\BLO(\Vspace^1,\Lspace^2(\geocs))$ and \Cref{lem:solution_semigroup} it follows that $\curlg\Et\in L^2_{\tloc}(\geoext)$.
Due to $\bv\in H^1((0,z_1);\Vspace^0)$ for each $z_1>0$ we have that $\partial_z\rotm\bv=\rotm\partial_z\bv\in (L^2_{\tloc}(\geoext))^2$.\\
\emph{Step 3$\colon$ $\Eff$ satisfies \eqref{eq:MWextPDE} and \eqref{eq:MWextBC}.}\quad
Recall that \eqref{eq:MWextPDE}-\eqref{eq:MWextBC} read
\begin{subequations}\label{eq:MaxwellPDEproof}
\begin{align}
\label{eq:MaxwellPDEproof-Et}
\Curlg\mucs^{-1}\curlg\Et-\omega^2\epscs\Et-\partial_{zz}\mucs^{-1}\Et+\partial_z\mucs^{-1}\gradg\En &= 0 \quad\text{in }\geoext,\\
\label{eq:MaxwellPDEproof-En}
-\divg\mucs^{-1}\gradg\En-\omega^2\epscs\En+\partial_z\divg(\mucs^{-1}\Et) &= 0 \quad\text{in }\geoext,\\
\label{eq:MaxwellPDEproof-EtBc}
\tv\cdot\Et &=0 \quad\text{on }\partial\geocs\times\Rp,\\
\label{eq:MaxwellPDEproof-EnBc}
\En &=0 \quad\text{on }\partial\geocs\times\Rp.
\end{align}
\end{subequations}
$\En$ solves \eqref{eq:MaxwellPDEproof-En}, \eqref{eq:MaxwellPDEproof-EnBc} due to its definition \eqref{eq:def-En-sol}.
Plugging \eqref{eq:MaxwellPDEproof-En}, \eqref{eq:MaxwellPDEproof-EnBc} into \eqref{eq:MaxwellPDEproof-Et}, \eqref{eq:MaxwellPDEproof-EtBc} and expressing $\Et=\bv+\frac{1}{\ome}\gradg w$ ``$\Eff$ solves \eqref{eq:MaxwellPDEproof}'' is equivalent to
\begin{align*}
\bpm 1\\&\big(-\frac{1}{\ome\eps}(1-\ome^2\Aopu_W^{-1})\eps\big)^{-1} \epm
\bigg( \Qopu-\partial_{zz} \bigg) \bpm \bv\\w \epm = 0,
\end{align*}
i.e., $(\Qopu-\partial_{zz}) \cf=0$, $\cf=\bpm \bv\\w \epm$ for each $z>0$.
Now, we proceed in a classical fashion:
First we consider the integral part of \eqref{eq:solution_semigroup}.
\Cref{eq:solution_semigroup} allows to exchange (differential) operators and integrals.
We compute that
\begin{align*}
(\Qopu-\partial_{zz}) e^{\lam z}(\lam+\Qopu^{1/2})^{-1}
=(\Qopu-\lam^2) e^{\lam z}(\lam+\Qopu^{1/2})^{-1}
&=(\Qopu^{1/2}-\lam)(\Qopu^{1/2}+\lam) e^{\lam z}(\lam+\Qopu^{1/2})^{-1}\\
&=(\Qopu^{1/2}-\lam)e^{\lam z}
\end{align*}
and note that $\int_{-\Sector_\theta} (\Qopu^{1/2}-\lam)e^{\lam z}\,\dd\lam=0$ due to the Cauchy integral theorem.
The finite sum in \eqref{eq:solution_semigroup} is handled analogously.
Thus the claim is proven.\\
\emph{Step 4$\colon$ $\Eff$ satisfies \eqref{eq:MWextIC}.}\quad
For each $\del,z_1>0$ the shift operator $(\tau_\del\Eff)((x,y,z)^\top):=\Eff((x,y,z+\del)^\top)$ is continuous on $\Hspace_{\partial\geocs\times(0,z_1)}(\curlF;\geocs\times(0,z_1))$.
The partial tangential component trace operator $\tr_{T,\geobdtrans}\colon\Hspace_{\partial\geocs\times(0,z_1)}(\curlF;\geocs\times(0,z_1)) \to \Hspace_0^{-1/2}(\curlg;\geobdtrans)$ is continuous too (see \Cref{sec:ttrspaces}).
Thus it remains to show that $\lim_{\del\to0+}\|\bfE_{T,0}-\tr_{T,\geobdtrans}\tau_\del\Eff\|_{\Hspace_0^{-1/2}(\curlg;\geobdtrans)}=0$.
We recall that $\bfE_{T,0}=\bv_0+\frac{1}{\ome}\gradg w_0$, $\Et=\bv+\frac{1}{\ome}\gradg w$ and that the decomposition $\Hspace_0^{-1/2}(\curlg;\geocs)=\Vspace^{1/2}+\gradg \Wspace^{1/2}$ is stable (see \Cref{rem:ass-top-flat,lem:tracespace-decomp-orth,lem:interpolation-calH}).
Hence the claim follows from \Cref{item:IVplus} of \Cref{lem:solution_semigroup} (for which we recall that $\Xspace^{1/2}=\Vspace^{1/2}\times\Wspace^{1/2}$).\\
\emph{Step 5$\colon$ $\Eff$ satisfies \eqref{eq:MWextRC}.}\quad
This readily follows from \eqref{eq:solution_hcurlloc} and \Cref{item:RC} of \Cref{lem:solution_semigroup}.
\end{proof}
We conclude the existence and boundedness of a \Calderon map in a rather abstract form.
\begin{corollary}\label{cor:calop-boundedness}
For $\bfE_{T,0}\in\Hspace_0^{-1/2}(\curl_{\geobdtrans};\geobdtrans)$ let $\Eff$ be the solution to \Cref{prob:MWext} provided by \Cref{thm:sol-ext}.
Then the \Calderon operator defined by $\calop\Ef_{T,0}:=-\nvf_{\geoext}\times\mucs^{-1}\curlF\Eff$ is bounded: $\calop\in\BLO(\Hspace_0^{-1/2}(\curl_{\geobdtrans};\geobdtrans), \Hspace_0^{-1/2}(\div_{\geobdtrans};\geobdtrans))$.
In addition, \Cref{prob:MWradiation} and \Cref{prob:MWint} are equivalent.
\end{corollary}
\begin{proof}
The claim follows from \Cref{thm:sol-ext}.
\end{proof}

\subsection{Analysis of the \Calderon operator}\label{subsec:calderon}

In order to analyze the well-posedness of \Cref{prob:MWintweak} we need in addition to its boundedness additional properties of $\calop$.
To this end we work with the more explicit representation of $\calop$ in terms of Dunford-Riesz integrals.
In preparation, let
\begin{align}\label{eq:Tscbd-def}
\Tscbd\Et:=\bv-\frac{1}{\overline{\ome}}\gradg w
\qquad\text{for }\Et=\bv+\frac{1}{\ome}\gradg w, \quad \bv\in\Vspace^{1/2}, w\in\Wspace^{1/2},
\end{align}
for which we have that $\Tscbd, \Tscbd^{-1}\in\BLO(\Hspace_0^{-1/2}(\curl_{\geobdtrans};\geobdtrans))$.
\begin{theorem}\label{thm:calop-decomp}
Let the assumptions of \Cref{subsec:geometry,subsec:parameters}, \Cref{ass:cut-off} and \Cref{ass:groupvel} be satisfied.
Then $\calop=\calop_\mathrm{sa}+\calop_\mathrm{co}$, $\calop_\mathrm{sa},\calop_\mathrm{co}\in\BLO(\Hspace_0^{-1/2}(\curl_{\geobdtrans};\geobdtrans), \Hspace_0^{-1/2}(\div_{\geobdtrans};\geobdtrans))$ where $\calop_\mathrm{co}$ is compact and
\begin{align*}
\spl \calop_\mathrm{sa}\bfE_{T,0},\Tscbd\bfE_{T,0} \spr_{\geobdtrans,-1/2,0}>0
\end{align*}
for each $\bfE_{T,0}\in\Hspace_0^{-1/2}(\curl_{\geobdtrans};\geobdtrans)\setminus\{0\}$ (and $\spl \plh,\plh \spr_{\geobdtrans,-1/2,0}$, $\Tscbd$ are defined in \Cref{def:dualitypairing-partial} and \eqref{eq:Tscbd-def} respectively).
\end{theorem}
\begin{proof}
To start with we recall the notation $\cf_0=(\bv_0,w_0)^\top$, $\cf_0\tf=(\bv_0\tf,w_0\tf)^\top$.
Let $\Eff=(\Et,\En)^\top$, $\Et=\bv+\frac{1}{\ome}\gradg w$, $\cf=(\bv,w)^\top$ be as in \Cref{cor:calop-boundedness}.
Since $\curlF\Eff=\bpm\Curlg\En+\rotm\partial_z\Et\\\curlg\Et\epm$ we have (identifying the tangential space on $\geobdtrans$ with $\bbC^2$) that
\begin{align*}
\calop\bfE_{T,0}=-\mucs^{-1}(\partial_z\Et-\gradg\En)=-\mucs^{-1}(\bv'+\gradg(\frac{1}{\ome}w'-\En)).
\end{align*}
We further compute that
\begin{align*}
\calop\bfE_{T,0}=
\calop_{\mathrm{co},1}\bfE_{T,0}
&-\int_{-\partial\Sector_\theta} \frac{\lam}{\mucs} P_{\Vspace} (\lam+\Qopu^{1/2})^{-1} \cf_0 \,\dd\lam\\
&+\int_{-\partial\Sector_\theta} \frac{1}{\mucs} \gradg \frac{\lam}{\ome} (\Aopu_W-\ome^2)^{-1}\ome^2 P_W (\lam+\Qopu^{1/2})^{-1} \cf_0 \,\dd\lam
\end{align*}
where $\calop_{\mathrm{co},1}\bfE_{T,0}:=
-\sum_{n=1}^{\tNpr} \frac{i\tilde\beta_n}{\mucs} P_{\Vspace} P_n \cf_0
+\sum_{n=1}^{\tNpr} \frac{1}{\mucs} \gradg \frac{i\tilde\beta_n}{\ome} (\Aopu_W-\ome^2)^{-1}\ome^2 P_W P_n \cf_0$.
Since its range is finite dimensional the operator $\calop_{\mathrm{co},1}$ is compact.
Exploiting the divergence-free of constraint of $\bv_0,\bv_0\tf$ and integration by parts we continue to compute that
\begin{align*}
\big\spl \calop\bfE_{T,0}-\calop_{\mathrm{co},1} \bfE_{T,0},& \bv_0\tf+\frac{1}{\ome}\gradg w\tf \big\spr_{\Lspace^2(\geocs)}
=\bigg\spl \frac{-1}{\mucs} \int_{-\partial\Sector_\theta} \lam P_{\Vspace} (\lam+\Qopu^{1/2})^{-1} \cf_0 \,\dd\lam,\bv_0\tf \bigg\spr_{\Lspace^2(\geocs)}\\
&+\bigg\spl \epscs \int_{-\partial\Sector_\theta} \frac{\lam}{\ome} \Aopu_W(\Aopu_W-\ome^2)^{-1} \ome^2 P_{\Wspace} (\lam+\Qopu^{1/2})^{-1} \cf_0 \,\dd\lam, \frac{1}{\ome} w_0\tf \bigg\spr_{L^2(\geocs)},
\end{align*}
where for the moment we use $\Lspace^2(\geocs)$, $L^2(\geocs)$ as pivot spaces or assume $\bv_0, w_0$ to be smooth enough.
As in the proof of \Cref{lem:solution_semigroup}.\ref{item:Sobolev}) we have that $(\lam+\Qopu^{1/2})^{-1}=(\lam+\Aopu^{1/2})^{-1}+(\lam+\Qopu^{1/2})^{-1} Z (\lam+\Aopu^{1/2})^{-1}$ with $Z\in\BLO(\Xspace^0)$.
Thus
\begin{align*}
\int_{-\partial\Sector_\theta} \lam P_{\Vspace} (\lam+\Qopu^{1/2})^{-1} \cf_0 \,\dd\lam
&=\int_{-\partial\Sector_\theta} \lam P_{\Vspace} (\lam+\Aopu^{1/2})^{-1} \cf_0 \,\dd\lam+\tilde\calop_{\mathrm{co},2} \cf_0 
\end{align*}
where $\tilde \calop_{\mathrm{co},2} \cf_0 := \int_{-\partial\Sector_\theta} \lam P_{\Vspace} (\lam+\Qopu^{1/2})^{-1} Z (\lam+\Aopu^{1/2})^{-1}\cf_0 \,\dd\lam$.
Exploiting resolvent estimates for $\Qopu^{1/2}$ and $\Aopu^{1/2}$ we obtain that $\tilde\calop_{\mathrm{co},2}\in\BLO(\Xspace^{1/2},\Vspace^0)$.
Thus the operator $\calop_\mathrm{co,2}\in\BLO\big(\Hspace_0^{-1/2}(\curl_{\geobdtrans};\geobdtrans),$ $\Hspace_0^{-1/2}(\div_{\geobdtrans};\geobdtrans)\big)$ induced by the form $\bigg\spl \frac{-1}{\mucs} \int_{-\partial\Sector_\theta} \lam P_{\Vspace} (\lam+\Qopu^{1/2})^{-1} Z (\lam+\Aopu^{1/2})^{-1/2}\cf_0 \,\dd\lam,\bv_0\tf \bigg\spr_{\Lspace^2(\geocs)}$ is compact.
In a similar fashion we obtain that the operators $\calop_\mathrm{co,3}, \calop_\mathrm{co,4}\in\BLO\big(\Hspace_0^{-1/2}(\curl_{\geobdtrans};\geobdtrans),$ $\Hspace_0^{-1/2}(\div_{\geobdtrans};\geobdtrans)\big)$ induced by
\begin{align*}
\spl \calop_{\mathrm{co},3}\bfE_{T,0},\bfE_{T,0}\tf &\spr_{\geobdtrans,-1/2,0}:=\\
&\bigg\spl \epscs \int_{-\partial\Sector_\theta} \frac{\lam}{\ome} \Aopu_W(\Aopu_W-\ome^2)^{-1} \ome^2 P_{\Wspace} (\lam+\Qopu^{1/2})^{-1} Z (\lam+\Aopu^{1/2})^{-1} \cf_0 \,\dd\lam, \frac{1}{\ome} w_0\tf \bigg\spr_{L^2(\geocs)}
\end{align*}
and
\begin{align*}
\spl \calop_{\mathrm{co},4}\bfE_{T,0},\bfE_{T,0}\tf \spr_{\geobdtrans,-1/2,0}:=
\bigg\spl \epscs \int_{-\partial\Sector_\theta} \lam \ome(\Aopu_W-\ome^2)^{-1} \ome^2 P_{\Wspace} (\lam+\Aopu^{1/2})^{-1} \cf_0 \,\dd\lam, \frac{1}{\ome} w_0\tf \bigg\spr_{L^2(\geocs)}
\end{align*}
are compact.
The claim follows now with $\calop_{\mathrm{co}}:=\sum_{n=1}^4 \calop_{\mathrm{co},n}$ and
\begin{align*}
\spl \calop_\mathrm{sa}\bfE_{T,0},\Tscbd\bfE_{T,0} \spr_{\geobdtrans,-1/2,0}
&=\bigg\spl \frac{-1}{\mucs} \int_{-\partial\Sector_\theta} \lam (\lam+\Aopu_V^{1/2})^{-1}\,\dd\lam \, \bv_0, \bv_0 \bigg\spr\\
&-\bigg\spl \epscs \int_{-\partial\Sector_\theta} \lam (\lam+\Aopu_W^{1/2})^{-1}\,\dd\lam \, w_0, w_0 \bigg\spr,
\end{align*}
where we have exploited that $\Aopu^{1/2}$ is block diagonal yielding $P_{\Vspace} (\lam+\Aopu^{1/2})^{-1} \cf_0=(\lam+\Aopu_{\Vspace}^{1/2})^{-1} \bv_0$, $P_{\Wspace} (\lam+\Aopu^{1/2})^{-1} \cf_0=(\lam+\Aopu_{\Wspace}^{1/2})^{-1} w_0$.
The positivity of $\spl \calop_\mathrm{sa}\bfE_{T,0},\Tscbd\bfE_{T,0} \spr_{\geobdtrans,-1/2,0}$ follows from the selfadjointness and positivity of $\Aopu^{1/2}_{\Vspace}$ and $\Aopu^{1/2}_{\Wspace}$.
\end{proof}

\section{Stability analysis of the interior problem}\label{sec:interior}

Let
\begin{align*}
\Vint:=\{\Eint\in \Hspace_{\bci}(\curlF;\geoint)\colon \,\,\, \divF\Eint=0 \text{ in }\geoint, \quad \tr_{T,\geobdtrans}\Eint\in\Vspace^{1/2} \}.
\end{align*}
Then, due to \Cref{lem:tracespace-decomp-orth-Omega} we have the stable decomposition
\begin{align*}
\Hspace_{\bci}(\curlF;\geoint)=\Vint\oplus\gradF H^1_{\bci}(\geoint),
\end{align*}
and henceforth use the notation $\Eint=\vint+\frac{1}{\ome}\gradF\wint$, $\vint\in\Vint$, $\wint\in H^1_{\bci}(\geoint)$ and like-wise $\Eint\tf=\vint\tf+\frac{1}{\ome}\gradF\wint\tf$ for $\Eint,\Eint\tf\in \Hspace_{\bci}(\curlF;\geoint)$.
Let
\begin{align*}
\Tscvol\Eint:=\vint-\frac{1}{\overline{\ome}}\gradF\wint,
\end{align*}
for which we have $\Tscvol, \Tscvol^{-1} \in \BLO(\Hspace_{\bci}(\curlF;\geoint))$.
We note the crucial property
\begin{align*}
\tr_{T,\geobdtrans}\Tscvol=\Tscbd\tr_{T,\geobdtrans}.
\end{align*}
We consider the $\Hspace_{\bci}(\curlF;\geoint)$-bounded sesquilinear form $\asesint(\plh,\plh)$ appearing in \Cref{prob:MWint}:
\begin{align*}
\asesint(\Eint,\Eint\tf):=\spl \muinv\curlF\Eint,\curlF\Eint\tf \spr_{\geoint}
-\ome^2\spl \eps\Eint,\Eint\tf \spr_{\geoint}
+\spl \calop \Eint_T,\Eint_T\tf \spr_{\geobdtrans},\qquad
\Eint,\Eint\tf\in\Hspace_{\bci}(\curlF;\geoint).
\end{align*}
Let $\Aopint\in\BLO(\Hspace_{\bci}(\curlF;\geoint))$ be the associated Riesz representation, i.e.,
\begin{align*}
\spl \Aopint\Eint,\Eint\tf \spr_{\Hspace_{\bci}(\curlF;\geoint)}=\asesint(\Eint,\Eint\tf) \qquad\forall\Eint,\Eint\tf\in\Hspace_{\bci}(\curlF;\geoint).
\end{align*}
We obtain one of our main results.
\begin{theorem}\label{thm:wTc}
Let the assumptions of \Cref{subsec:geometry,subsec:parameters}, \Cref{ass:cut-off} and \Cref{ass:groupvel} be satisfied.
Then, $\Aopint$ is weakly $\Tscvol$-coercivity, i.e., there exists a compact $K\in\BLO(\Hspace_{\bci}(\curlF;\geoint))$ such that $\Tscvol^*\Aopint+K$ is coercive.
In particular, $\Aopint$ is Fredholm with index zero.
\end{theorem}
\begin{proof}
We compute that
\begin{align*}
\asesint(\Eint,\Tscvol\Eint\tf)=a_1(\Eint,\Tscvol\Eint\tf)+a_2(\Eint,\Tscvol\Eint\tf)
\end{align*}
where
\begin{align*}
a_1(\Eint,\Tscvol\Eint\tf)&:=
\spl \muinv\curlF\vint,\curlF\vint\tf \spr_{\geoint}
+\spl \eps\gradF\wint,\gradF\wint\tf \spr_{\geoint}
+\spl \calop_\mathrm{sa} \Eint_T,\Tscbd\Eint_T\tf \spr_{\geobdtrans},\\
a_2(\Eint,\Tscvol\Eint\tf)&:=
-\ome^2\spl \eps\vint,\vint\tf \spr_{\geoint}
-\ome\spl \eps\gradF\wint,\vint\tf \spr_{\geoint}
+\ome\spl \eps\vint,\gradF\wint\tf \spr_{\geoint}
+\spl \calop_\mathrm{co} \Eint_T,\Tscbd\Eint_T\tf \spr_{\geobdtrans}.
\end{align*}
The operator corresponding to form $a_1(\plh,\plh)$ is coercive.
The operator corresponding to form $a_2(\plh,\plh)$ is compact due to \Cref{thm:calop-decomp} and \Cref{lem:HcurlDecompCompact}.
\end{proof}
\begin{proposition}\label{prop:resolventset}
Let the assumptions of \Cref{subsec:geometry,subsec:parameters} be satisfied.
Then, the operator $\Aopint$ is coercive for all $\ome\in i\Rp$.
\end{proposition}
\begin{proof}
We note that for $\ome\in i\Rp$ after multiplication of \eqref{eq:modal-Etw-operator} with $\diag(I_{\Vspace^1},-I_{\Wspace^1})$ the operator on the left is selfadjoint and the operator on the right is selfadjoint and positive.
Hence, the appropriately scaled eigenfunctions form Riesz bases in $\Xspace^s$, $s=0,1/2,1$ and all outgoing wavenumbers $\beta_n$ are semi simple and satisfy $\beta_n\in i\Rp$.
It can then be computed that $\calop$ is positive and the claim can be deduced.
\end{proof}
\begin{remark}\label{rem:wellposedness}
Now, by means of the Fredholm theorem for holomorphic operator functions we could almost deduce from \Cref{thm:wTc} and \Cref{prop:resolventset} the existence of a discrete set $D\subset\bbR$ such that the operator $\Aopint$ is bijective for all $\ome\in\bbR\setminus D$.
However, in our analysis we excluded by means of \Cref{ass:groupvel} the existence of nontrivial Jordan chains of guided modes.
Since we have no result on the discreteness of such frequencies, we cannot conclude the desired claim.
Nonetheless, we note that \Cref{ass:groupvel} was only introduced to reduce the technical effort of our analysis and we do not expect that the handling of nontrivial Jordan chains poses an impregnable challenge.
\end{remark}

\appendix

\section{Tangential trace spaces}\label{sec:ttrspaces}

In order to deal with a transparent boundary we have to discuss the matching of tangential traces and hence need to discuss the respective trace spaces.
However, when building tangential traces from different sides the definitions of the respective tangential trace spaces a priori do not coincide.
Hence, we need to establish that no obscurities arise.

In particular, in this section we discuss the tangential trace space of $\Hspace_{\Gamma_-}(\curlF;\Omega)$, i.e., the subspace of $\Hspace(\curlF;\Omega)$ with a homogeneous trace on the part $\Gamma_-$ of the boundary $\Gamma$, for domains $\Omega$ having only Lipschitz regularity.
For boundary conditions on the entire boundary the trace space is rigorously analyzed in \cite{Tartar97,BuffaCostabelSheen02}.
For Lipschitz polyhedra and partial boundary conditions the analysis is conducted in \cite{BuffaCiarlet01a,BuffaCiarlet01b}.
Here, we close this gap and discuss that no obstacles occur when generalizing the analysis of \cite{BuffaCiarlet01a,BuffaCiarlet01b}.
A detail is that we have to work with a non-standard Sobolev space $\Hvar^1(\partial\Gamma_-)$.
At the end we include some complementary technical results, which are applied in the manuscript.\\

Note, that for convenience we denote in this section our geometric objects as $\Omega$, $\Gamma$ and $\Sigma$ which do not coincide with the notation introduced in \Cref{sec:setting}, i.e., $\geoint\sim\Omega$, $\partial\geoint\sim\Gamma$, $\geobddir\sim\Gamma_-$, $\geobdtrans=\geocs\times\{0\}\sim\Gamma_+$.

\subsection{Sobolev spaces}

We consider all vector spaces over the scalar body $\bbC$.
For Lipschitz domains $\Omega\subset\bbR^d$, $d=1,2,3$ and Lipschitz manifolds $\Gamma_0\subset\Gamma:=\partial \Omega$ let $L^2(\Omega)$, $H^s(\Omega)$, $H^s_0(\Omega)$, $H^s_{00}(\Omega)$, $s\in\bbR$ and $L^2(\Gamma_0)$, $H^t(\Gamma_0)$, $H^t_0(\Gamma_0)$, $H^t_{00}(\Gamma_0)$, $t\in[-1,1]$ be the convenient Lebesgue and Sobolev spaces, see, e.g., \cite{McLean00,AgranovichBook}.
For a space $X$ it's dual space is denoted as $X'$ and we recall that $H^s(\Omega)'=H^{-s}_{00}(\Omega)$, $H_{00}^s(\Omega)'=H^{-s}(\Omega)$, $H^t(\Gamma_0)'=H^{-t}_{00}(\Gamma_0)$, $H_{00}^t(\Gamma_0)'=H^{-t}(\Gamma_0)$ with pivot spaces $L^2(\Omega)$ and $L^2(\Gamma_0)$ respectively.
For Lipschitz domains $\Omega\subset\bbR^3$ and Lipschitz manifolds $\Gamma_0\subset\Gamma$ let $\Lspace^2(\Omega):=(L^2(\Omega))^3$, $\Hspace^s(\Omega):=(H^s(\Omega))^3$, $\Hspace^s(\Omega):=(H^s_0(\Omega))^3$, $\Hspace^s(\Omega):=(H^s_{00}(\Omega))^3$, $s\in\bbR$ and $\Lspace^2(\Gamma_0):=(L^2(\Gamma_0))^3$, $\Hspace^t(\Gamma_0):=(H^t(\Gamma_0))^3$, $\Hspace^t_0(\Gamma_0):=(H^t_0(\Gamma_0))^3$, $\Hspace^t_{00}(\Gamma_0):=(H^t_{00}(\Gamma_0))^3$, $t\in[-1,1]$.
For $X\subset L^2$ let $X_*:=\{u\in X\colon \spl u,1\spr_X=0\}$ and $X'_*:=\{f\in X'\colon \spl f,1 \spr_{X'\times X}=0\}$.
We denote the convenient differential operators in 3D as follows
\[
\gradF u:=(\partial_x u, \partial_y u, \partial_z u)^\top,\qquad
\curlF \bfu:=\gradF\times\bfu,\qquad
\divF\bfu :=\gradF\cdot\bfu.
\]
We abbreviate scalar and vectorial hermitian $L^2$ scalar products and hermitian duality pairings with (possibly tangential) $L^2$ pivot space over a Lipschitz domain $D$ or a Lipschitz surface $S$ as $\spl\plh,\plh\spr_D$ and $\spl\plh,\plh\spr_S$ respectively.
Henceforth, if not otherwise specified, let the norm of a Banach space $X$ be implicitly defined as the graph norm or lifting norm respectively induced by the definition of $X$, e.g.,
\begin{align*}
\Hspace(\curlF;\Omega)&:=\{\bfu\in\Lspace^2(\Omega)\colon\curlF\bfu\in\Lspace^2(\Omega)\},\qquad
\|\bfu\|_{\Hspace(\curlF;\Omega)}^2:=\|\bfu\|_{\Lspace^2(\Omega)}^2+\|\curlF\bfu\|_{\Lspace^2(\Omega)}^2,
\end{align*}
or
\begin{align*}
\Hvar^{3/2}(\Gamma):=\tr H^2(\Omega),\qquad
\|u\|_{\small\Hvar^{3/2}(\Gamma)}:=\inf \{\|v\|_{H^2(\Omega)}\colon v\in H^2(\Omega) \text{ and } \tr v=u\}.
\end{align*}
Further, let $B_r:=\{\xf\in\bbR^3\colon |\xf|<r\}$, $r>0$ and
\begin{align*}
\Hspace_{\tloc}(\curlF;\Omega)&:=\{\bfu\colon \Omega\to\bbC^3 \text{ such that } \bfu|_{\Omega\cap B_r}\in \Hspace(\curlF;\Omega\cap B_r) \text{ for all } r>0\},\\
\Lspace^2_{\tloc}(\Omega)&:=\{\bfu\colon \Omega\to\bbC^3 \text{ such that } \bfu|_{\Omega\cap B_r}\in \Lspace^2(\Omega\cap B_r) \text{ for all } r>0\}.
\end{align*}

\subsection{Tangential traces of $\Hspace(\curlF;\Omega)$}

In the remainder let $\Omega\subset\bbR^3$ be a bounded Lipschitz domain with outward unit normal vector $\nvf$ and $\Gamma_-$, $\Gamma_+$ be a Lipschitz dissection of it's boundary $\partial\Omega=\Gamma=\Gamma_+\dot\cup\Gamma_-\dot\cup \Sigma$, $\Sigma:=\ol{\Gamma_+}\cap\overline{\Gamma_-}$.
We recall that the scalar trace operator $\tr$, the tangential component trace operator $\trtan$ and it's rotated version $\trtanR$ defined as
\begin{align*}
\tr u:=u|_{\Gamma},\qquad \trtan\bfu:=\bfu_T:=\nvf\times\bfu|_{\Gamma}\times\nvf,\qquad \trtanR\bfu:=-\nvf\times\bfu|_{\Gamma}.
\end{align*}
For $S\in\{\Gamma,\Gamma_+,\Gamma_-\}$ let
\begin{align*}
\Lspace^2_T(S)&:=\{\bfu\in\Lspace^2(S)\colon \nvf\cdot\bfu=0\},\\
\Hspace^{1/2}_{T}(S)&:=\{\nvf\times\bfu\times\nvf\colon\bfu\in\Hspace^{1/2}(S)\},
&\Hspace^{1/2}_{T,00}(S)&:=\{\nvf\times\bfu\times\nvf\colon\bfu\in\Hspace^{1/2}_{00}(S)\},\\
\Hspace^{1/2}_{RT}(S)&:=\{-\nvf\times\bfu\colon\bfu\in\Hspace^{1/2}(S)\},
&\Hspace^{1/2}_{RT,00}(S)&:=\{-\nvf\times\bfu\colon\bfu\in\Hspace^{1/2}_{00}(S)\},
\end{align*}
where we adopt the notion to interpret tangential vectors as two-dimensional objects \cite{BuffaCostabelSheen02}.
Note that $\Hspace^{1/2}_{T}(\Gamma)=V_\pi$, $\Hspace^{1/2}_{RT}(\Gamma)=V_\gam$ with $V_\pi, V_\gam$ as in \cite{BuffaCostabelSheen02}.
Consider the surface curl and surface gradient operators $\nabla_\Gamma\colon H^1(\Gamma)\to\Lspace^2_T(\Gamma)$ and $\Curl_\Gamma\colon H^1(\Gamma)\to\Lspace^2_T(\Gamma)$.
On $H^2(\Omega)$ we have that $\tr_T\gradF u=\nabla_{\Gamma} \tr u$ and $\tr_{RT}\gradF u=\Curl_{\Gamma} \tr u$.
On the non-standard space%
\footnote{Note the non-standard symbol $\Hvar$ to distinguish it from the convenient Sobolev space $H^{3/2}(\Gamma)$ (defined, e.g., via the Sobolev–Slobodeckij norm).}
$\Hvar^{3/2}(\Gamma)$ we have that
\begin{align*}
\Curl_{\Gamma} \in \BLO( \Hvar^{3/2}(\Gamma), \Hspace^{1/2}_{RT}(\Gamma)),\qquad
-\nabla_{\Gamma} \in \BLO( \Hvar^{3/2}(\Gamma), \Hspace^{1/2}_{T}(\Gamma)).
\end{align*}
Subsequently, let
\begin{align*}
\curl_{\Gamma} \in \BLO(\Hspace^{1/2}_{RT}(\Gamma)', \Hvar^{3/2}(\Gamma)'),\qquad
\div_{\Gamma} \in \BLO(\Hspace^{1/2}_{T}(\Gamma)', \Hvar^{3/2}(\Gamma)'),
\end{align*}
be the respective dual operators.
Thence, with
\begin{align*}
\Hspace^{-1/2}(\curl_{\Gamma};\Gamma)&:=\{\bflam\in \Hspace^{1/2}_{RT}(\Gamma)'\colon  \curl_{\Gamma}\bflam\in H^{-1/2}(\Gamma)\},\\
\Hspace^{-1/2}(\div_{\Gamma};\Gamma)&:=\{\bflam\in \Hspace^{1/2}_{T}(\Gamma)'\colon  \div_{\Gamma}\bflam\in H^{-1/2}(\Gamma)\},
\end{align*}
as in \cite[Thm.~4.1]{BuffaCostabelSheen02} it holds that
\begin{align*}
\trtan\in\BLO(\Hspace(\curlF;\Omega),\Hspace^{-1/2}(\curl_{\Gamma};\Gamma)),\qquad
\trtanR\in\BLO(\Hspace(\curlF;\Omega),\Hspace^{-1/2}(\div_{\Gamma};\Gamma)).
\end{align*}
%

\subsection{Partial tangential traces of $\Hspace(\curlF;\Omega)$}

Let
\begin{align*}
\Hspace^{-1/2}(\curl_{\Gamma_\pm};\Gamma_\pm)
&:=\{\bflam\in \Hspace^{1/2}_{RT,00}(\Gamma_\pm)'\colon \curl_{\Gamma_\pm}\bflam \in H^{1/2}_{00}(\Gamma_\pm)'\},\\
\Hspace^{-1/2}(\div_{\Gamma_\pm};\Gamma_\pm)
&:=\{\bflam\in \Hspace^{1/2}_{T,00}(\Gamma_\pm)'\colon \div_{\Gamma_\pm}\bflam \in H^{1/2}_{00}(\Gamma_\pm)'\},
\end{align*}
where
\begin{align*}
\curl_{\Gamma_\pm} \in \BLO(\Hspace^{1/2}_{RT,00}(\Gamma_\pm)', \Hvar^{3/2}_{00}(\Gamma_\pm)'),\qquad
\div_{\Gamma_\pm} \in \BLO(\Hspace^{1/2}_{T,00}(\Gamma_\pm)', \Hvar^{3/2}_{00}(\Gamma_\pm)'),
\end{align*}
are defined as the dual operators of 
\begin{align*}
\Curl_{\Gamma_\pm} \in \BLO( \Hvar^{3/2}_{00}(\Gamma_\pm), \Hspace^{1/2}_{RT,00}(\Gamma_\pm)),\qquad
-\nabla_{\Gamma_\pm} \in \BLO( \Hvar^{3/2}_{00}(\Gamma_\pm), \Hspace^{1/2}_{T,00}(\Gamma_\pm)).
\end{align*}
Then, in a natural way acting on $\Hspace(\curlF;\Omega)$ the restrictions $\tr_{T,\Gamma_\pm}$, $\tr_{RT,\Gamma_\pm}$ of $\trtan$, $\trtanR$ to $\Gamma_\pm$ are well-defined and bounded as operators into the spaces $\Hspace^{-1/2}(\curl_{\Gamma_\pm};\Gamma_\pm)$ and $\Hspace^{-1/2}(\div_{\Gamma_\pm};\Gamma_\pm)$, i.e.,
\begin{align*}
\tr_{T,\Gamma_\pm} \in \BLO(\Hspace(\curlF;\Omega), \Hspace^{-1/2}(\curl_{\Gamma_\pm};\Gamma_\pm)),\qquad
\tr_{RT,\Gamma_\pm} \in \BLO(\Hspace(\curlF;\Omega), \Hspace^{-1/2}(\div_{\Gamma_\pm};\Gamma_\pm)).
\end{align*}

\subsection{Stronger partial tangential spaces}

Let
\begin{align*}
\Hspace^{-1/2}_{00}(\curl_{\Gamma_\pm};\Gamma_\pm)
&:=\{\bfu\in (\Hspace^{1/2}_{RT}(\Gamma_\pm))'\colon \curl_{\Gamma_\pm}\bfu \in (H^{1/2}(\Gamma_\pm))'\},\\
\Hspace^{-1/2}_{00}(\div_{\Gamma_\pm};\Gamma_\pm)
&:=\{\bfu\in (\Hspace^{1/2}_{T}(\Gamma_\pm))'\colon \div_{\Gamma_\pm}\bfu \in (H^{1/2}(\Gamma_\pm))'\},
\end{align*}
and note that
\begin{align*}
\Curl_{\Gamma_\pm} \in \BLO( \Hvar^{3/2}(\Gamma_\pm), \Hspace^{1/2}_{RT}(\Gamma_\pm)),\qquad
-\nabla_{\Gamma_\pm} \in \BLO( \Hvar^{3/2}(\Gamma_\pm), \Hspace^{1/2}_{T}(\Gamma_\pm)).
\end{align*}
For smooth $\Omega$ it holds that $\Hspace^{1/2}_{T}(\Gamma_\pm)=\Hspace^{1/2}_{RT}(\Gamma_\pm)$, the spaces $\Hspace^{-1/2}_{T}(\Gamma_\pm)$, $\Hspace^{-1/2}_{T,00}(\Gamma_\pm)$ have a meaningful definition, it holds that $\Hspace^{-1/2}_{T,00}(\Gamma_\pm)=\Hspace^{1/2}_{T}(\Gamma_\pm)'$, $\Hspace^{-1/2}_{T}(\Gamma_\pm)=\Hspace^{1/2}_{T,00}(\Gamma_\pm)'$ and thus
\begin{align*}
\Hspace^{-1/2}(\curl_{\Gamma_\pm};\Gamma_\pm)
&=\{\bflam\in \Hspace^{-1/2}_{T}(\Gamma_\pm)\colon \curl_{\Gamma_\pm}\bflam \in H^{-1/2}(\Gamma_\pm)\},\\
\Hspace^{-1/2}_{00}(\curl_{\Gamma_\pm};\Gamma_\pm)
&=\{\bflam\in \Hspace^{-1/2}_{T,00}(\Gamma_\pm)\colon \curl_{\Gamma_\pm}\bflam \in H^{-1/2}_{00}(\Gamma_\pm)\},\\
\Hspace^{-1/2}(\div_{\Gamma_\pm};\Gamma_\pm)
&=\{\bflam\in \Hspace^{-1/2}_{T}(\Gamma_\pm)\colon \div_{\Gamma_\pm}\bflam \in H^{-1/2}(\Gamma_\pm)\},\\
\Hspace^{-1/2}_{00}(\div_{\Gamma_\pm};\Gamma_\pm)
&=\{\bflam\in \Hspace^{-1/2}_{T,00}(\Gamma_\pm)\colon \div_{\Gamma_\pm}\bflam \in H^{-1/2}_{00}(\Gamma_\pm)\},
\end{align*}
which motivates the choice of indices in the foregoing definitions.
Also note that compared to the notation of \cite{BuffaCiarlet01a,BuffaCiarlet01b} the indices of these spaces are reversed%
\footnote{When comparing with \cite{BuffaCiarlet01a,BuffaCiarlet01b} note that therein $D$ is assumed to be a Lipschitz polyhedron and subsequently more specified characterizations of the spaces $\Hspace^{1/2}_{T}(\Gamma_\pm)$, $\Hspace^{1/2}_{RT}(\Gamma_\pm)$ are used.}.
Since $\Hvar^{3/2}(\Gamma_\pm)\subset H^s(\Gamma_\pm)$, $s\in(1/2,1)$ the trace operator $\tr_{\Sigma}\in \BLO(\Hvar^{3/2}(\Gamma_\pm), H^{s-1/2}(\Sigma))$ is well-defined and bounded.
Hence, let
\begin{align*}
\Hvar^1(\Sigma):=\tr_{\Sigma} \Hvar^{3/2}(\Gamma_\pm),
\end{align*}
$\nvt_\pm$ and $\tv_\pm$ be the unit normal and tangential vectors on $\Gamma_\pm$, where $\nvt_\pm$ is outward of $\Gamma_\pm$ and $\tv_\pm$ is such that $(\nvf,\nvt_\pm,\tv_\pm)$ forms a positively orientated coordinate system.
Note that $\nvt_+=-\nvt_-$ and $\tv_+=-\tv_-$.
Following the lines of the proof of \cite[Prop.~3.14]{BuffaCiarlet01a} the surface tangential trace $\tr_{tv_+}\bflam:=\tv_\pm\cdot\bflam|_{\Sigma}$ and the surface normal trace $\tr_{\nvt_\pm}\bflam:=\nvt_\pm\cdot\bflam|_{\Sigma}$ (defined for smooth $\bflam$)
admit continuous extensions
\begin{align*}
\tr_{\tv_\pm}\in\BLO(\Hspace^{-1/2}_{00}(\curl_{\Gamma_\pm};\Gamma_\pm), \Hvar^1(\Sigma)'),\qquad
\tr_{\nvt_\pm}\in\BLO(\Hspace^{-1/2}_{00}(\div_{\Gamma_\pm};\Gamma_\pm), \Hvar^1(\Sigma)'),
\end{align*}
which satisfy
\begin{align*}
\spl \bflam, \Curl_{\Gamma_\pm} \phi \spr_{\Hspace^{1/2}_{RT}(\Gamma_\pm)'\times \Hspace^{1/2}_{RT}(\Gamma_\pm)}
-\spl \curl_{\Gamma_\pm} \bflam, \phi \spr_{H^{1/2}(\Gamma_\pm)'\times H^{1/2}(\Gamma_\pm)}
&=-\spl \tr_{\tv_\pm}\bflam, \tr \phi \spr_{\small\Hvar^1(\Sigma)'\times\Hvar^1(\Sigma)},\\
\spl \bflam, \nabla_{\Gamma_\pm} \phi \spr_{\Hspace^{1/2}_T(\Gamma_\pm)'\times \Hspace^{1/2}_T(\Gamma_\pm)}
+\spl \div_{\Gamma_\pm} \bflam, \phi \spr_{H^{1/2}(\Gamma_\pm)'\times H^{1/2}(\Gamma_\pm)}
&=\spl \tr_{\nvt_\pm}\bflam, \tr \phi \spr_{\small\Hvar^1(\Sigma)'\times\Hvar^1(\Sigma)},
\end{align*}
for all $\phi\in\Hvar^{3/2}(\Gamma_+)$.
The following result holds.
\begin{lemma}\label{lem:tracespace-pw}
If $\bflam_{\pm}\in \Hspace^{-1/2}_{00}(\curl_{\Gamma_\pm};\Gamma_\pm)$ and $\tr_{\tv_+}\bflam_{+}+\tr_{\tv_-}\bflam_-=0$, then with $\bflam(\bfv):=\bflam_{+}(\bfv|_{\Gamma_+})+\bflam_-(\bfv|_{\Gamma_-})$, $\bfv\in \Hspace^{1/2}_{RT}(\Gamma)$ it holds that $\bflam\in\Hspace^{-1/2}(\curl_{\Gamma};\Gamma)$ and
\[
\|\bflam\|_{\Hspace^{-1/2}(\curl_{\Gamma};\Gamma)} \leq \sqrt{2}\big( \|\bflam_{+}\|_{\Hspace^{-1/2}_{00}(\curl_{\Gamma+};\Gamma+)}+\|\bflam_-\|_{\Hspace^{-1/2}_{00}(\curl_{\Gamma_-};\Gamma_-)} \big).
\]
Like-wise,
if $\bflam_{\pm}\in \Hspace^{-1/2}_{00}(\div_{\Gamma_\pm};\Gamma_\pm)$ and $\tr_{\nvt_+}\bflam_{+}+\tr_{\nvt_-}\bflam_-=0$, then with $\bflam(\bfv):=\bflam_{+}(\bfv|_{\Gamma_+})+\bflam_-(\bfv|_{\Gamma_-})$, $\bfv\in \Hspace^{1/2}_{RT}(\Gamma)$ it holds that $\bflam\in\Hspace^{-1/2}(\div_{\Gamma};\Gamma)$ and
\[
\|\bflam\|_{\Hspace^{-1/2}(\div_{\Gamma};\Gamma)} \leq \sqrt{2}\big( \|\bflam_{+}\|_{\Hspace^{-1/2}_{00}(\div_{\Gamma+};\Gamma+)}+\|\bflam_-\|_{\Hspace^{-1/2}_{00}(\div_{\Gamma_-};\Gamma_-)} \big).
\]
\end{lemma}
\begin{proof}
We only present the proof for the first case.
It follows from it's definition that $\bflam\in\Hspace^{1/2}_{RT}(\Gamma)'$ and $\|\bflam\|_{\Hspace^{1/2}_{RT}(\Gamma)'} \leq \|\bflam_{+}\|_{\Hspace^{1/2}_{RT}(\Gamma_+)'} + \|\bflam_{-}\|_{\Hspace^{1/2}_{RT}(\Gamma_-)'}$.
In addition, we have that
\begin{align}
\label{eq:pwcurl}
\begin{aligned}
\spl \curl_{\Gamma} \bflam,v\spr_{\small\Hvar^{3/2}(\Gamma)'\times\Hvar^{3/2}(\Gamma)}
&=\spl \bflam,\Curl_{\Gamma} v\spr_{\Hspace^{1/2}_{RT}(\Gamma)'\times\Hspace^{1/2}_{RT}(\Gamma)}\\
&\hspace{-25mm}=\spl \bflam_{+},\Curl_{\Gamma_+} v|_{\Gamma_+}\spr_{\Hspace^{1/2}_{RT}(\Gamma_+)'\times\Hspace^{1/2}_{RT}(\Gamma_+)}
+\spl \bflam_-,\Curl_{\Gamma_-} v|_{\Gamma_-}\spr_{\Hspace^{1/2}_{RT}(\Gamma_-)'\times\Hspace^{1/2}_{RT}(\Gamma_-)}\\
&\hspace{-25mm}=\spl \curl_{\Gamma_+} \bflam_{+},v|_{\Gamma_+}\spr_{\small\Hvar^{1/2}(\Gamma_+)'\times\Hvar^{1/2}(\Gamma_+)}
+\spl \curl_{\Gamma_-} \bflam_{-},v|_{\Gamma_-}\spr_{\small\Hvar^{3/2}(\Gamma_-)'\times\Hvar^{3/2}(\Gamma_-)}\\
&\hspace{-25mm}+\spl \tr_{\tv_+} \bflam_{+}, \tr_{\Sigma} v|_{\Gamma_+}\spr_{\small\Hvar^{1}(\Sigma)'\times\Hvar^{1}(\Sigma)}
+\spl \tr_{\tv_-} \bflam_{-}, \tr_{\Sigma} v|_{\Gamma_-}\spr_{\small\Hvar^{1}(\Sigma)'\times\Hvar^{1}(\Sigma)}\\
&\hspace{-25mm}=\spl \curl_{\Gamma_+} \bflam_{+},v|_{\Gamma_+}\spr_{\small\Hvar^{1/2}(\Gamma_+)'\times\Hvar^{1/2}(\Gamma_+)}
+\spl \curl_{\Gamma_-} \bflam_{-},v|_{\Gamma_-}\spr_{\small\Hvar^{1/2}(\Gamma_-)'\times\Hvar^{1/2}(\Gamma_-)}
\end{aligned}
\end{align}
for all $v\in\Hvar^{3/2}(\Gamma)$, from which the claim follows.
\end{proof}
\begin{lemma}\label{lem:tvtraces-pw}
If $\bflam\in \Hspace^{-1/2}(\curl_{\Gamma};\Gamma)$ and 
$\bflam|_{\Gamma_\pm}\in \Hspace^{-1/2}_{00}(\curl_{\Gamma_\pm};\Gamma_\pm)$,
then $\tr_{\tv_+}\bflam|_{\Gamma+}+\tr_{\tv_-}\bflam|_{\Gamma_-}=0$.
Like-wise,
If $\bflam\in \Hspace^{-1/2}(\div_{\Gamma};\Gamma)$ and 
$\bflam|_{\Gamma_\pm}\in \Hspace^{-1/2}_{00}(\div_{\Gamma_\pm};\Gamma_\pm)$,
then $\tr_{\nvt_+}\bflam|_{\Gamma+}+\tr_{\nvt_-}\bflam|_{\Gamma_-}=0$.
\end{lemma}
\begin{proof}
When only present the proof for the first statement.
Since $\curl_{\Gamma_\pm}(\bflam|_{\Gamma_\pm})=(\curl_{\Gamma}\bflam)|_{\Gamma_\pm}$, the claim follows from the first four lines of \eqref{eq:pwcurl} where $\bflam_{\pm}:=\bflam|_{\pm\Gamma}$.
\end{proof}

\subsection{Tangential traces of $\Hspace_{\Gamma_-}(\curlF;\Omega)$}

The operator $\tr$ applied to $\bfu\in\Hspace^1(\Omega)$ is interpreted component-wise.
Let
\begin{align*}
H_{\Gamma_-}^1(\Omega)&:=\{u\in H^1(\Omega)\colon \tr u=0 \text{ on }\Gamma_-\},\\
\Hspace_{\Gamma_-}^1(\Omega)&:=\{\bfu\in\Hspace^1(\Omega)\colon \tr\bfu=0 \text{ on }\Gamma_-\},\\
\Hspace_{\Gamma_-}(\curlF;\Omega)&:=\{\bfu\in\Hspace(\curlF;\Omega)\colon (\trtan\bfu)|_{\Gamma_-}=0 \text{ in } \Hspace^{1/2}_{RT,00}(\Gamma_-)'\},
\end{align*}
and we refer to \cite[Thm.~4.5]{BauerPaulySchomburg16}
for alternative, equivalent characterizations of these spaces.
In particular, $H^1_{\Gamma_-}(\Omega)=\ol{C^\infty_{\Gamma_-}(\Omega)}^{H^1(\Omega)}$, $\Hspace^1_{\Gamma_-}(\Omega)=\ol{\bfC^\infty_{\Gamma_-}(\Omega)}^{\Hspace^1(\Omega)}$, $\Hspace_{\Gamma_-}(\curlF;\Omega)=\ol{\bfC^\infty_{\Gamma_-}(\Omega)}^{\Hspace(\curlF;\Omega)}$ where
\begin{align*}
C^\infty_{\Gamma_-}(\Omega):=\{u|_{\Omega} \colon u\in C^\infty(\bbR^3) \text{ and } \dist(\supp u, \Gamma_-)>0\},\qquad
\bfC^\infty_{\Gamma_-}(\Omega):=(C^\infty_{\Gamma_-}(\Omega))^3.
\end{align*}
As discussed previously $\tr_{T,\Gamma_\pm}$ and $\tr_{RT,\Gamma_\pm}$ are well-defined on $\Hspace(\curlF;\Omega)$.
However, acting on $\Hspace_{\Gamma_-}(\curlF;\Omega)$ the partial traces $\tr_{T,\Gamma_\pm}$ and $\tr_{RT,\Gamma_\pm}$ map continuously into the stronger spaces
\begin{align*}
\Hspace^{-1/2}_{0}(\curl_{\Gamma_\pm};\Gamma_\pm)
&:=\{\bflam\in\Hspace^{-1/2}_{00}(\curl_{\Gamma_\pm};\Gamma_\pm)\colon \tr_{\tv_\pm}\bflam=0 \text{ on }\Sigma\},\\
\Hspace^{-1/2}_{0}(\div_{\Gamma_\pm};\Gamma_\pm)
&:=\{\bflam\in\Hspace^{-1/2}_{00}(\div_{\Gamma_\pm};\Gamma_\pm)\colon \tr_{\nvt_\pm}\bflam=0 \text{ on }\Sigma\}.
\end{align*}
The following results hold.
\begin{lemma}\label{lem:tracespace}
It holds that
\begin{align*}
\tr_{T,\Gamma_\pm} \in \BLO( \Hspace_{\Gamma_-}(\curlF;\Omega), \Hspace^{-1/2}_{0}(\curl_{\Gamma_\pm};\Gamma_\pm)),\quad
\tr_{RT,\Gamma_\pm} \in \BLO( \Hspace_{\Gamma_-}(\curlF;\Omega), \Hspace^{-1/2}_{0}(\div_{\Gamma_\pm};\Gamma_\pm)),
\end{align*}
$\tr_{T,\Gamma_+}$, $\tr_{RT,\Gamma_+}$ are surjective and admit continuous right-inverses.
\end{lemma}
\begin{proof}
We only present the proof for the first case.
Note that $\tr_{T,\Gamma_-}=0$.
Since
\begin{align*}
\spl \tr_{T,\Gamma_+}\bfu,\bfv|_{\Gamma_+} \spr_{\Hspace^{1/2}_{RT}(\Gamma_+)'\times \Hspace^{1/2}_{RT}(\Gamma_+)}&= \spl \tr_{T}\bfu, \bfv \spr_{\Hspace^{1/2}_{RT}(\Gamma)'\times \Hspace^{1/2}_{RT}(\Gamma)} \quad\forall\bfv\in\Hspace^{1/2}_{RT}(\Gamma),\\
\spl \curl_{\Gamma_+}\tr_{T,\Gamma_+}\bfu,v|_{\Gamma_+} \spr_{H^{1/2}(\Gamma_+)'\times H^{1/2}(\Gamma_+)}&=\spl \curl_{\Gamma}\tr_{T}\bfu,v \spr_{H^{1/2}(\Gamma)'\times H^{1/2}(\Gamma)} \quad\forall v\in H^{1/2}(\Gamma),
\end{align*}
it follows that $\tr_{T,\Gamma_\pm} \in \BLO( \Hspace_{\Gamma_-}(\curlF;\Omega), \Hspace^{-1/2}_{00}(\curl_{\Gamma_\pm};\Gamma_\pm))$.
Due to \Cref{lem:tvtraces-pw} we have that indeed $\tr_{T,\Gamma_\pm} \in \BLO( \Hspace_{\Gamma_-}(\curlF;\Omega), \Hspace^{-1/2}_{0}(\curl_{\Gamma_\pm};\Gamma_\pm))$.
The surjectivity and existence of a continuous right-inverse are deduced from the respective properties of $\tr_T$.
\end{proof}
Let
\begin{align*}
\bfC^\infty_{T,00}(\Gamma_+)&:=\tr_{T,\Gamma_+} \bfC^\infty_{\Gamma_-}
=\{\nvf\times\bfu|_{\Gamma_+}\times\nvf\colon \bfu\in\bfC^\infty_{\Gamma_-}\},\\
\bfC^\infty_{RT,00}(\Gamma_+)&:=\tr_{RT,\Gamma_+} \bfC^\infty_{\Gamma_-}
=\{-\nvf\times\bfu|_{\Gamma_+}\colon \bfu\in\bfC^\infty_{\Gamma_-}\}.
\end{align*}
\begin{lemma}\label{lem:tracespace-density}
The space $\bfC^\infty_{T,00}(\Gamma_+)$ is dense in $\Hspace^{-1/2}_{0}(\curl_{\Gamma_+};\Gamma_+)$ and $\bfC^\infty_{RT,00}(\Gamma_+)$ is dense in $\Hspace^{-1/2}_{0}(\div_{\Gamma_+};\Gamma_+)$.
\end{lemma}
\begin{proof}
Follows from \Cref{lem:tracespace} and the density of $\bfC^\infty_{\Gamma_-}$ in $\Hspace_{\Gamma_-}(\curlF;\Omega)$.
\end{proof}
\begin{lemma}\label{lem:tracespace-inherent}
Let $\Omega^j$, $j=1,2$ be two bounded Lipschitz domains with respective Lipschitz dissections $\partial\Omega^j=\Gamma_+^j \dot\cup \Gamma_- \dot\cup \Sigma^j$, $\Sigma^j:=\overline{\Gamma_+^j}\cap\overline{\Gamma_-^j}$, $j=1,2$ such that
$\Gamma_+^1=\Gamma_+^2=:\Gamma_+$ with $\nvf^1=\fraks\nvf^2$, $\fraks=+1$ or $\fraks=-1$.
Then
\begin{align*}
\Hspace^{-1/2}_{0}(\curl_{\Gamma_+^1};\Gamma_+^1)
&= \Hspace^{-1/2}_{0}(\curl_{\Gamma_+^2};\Gamma_+^2)
=:\Hspace^{-1/2}_{0}(\curl_{\Gamma_+};\Gamma_+),\\
\Hspace^{-1/2}_{0}(\div_{\Gamma_+^1};\Gamma_+^1)
&= \Hspace^{-1/2}_{0}(\div_{\Gamma_+^2};\Gamma_+^2)
=:\Hspace^{-1/2}_{0}(\div_{\Gamma_+};\Gamma_+),
\end{align*}
and
\begin{align*}
\curl_{\Gamma_+^1}=\fraks\curl_{\Gamma_+^2},\qquad
\div_{\Gamma_+^1}=\fraks\div_{\Gamma_+^2},
\end{align*}
i.e., $\Hspace^{-1/2}_{0}(\curl_{\Gamma_+};\Gamma_+)$ and $\Hspace^{-1/2}_{0}(\div_{\Gamma_+};\Gamma_+)$ are independent of the Lipschitz manifold $\Gamma$ in which $\Gamma_+$ is embedded in.
\end{lemma}
\begin{proof}
We only present the proof for the first case.
The space $H^{1/2}(\Gamma_+)$ admits a $\Gamma$-independent characterization by means of the Sobolev-Slobodeckij norm, and hence does $\Hspace^{1/2}_{RT}(\Gamma_+)'$.
For $\bflam\in\bfC^\infty_{T,00}(\Gamma_+)$ we have that $\curl_{\Gamma_+^1}\bflam=\fraks\curl_{\Gamma_+^2}\bflam$ and thus the claim follows from the density of $\bfC^\infty_{T,00}(\Gamma_+)$.
\end{proof}
\begin{remark}\label{rem:localization}
Note that although \Cref{lem:tracespace-inherent} yields that $\Hspace^{-1/2}_{0}(\curl_{\Gamma_+};\Gamma_+)$ and $\Hspace^{-1/2}_{0}(\div_{\Gamma_+};\Gamma_+)$ are independent of the \emph{particular} larger Lipschitz manifold $\Gamma\supset \Gamma^+$, their definitions, the operators $\curl_{\Gamma_+}$, $\div_{\Gamma_+}$ and $\tr_{\tv_+}$, $\tr_{\nvt_+}$ to be more precise, still involve a lifting to a volumetric domain $\Omega$.
However, for (piece-wise) smooth $\Gamma_+$ equivalent localized definitions of $\Hvar^{3/2}(\Gamma_+)$, $\Hvar^1(\Sigma)$ exist and hence localized definitions of $\Hspace^{-1/2}_{0}(\curl_{\Gamma_+};\Gamma_+)$ and $\Hspace^{-1/2}_{0}(\div_{\Gamma_+};\Gamma_+)$ are possible.
\end{remark}
We recall the bounded extensions
\begin{align*}
\nabla_{\Gamma} &\in \BLO( H^{1/2}(\Gamma), \Hspace^{1/2}_{RT}(\Gamma)' ),
&\Curl_{\Gamma} &\in \BLO( H^{1/2}(\Gamma), \Hspace^{1/2}_{T}(\Gamma)' ),\\
\nabla_{\Gamma_+} &\in \BLO( H^{1/2}_{00}(\Gamma_+), \Hspace^{1/2}_{RT}(\Gamma_+)' ),
&\Curl_{\Gamma_+} &\in \BLO( H^{1/2}_{00}(\Gamma_+), \Hspace^{1/2}_{T}(\Gamma_+)' ).
\end{align*}
We call an algebraic decomposition of a Banach space \emph{regular}, if it is a topological decomposition (see, e.g., \cite[Def.~3.4]{HallaHohage21} for a definition).
The following is a consequence of \cite{PaulySchomburg22}.
\begin{lemma}\label{lem:tracespace-decomp-reg}
The decompositions
\begin{align*}
\Hspace^{-1/2}_{0}(\curl_{\Gamma_+};\Gamma_+)&=
\Hspace^{1/2}_{T,00}(\Gamma_+)
\oplus \nabla_{\Gamma_+} H^{1/2}_{00}(\Gamma_+),\\
\Hspace^{-1/2}_{0}(\div_{\Gamma_+};\Gamma_+)&=
\Hspace^{1/2}_{RT,00}(\Gamma_+)
\oplus \Curl_{\Gamma_+} H^{1/2}_{00}(\Gamma_+)
\end{align*}
are regular.
\end{lemma}
\begin{proof}
The claim follows from \Cref{lem:tracespace} and the regular decomposition \cite[Lem.~4.6]{PaulySchomburg22}
\begin{align}\label{eq:decomp_Hcurldir}
\Hspace_{\Gamma_-}(\curlF;\Omega)=\Hspace^1_{\Gamma_-}(\Omega)\oplus\gradF H^1_{\Gamma_-}(\Omega).
\end{align}
\end{proof}
\begin{definition}\label{def:dualitypairing-partial}
For $\bflam\in\Hspace^{-1/2}_0(\div_{\Gamma_+};\Gamma_+)$ and $\bflam\tf\in\Hspace^{-1/2}_0(\curl_{\Gamma_+};\Gamma_+)$ we define the duality pairings
\begin{align*}
\spl \bflam, \bflam\tf \spr_{\Gamma_+,-1/2,0}&:=
\spl \bfv, \bfv\tf \spr_{\Lspace^2_T(\Gamma_+)}
+\spl \curl_{\Gamma_+}\bfv, w\tf \spr_{H^{1/2}(\Gamma_+)'\times H^{1/2}(\Gamma_+)}\\
&\hspace{26mm}-\spl w, \div_{\Gamma_+}\bfv\tf \spr_{H^{1/2}(\Gamma_+)\times H^{1/2}(\Gamma_+)'},\\
\spl \bflam\tf, \bflam \spr_{\Gamma_+,-1/2,0}&:=\overline{\spl \bflam, \bflam' \spr_{\Gamma_+,-1/2,0}},
\end{align*}
with decomposition $\bflam=\bv+\Curl_{\Gamma_+} w$, $\bflam\tf=\bv\tf+\nabla_{\Gamma_+} w\tf$, $\bv\in\Hspace^{1/2}_{T,00}(\Gamma_+)$, $\bv\tf\in\Hspace^{1/2}_{RT,00}(\Gamma_+)$, $w,w\tf\in H^{1/2}_{00}(\Gamma_+)$ as in \Cref{lem:tracespace-decomp-reg}.
\end{definition}
\begin{lemma}\label{lem:tracespace-intbyparts}
It holds that
\begin{align*}
\spl \curlF\bfu,\bfu\tf\spr_\Omega -\spl \bfu,\curlF\bfu\tf\spr_\Omega
=&-\spl \tr_{RT,\Gamma_+}\bfv,\tr_{T,\Gamma_+}\bfv\tf\spr_{\Gamma_{+,-1/2,0}}
\end{align*}
for all $\bfu,\bfu\tf\in\Hspace_{\Gamma_-}(\curlF;\Omega)$.
\end{lemma}
\begin{proof}
Follows like in \cite[(35)-(36)]{BuffaCostabelSheen02}.
\end{proof}
Let $\sigma\in L^\infty(\Gamma_+,\Rp)$ such that also $\sigma^{-1}\in L^\infty(\Gamma_+,\Rp)$ and $\Sigma_n$, $n=0,\dots,N_\Sigma$ be the $N_\Sigma+1$ connected parts of $\Sigma$ with $\Sigma_0$ being the exterior part.
Further, let $\phi_n\in H^1(\Gamma_+)$, $n=1,\dots,N_\Sigma$ be the solutions to
\begin{align*}
-\div_{\Gamma_+}\sigma^{-1}\nabla_{\Gamma_+} \phi_n=0 \quad\text{in }\Gamma_+,\qquad
\phi_n|_{\Sigma_m}=\delta_{nm} \quad\text{for all }m=0,\dots,N_\Sigma.
\end{align*}
Thence, let
\begin{align*}
\calK_\sigma(\Gamma_+)&:=\Span \{\phi_n, n=1,\dots,N_\Sigma\},\\
\calH^{0}_\sigma(\Gamma_+)&:=H^1_*(\Gamma_+),\\
\calH^{1/2}_\sigma(\Gamma_+)&:=\{u\in H^1_*(\Gamma_+)\colon \quad \div_{\Gamma_+}\sigma\nabla_{\Gamma_+}u\in H^{1/2}(\Gamma_+)' \quad\text{and}\quad \tr_{\nvt_+}\sigma\nabla_{\Gamma_+}u=0\},\\
\calH^{1}_\sigma(\Gamma_+)&:=\{u\in H^1_*(\Gamma_+)\colon \quad \div_{\Gamma_+}\sigma\nabla_{\Gamma_+}u\in L^{2}(\Gamma_+) \hspace{4.65mm}\quad\text{and}\quad \tr_{\nvt_+}\sigma\nabla_{\Gamma_+}u=0\}.
\end{align*}
\begin{assumption}\label{ass:Gamma-top}
There exists a simply connected Lipschitz domain $\Omega_0$ such that $\Gamma_+\subset\partial\Omega_0$.
\end{assumption}
\begin{remark}\label{rem:ass-top-flat}
	Note that if $\Gamma_+$ is flat, i.e., $\Gamma_+=\Upsilon\times\{z\}$, $\Upsilon\subset(-c,c)^2, c>0$, $z\in\bbR$, then \Cref{ass:Gamma-top} is satisfied, because $\Omega_0=(-c,c)^2\times(z,z+1)$ is a suitable simply connected domain.
\end{remark}
\begin{lemma}\label{lem:tracespace-decomp-orth}
Let \Cref{ass:Gamma-top} be satisfied.
Then, the decompositions
\begin{align*}
\Hspace^{-1/2}_{0}(\curl_{\Gamma_+};\Gamma_+)&=
\sigma\Curl_{\Gamma_+}\calH^{1/2}_\sigma(\Gamma_+)
\oplus \nabla_{\Gamma_+} \calK_\sigma(\Gamma_+)
\oplus \nabla_{\Gamma_+} H^{1/2}_{00}(\Gamma_+),\\
\Hspace^{-1/2}_{0}(\div_{\Gamma_+};\Gamma_+)&=
\sigma\nabla_{\Gamma_+}\calH^{1/2}_\sigma(\Gamma_+)
\oplus \Curl_{\Gamma_+} \calK_\sigma(\Gamma_+)
\oplus \Curl_{\Gamma_+} H^{1/2}_{00}(\Gamma_+)
\end{align*}
are regular.
\end{lemma}
\begin{proof}
We only present the proof for the first case.
Without loss of generality we assume that $\Omega$ is simply connected.
For $\bflam\in\Hspace^{-1/2}_{0}(\curl_{\Gamma_+};\Gamma_+)$ let $v\in H^1_*(\Gamma_+)$ be the solution to
\[
\spl \sigma\Curl_{\Gamma_+} v,\Curl_{\Gamma_+} v\tf \spr_{\Lspace^2_T(\Gamma_+)} = \spl \bflam,\Curl_{\Gamma_+} v\tf \spr_{\Lspace^2_T(\Gamma_+)}
\]
for all $v\tf\in H^1_*(\Gamma_+)$, i.e., $\curl_{\Gamma_+}\sigma\Curl_{\Gamma_+} v=\curl_{\Gamma_+}\bflam \in H^{1/2}(\Gamma_+)'$ and $0=$ $\tr_{\nvt_+}\sigma\nabla_{\Gamma_+}v$ $=$ $\tr_{\tv_+}\sigma\Curl_{\Gamma_+}u$.
Thence $v\in\calH^{1/2}_\sigma(\Gamma_+)$.
Let $\widetilde\bflam$, $\widetilde v$ be the extenions by zero to $\Gamma$ of $\bflam$, $v$ respectively.
Due to \Cref{lem:tracespace-pw} we have $\widetilde\bflam-\sigma\Curl_{\Gamma_+}\widetilde v\in\Hspace^{-1/2}(\curl_\Gamma;\Gamma)$ and $\curl_\Gamma( \widetilde\bflam-\sigma\Curl_{\Gamma_+}\widetilde v)=0$.
It follows from \cite[Thm.~5.1]{BuffaCostabelSheen02} that there exists $y\in H^{1/2}(\Gamma)$ such that $\widetilde\bflam-\sigma\Curl_{\Gamma_+}\widetilde v=\nabla_\Gamma y$.
Let us now prove that $(\nabla_\Gamma y)|_{\Gamma_-}=0$ implies that $y_n:=y|_{\Gamma_-,n}$, $n=1,\dots,N$ are constant with $\Gamma_{-,n}$, $n=1,\dots,N$ being the $N\in\bbN$ connected parts of $\Gamma_-$.
Let $\bfq\in\Hspace^{1/2}_{RT,00}(\Gamma_{-,n})$ and $\tilde\bfq$ be it's extension by zero to $\Gamma$.
Then, $0=\spl \nabla_{\Gamma} y, \tilde\bfq \spr_{\Hspace^{1/2}_{RT}(\Gamma)'\times\Hspace^{1/2}_{RT}(\Gamma)} = -\spl y_n,\div_{\Gamma_-,n}\bfq\spr_{H^{1/2}(\Gamma_{-,n})\times H^{1/2}(\Gamma_{-,n})'}$.
Since $\div_{\Gamma_-,n}\Hspace^{1/2}_{RT,00}(\Gamma_{-,n})=\div_{\Gamma_-,n}\big(\Hspace^{1/2}_{RT,00}(\Gamma_{-,n})\oplus\Curl_{\Gamma_{-,n}} H^{1/2}_{00}(\Gamma_{-,n})\big)=\div_{\Gamma_-,n}\Hspace^{-1/2}_{RT,0}(\div_{\Gamma_{-,n}};\Gamma_{-,n})=H^{1/2}(\Gamma_{-,n})'_*$ it follows that $y_n$ is constant.
Now, let $z\in\calK_\sigma(\Gamma_+)$ be such that $z=y$ on $\Sigma$, and $\tilde z\in H^1(\Gamma)$ be defined by $\tilde z|_{\Gamma_+}=z$, $\tilde z|_{\Gamma_-}=y|_{\Gamma_-}$.
We have $y-\tilde z\in H^{1/2}(\Gamma)$ and $(y-\tilde z)|_{\Gamma_-}=0$, i.e., $w:=(y-\tilde z)|_{\Gamma_+}\in H^{1/2}_{00}(\Gamma_+)$.
Thence $\bflam=\sigma\Curl_{\Gamma_+}v+\nabla_{\Gamma_+}z+\nabla_{\Gamma_+}w$ is the claimed decomposition.
\end{proof}
\begin{lemma}\label{lem:tracespace-decomp-orth-Hone}
Let \Cref{ass:Gamma-top} be satisfied.
Then, the decompositions
\begin{align*}
\Hspace_{0}(\curl_{\Gamma_+};\Gamma_+)&=
\sigma\Curl_{\Gamma_+}\calH^1_\sigma(\Gamma_+)
\oplus \nabla_{\Gamma_+} \calK_\sigma(\Gamma_+)
\oplus \nabla_{\Gamma_+} H^{1}_{0}(\Gamma_+),\\
\Hspace_{0}(\div_{\Gamma_+};\Gamma_+)&=
\sigma\nabla_{\Gamma_+}\calH^1_\sigma(\Gamma_+)
\oplus \Curl_{\Gamma_+} \calK_\sigma(\Gamma_+)
\oplus \Curl_{\Gamma_+} H^{1}_{0}(\Gamma_+),\\
\Lspace^2_T(\Gamma_+)&=
\sigma\Curl_{\Gamma_+}\calH^0_\sigma(\Gamma_+)
\oplus \nabla_{\Gamma_+} \calK_\sigma(\Gamma_+)
\oplus \nabla_{\Gamma_+} H^{1}_{0}(\Gamma_+),\\
\Lspace^2_T(\Gamma_+)&=
\sigma\nabla_{\Gamma_+}\calH^0_\sigma(\Gamma_+)
\oplus \Curl_{\Gamma_+} \calK_\sigma(\Gamma_+)
\oplus \Curl_{\Gamma_+} H^{1}_{0}(\Gamma_+)
\end{align*}
are regular.
\end{lemma}
\begin{proof}
The proof proceeds in a similar fashion to the proof of \Cref{lem:tracespace-decomp-orth}.
\end{proof}
Let us briefly recall the concept of interpolation of Hilbert spaces, see. e.g., \cite{BerghLoefstroem76,Bramble93}.
Consider a Hilbert space $H_0$ and a densly defined, positive, selfadjoint $\Lambda\colon \opdom(\Lambda)\to H_0$.
Let $H_1:=\opdom$, $\|\plh\|_{H_1}:=\|\Lambda\plh\|_{H_0}$.
Then, for $s\in(0,1)$ the $s$-Hilbert space interpolation of $H_0$ and $H_1$ is defined as
\begin{align*}
[H_0,H_1]_{H,s}:=\opdom(A^s),\qquad \|\plh\|_{[H_0,H_1]_{H,s}}:=\|\Lambda^s\plh\|_{H_0}.
\end{align*}
Now, for $t>0$, $u\in H_0$ let $K(t,u):=\inf_{u=u_0+u_1,u_0\in H_0,u_1\in H_1} \sqrt{\|u_0\|_{H_0}^2+t^2\|u_1\|_{H_1}^2}$.
Thence, for $s\in(0,1)$, $p\in [1,+\infty)$ the real $(s,p)$-interpolation space of $H_0$ and $H_1$ is defined as
\begin{align*}
[H_0,H_1]_{K,s,p}:=\{u\in H_0\colon \|u\|_{[H_0,H_1]_{K,s,p}}<+\infty\},\quad
\|u\|_{[H_0,H_1]_{K,s,p}}:=\bigg( \int_0^\infty t^{-sp} K^p(t,u)\frac{\dd t}{t} \bigg)^{1/p}.
\end{align*}
It is known that $[H_0,H_1]_{H,s}=[H_0,H_1]_{K,s,2}$ and $\|\plh\|_{[H_0,H_1]_{H,s}}= \sqrt{\frac{2}{\pi}\sin(\pi s)} \|\plh\|_{[H_0,H_1]_{K,s,2}}$.
\begin{lemma}\label{lem:interpolation-calH}
The space $\calH^{1/2}_\sigma(\Gamma_+)$ equals the $1/2$-Hilbert interpolation space of $\calH^{0}_\sigma(\Gamma_+)$ and $\calH^{1}_\sigma(\Gamma_+)$, i.e.,
\begin{align*}
\calH^{1/2}_\sigma(\Gamma_+) = [\calH^0_\sigma(\Gamma_+),\calH^1_\sigma(\Gamma_+)]_{H,1/2} = [\calH^0_\sigma(\Gamma_+),\calH^1_\sigma(\Gamma_+)]_{K,1/2,2}
\end{align*}
with equivalent norms.
\end{lemma}
\begin{proof}
For $u\in\calH^0_\sigma(\Gamma_+)$ let $f_u\in H^1(\Gamma_+)'_*$ be defined by
\[
\spl f_u,\plh\spr_{H^1(\Gamma_+)'\times H^1(\Gamma_+)}:=\spl \sigma\nabla_{\Gamma_+}u,\nabla_{\Gamma_+}\plh\spr_{\Lspace^2_T(\Gamma_+)}.
\]
Note that on $\calH^s_\sigma(\Gamma_+)$ we have the norm equivalences $\|f_u\|_{H^{s-1}(\Gamma_+)'} \sim \|u\|_{\calH^s_\sigma(\Gamma_+)}$, $s=0,1/2,1$ and the map $u\mapsto f_u$ is a bijection $\calH^s_\sigma(\Gamma_+) \to H^{s-1}(\Gamma_+)'_*$, $s=0,1/2,1$.
Let $f\mapsto u_f$ be it's inverse, i.e., $u_f\in\calH^0_\sigma(\Gamma_+)$ is the solution to
\[
\spl \sigma\nabla_{\Gamma_+}u_f,\nabla_{\Gamma_+}u\tf\spr_{\Lspace^2_T(\Gamma_+)}=\spl f_u,u\tf \spr_{H^1(\Gamma_+)'\times H^1(\Gamma_+)}
\]
for all $u\tf\in\calH^0_\sigma(\Gamma_+)$.
Indeed, for $s=1/2,1$ the previous equality implies $-\div_{\Gamma_+}\sigma\nabla_{\Gamma_+} u_f=f_u$ and $\tr_{\nvt_+}u_f=0$.
Now, applying the real interpolation method we obtain that $H^{1/2}(\Gamma_+)'_*=[L^2_*(\Gamma_+), H^{1}(\Gamma_+)'_*]_{K,1/2,2}$ and thus $\calH^{1/2}_\sigma(\Gamma_+) = [\calH^0_\sigma(\Gamma_+),\calH^1_\sigma(\Gamma_+)]_{K,1/2,2}$ with the norm equivalence $\|\plh\|_{\calH^{1/2}_\sigma(\Gamma_+)} \sim \|\plh\|_{[\calH^0_\sigma(\Gamma_+),\calH^1_\sigma(\Gamma_+)]_{K,1/2,2}}$.
Subsequently the equivalence to the Hilbert interpolation space follows too.
\end{proof}
Let us recall that for Hilbert spaces $X,Y$ the Schatten class $\calK_p(X,Y)$, $p\geq1$ is the space of compact operators $A\in L(X,Y)$ such that the sequence of singular values $(s_n(A))_{n\in\bbN}$ is in $\ell_p(\bbN)$.
An operator $A\in L(X,Y)$ is said to be of finite order, if there exist $p\geq1$ with $A\in\calK_p(X,Y)$.
\begin{lemma}\label{lem:HoneLtwoFO}
The embedding $H^s(\Gamma_+)\hookrightarrow L^2(\Gamma_+)$, $s\in(0,1]$ is of finite order.
\end{lemma}
\begin{proof}
Let $(\Gamma_{+,n})_{n=1}^N$, $N\in\bbN$ with pairwise-disjoint Lipschitz submanifolds $\Gamma_{+,n}\subset\Gamma_+$ such that $\overline{\Gamma_+}=\bigcup \overline{\Gamma_{+,n}}$ and Lipschitz coordinate maps $a_n\colon [0,1]^2\to\Gamma_n$.
Then, by flattening $\Gamma_{+,n}$ and applying \cite[Thm.~2.2, Lem.~2.3]{Halla23MSteklov} we obtain that the embeddings $H^s(\Gamma_{+,n})\hookrightarrow L^2(\Gamma_{+,n})$ are of finite order.
By identifying $L^2(\Gamma_+)\sim \bigoplus_{n=1}^N L^2(\Gamma_{+,n})$ and applying \cite[Lem.~2.3]{Halla23MSteklov} we obtain that $H^s(\Gamma_+)\hookrightarrow L^2(\Gamma_+)$ is of finite order.
\end{proof}
\begin{lemma}\label{lem:Vfiniteorder}
Let \Cref{ass:Gamma-top} be satisfied.
Then, the embedding $\big( \sigma\Curl_{\Gamma_+}\calH^1_\sigma(\Gamma_+), \|\plh\|_{\Hspace(\curl_{\Gamma_+};\Gamma_+)} \big) \hookrightarrow\Lspace^2_T(\Gamma_+)$ is of finite order.
\end{lemma}
\begin{proof}
We proceed as in \cite{Delitsyn00}.
We need to show that the ascending sequence of eigenvalues $(\lambda_n)_{n\in\bbN}$ to:
Find $(\lambda,v)\in\Rp\times \calH^1_\sigma(\Gamma_+)$ such that
\begin{align*}
\spl \curl_{\Gamma_+}\sigma\Curl_{\Gamma_+} v,\curl_{\Gamma_+}\sigma\Curl_{\Gamma_+} v\tf \spr_{L^2(\Gamma_+)}
=\lambda \spl \sigma\Curl_{\Gamma_+} v,\sigma\Curl_{\Gamma_+} v\tf \spr_{\Lspace^2_T(\Gamma_+)}
\end{align*}
for all $v\tf\in\calH^1_\sigma(\Gamma_+)$ satisfies $(1/\lambda_n)_{n\in\bbN}\in\ell_p(\bbN)$ for an arbitrary $p\geq1$.
Due to the Courant-Fischer min-max formula for eigenvalues it suffices to prove $(1/\tilde\lambda)_{n\in\bbN}\in\ell_p(\bbN)$ with $(\tilde\lam_n)_{n\in\bbN}$ being the ascending sequence of eigenvalues to: Find $(\tilde\lambda,v)\in\Rp\times \calH^1_\sigma(\Gamma_+)$ such that
\begin{align*}
\spl \curl_{\Gamma_+}\sigma\Curl_{\Gamma_+} v,\curl_{\Gamma_+}\sigma\Curl_{\Gamma_+} v\tf \spr_{L^2(\Gamma_+)}
&=\tilde\lambda \spl \Curl_{\Gamma_+} v,\sigma\Curl_{\Gamma_+} v\tf \spr_{\Lspace_T^2(\Gamma_+)}.\\&=\tilde\lambda \spl \curl_{\Gamma_+}\sigma\Curl_{\Gamma_+} v, v\tf \spr_{L^2(\Gamma_+)}
\end{align*}
for all $v\tf\in\calH^1_\sigma(\Gamma_+)$.
Thence with \Cref{lem:tracespace-decomp-orth-Hone} it can be shown that $(\tilde\lambda,v)$ is an eigenpair to the former, if and only if it is an an eingenpair to:
Find $(\tilde\lambda,\tilde v)\in\Rp\times H^1_*(\Gamma_+)$ such that
\begin{align*}
\spl \sigma\Curl_{\Gamma_+} \tilde v,\Curl_{\Gamma_+} v\tf \spr_{L^2(\Gamma_+)}
=\tilde\lambda \spl \tilde v, v\tf \spr_{L^2(\Gamma_+)}
\end{align*}
for all $v\tf\in H^1_*(\Gamma_+)$.
Now, due to \Cref{lem:HoneLtwoFO} there exists $p\geq1$ such that $(1/\tilde\lambda_n)_{n\in\bbN}\in\ell_p(\bbN)$.
Thus the claim is proven.
\end{proof}
We introduce
\begin{align*}
\calV_{\Gamma_-}(\Omega):=\{\bfu\in \Hspace_{\Gamma_-}(\curlF;\Omega)\colon \,\,\, \divF\bfu=0 \text{ in }\Omega, \quad \tr_{T,\Gamma_+}\bfu\in \sigma\Curl_{\Gamma_+}\calH^{1/2}_\sigma(\Gamma_+) \oplus \nabla_{\Gamma_+} \calK_\sigma(\Gamma_+)\}.
\end{align*}
\begin{lemma}\label{lem:tracespace-decomp-orth-Omega}
Let \Cref{ass:Gamma-top} be satisfied.
Then, the decomposition
\begin{align*}
\Hspace_{\Gamma_-}(\curlF;\Omega)=\calV_{\Gamma_-}(\Omega)\oplus\nabla H^1_{\Gamma_-}(\Omega)
\end{align*}
is regular.
\end{lemma}
\begin{proof}
For $\bfu\in\Hspace_{\Gamma_-}(\curlF;\Omega)$ let $\tr_{T,\Gamma_+}\bfu=\sigma\Curl_{\Gamma_+} v+\nabla_{\Gamma_+}z+\nabla_{\Gamma_+}w_{\Gamma_+}$ be a decomposition as in \Cref{lem:tracespace-decomp-orth}.
Let $w_1\in H^1_{\Gamma_-}(\Omega)$ solve $-\Delta w=0$ in $\Omega$ and $w=w_{\Gamma_+}$ on $\Gamma_+$.
Let $w_2\in H^1_*(\Omega)$ solve $\spl \gradF w_2,\gradF w\tf \spr_{\Lspace^2(\Omega)}$ $=$ $\spl \bfu-\gradF w_2,\gradF w\tf \spr_{\Lspace^2(\Omega)}$ for all $w\tf\in H^1_*(\Omega)$.
Then $\bfu=\bfv+\gradF w$ with $w:=w_1+w_2$, $\bfv:=\bfu-\gradF w$ is the claimed decomposition.
\end{proof}
\begin{lemma}\label{lem:HcurlDecompCompact}
Let \Cref{ass:Gamma-top} be satisfied.
Then, the embedding $\calV_{\Gamma_-}(\Omega)\hookrightarrow\Hspace^{1/2}(\Omega)$ is bounded and hence the embedding $\calV_{\Gamma_-}(\Omega)\hookrightarrow\Lspace^2(\Omega)$ is compact.
\end{lemma}
\begin{proof}
There exists $C>0$ such that $\|\tr_{T}\bfv\|_{\Lspace^2_T(\Gamma)} \leq C \|\bfv\|_{\Hspace(\curlF;\Omega)}$ for all $\bfv\in\calV_{\Gamma_-}(\Omega)$.
Hence the claims follow from \cite{Costabel90}.
\end{proof}
\begin{lemma}\label{lem:tracespace-intbyparts-two}
Let \Cref{ass:Gamma-top} be satisfied.
Then, it holds that
\begin{align*}
\spl \curlF\bfu,\bfu\tf\spr_\Omega -\spl \bfu,\curlF\bfu\tf\spr_\Omega
=&-\spl \tr_{RT,\Gamma_+}\bfv,\tr_{T,\Gamma_+}\bfv\tf\spr_{\Lspace^2_T(\Gamma_+)}\\
&-\spl \tr w,\curl_{\Gamma_+}\tr_{T,\Gamma_+}\bfv\tf\spr_{H^{1/2}(\Gamma_+)\times H^{1/2}(\Gamma_+)'}\\
&+\spl \div_{\Gamma_+}\tr_{RT,\Gamma_+}\bfv, \tr w\tf\spr_{H^{1/2}(\Gamma_+)'\times H^{1/2}(\Gamma_+)}
\end{align*}
for all $\bfu,\bfu\tf\in\Hspace_{\Gamma_-}(\curlF;\Omega)$ with $\bfu=\bfv+\nabla w$, $\bfu\tf=\bfv\tf+\nabla w\tf$ being decompositions as in \Cref{lem:tracespace-decomp-orth-Omega}.
\end{lemma}
\begin{proof}
Follows like in \Cref{lem:tracespace-intbyparts}.
See also \cite[Lem.~5.6]{BuffaCostabelSheen02}.
\end{proof}

%

\section{Unbounded operators}\label{sec:app_operatortheory}

For the benefit of the reader we collect in this appendix several results concerning unbounded operators, which are applied in \Cref{sec:Calderon}.
Our sources are the books of Kato \cite{Kato95} and Haase \cite{Haase06book}, and an article of Agranovich \cite{Agranovich94}.
Concerning the nomenclature we follow Haase's book \cite{Haase06book}, i.e., the attribute \emph{sectorial} is used differently than in Kato's book \cite{Kato95}, see \cite[Cha.~7.1,~7.3.2]{Haase06book}.

\subsection{Operators in Banach spaces}

In the following, let $X$ be a Banach space over $\bbC$ with norm $\|\plh\|$.

\subsubsection{Sectorial operators}

For $\theta\in[0,\pi]$ let
\begin{align*}
\Sector_\theta:=\begin{cases} z\in\Cx\colon |\arg z|<\theta, &\theta\in\text{$($}0,\pi]\\
(0,+\infty), &\theta=0 \end{cases}.
\end{align*}
An operator $\Aopu\colon\opdom(\Aopu)\subset X\to X$ is called \textbf{sectorial of angle} $\mathbf{\theta}$, in short $\Aopu\in\Sect(\theta)$, if
\begin{align*}
a)~\spec(\Aopu)\subset \overline{\Sector_\theta} \qquad\text{and}\qquad b)~\Mod(\Aopu,\theta'):=\sup\{ \|\lam(\lam-\Aopu)^{-1}\|\colon \lam\in\bbC\setminus\overline{\Sector_{\theta'}}\}<\infty\,\,\forall \theta'\in(\theta,\pi).
\end{align*}
An operator $\Aopu$ is called \textbf{quasi sectorial (of angle $\mathbf{\theta}$)}, if there exists $\lam\in\bbR$ such that $\Aopu+\lam$ is sectorial (of angle $\theta$).
For a sectorial operator $\Aopu$ we define it's \textbf{spectral/sectorial angle} as
\begin{align*}
\theta_{\Aopu}:=\min\{\theta\in\text{$[$}0,\pi)\colon \Aopu\in\Sect(\theta)\}.
\end{align*}

\subsubsection{Classes of holomorphic functions}

For an open set $\Omega\subset\bbC$ let $\calO(\Omega)$ and $\calM(\Omega)$ be the spaces of \textbf{holomorphic} and \textbf{meromorphic functions} $\Omega\to\bbC$.
Let $\ome\in(0,\pi]$ and $f\in\calM(\Sector_\ome)$.
We say that $f$ has \textbf{polynomial limit} $c\in\bbC$ at $0$ if there is $\alpha>0$ such that $f(z)-c=O(|z|^\alpha)$ as $z\to0$.
We say that $f$ has polynomial limit $\infty$ at $0$ if $1/f$ has polynomial limit $0$ at $0$.
We say that $f$ has polynomial limit $c\in\bbC\cup\{\infty\}$ at $\infty$ if $f(z^{-1})$ has polynomial limit $c$ at $0$.
We say that $f$ has a \textbf{finite polynomial limit} at $0$ ($\infty$) if there is $c\in\bbC$ such that $f$ has polynomial limit $c$ at $0$ ($\infty$).
We say that $f$ is \textbf{regularly decaying} at $0$ ($\infty$) has if $f$ has polynomial limit $0$ at $0$ ($\infty$).
Let
\begin{align*}
H^\infty(\Sector_\ome):=\{f\in\calO(\Sector_\ome)\colon f\in L^\infty(\Sector_\ome)\},
\qquad
\|f\|_\infty:=\|f\|_{\infty,\Sector_\ome}:=\|f\|_{L^\infty(\Sector_\ome)}
\end{align*}
be the Banach algebra of all \textbf{bounded, holomorphic functions} on $\Sector_\ome$,
\begin{align*}
H_0^\infty(\Sector_\ome):=\{f\in H^\infty(\Sector_\ome)\colon f \text{ is regularly decaying at } 0 \text{ and at } \infty\}
\end{align*}
be the \textbf{Dunford-Riesz class} on $\Sector_\ome$ and
\begin{align*}
\calE(\Sector_\ome):=H_0^\infty(\Sector_\ome) \oplus \big\spl (1+z)^{-1} \big\spr \oplus \big\spl 1 \big\spr
\end{align*}
be the \textbf{extended Dunford-Riesz class}.
A function $f\in\calO(\Sector_\ome)$ belongs to $\calE(\Sector_\ome)$, if and only if $f$ is bounded and has finite polynomials limits at $0$ and $\infty$ \cite[Lem.~2.2.3]{Haase06book}.
We further set
\begin{align*}
H^\infty_{(0)}(\Sector_\theta):=\{f\in H^\infty(\Sector_\theta)\colon f \text{ is regularly decaying at } \infty \text{ and holomorphic at } 0\}
\end{align*}
and
\begin{align*}
\calE_\infty(\Sector_\theta):=\{f\in\calO(\Sector_\theta)\colon \exists \alpha>0 \text{ such that } f=O(|z|^{-\alpha}) \text{ as }z\to\infty\}.
\end{align*}
%

\subsubsection{Functional calculus}\label{subsubsec:functcalc}

\paragraph{Primary functional calculus}

Let $\theta\in(0,\pi)$ and $\del>0$.
We consider $\Gamma_\theta:=\partial\Sector_\theta$ oriented from $\infty e^{i\theta}$ to $\infty e^{-i\theta}$.
Further, let $B_\del(0):=\{z\in\bbC\colon|z|<\del\}$, $\Gamma_{\theta,\del,-}:=\partial (\Sector_\theta\cup B_\del(0))$ b and $\Gamma_{\theta,\del,+}:=\partial (\Sector_\theta\setminus B_\del(0))$ be oriented like $\Gamma_\theta$.
Let $A\in\Sect(\ome)$.
For $\theta\in(\ome,\pi)$ and $f\in H_0^\infty(\Sector_\theta)$ we define
\begin{align*}
f(A):=\frac{1}{2\pi i} \int_{\Gamma_{\ome'}} f(z) (z-A)^{-1} \,\dd z,
\end{align*}
where $\ome'\in(\ome,\theta)$ is arbitrary.
Properties of $H_0^\infty(\Sector_\theta)\to\BLO(X)\colon f\mapsto f(A)$ are stated in \cite[Lem.~2.3.1]{Haase06book}.
If $f\in H_{(0)}^\infty(\Sector_\theta)$ and $\del>0$ is such that $f$ is holomorphic in a neighborhood of $\overline{B_\del(0)}$, then $f(A)=\frac{1}{2\pi i} \int_{\Gamma_{\ome',\del,-}} f(z) (z-A)^{-1} \,\dd z$ \cite[Lem.~2.3.2]{Haase06book}.
For $g\in\calE(\Sector_\ome)$ we define
\begin{align*}
g(A):=f(A)+c(1+A)^{-1}+d,
\end{align*}
whenever $g=f+c(1+z)^{-1}+d$, $f\in H_0^\infty(\Sector_\theta)$, $c,d\in\bbC$.
This constitutes the \textbf{primary functional calculus} for sectorial operators $\calE(\Sector_\theta)\to\BLO(X)\colon g\mapsto g(A)$ with properties listed in \cite[Thm.~2.3.3]{Haase06book}.

\paragraph{Natural functional calculus}

Let
\begin{align*}
\calM(\Sector_\theta)_{\Aopu}:=\{f\in\calM(\Sector_\theta)\colon \exists e\in\calE(\Sector_\theta) \text{ such that } e(A) \text{ is injective and } ef\in\calE(\Sector_\theta)\}.
\end{align*}
Then, for $f\in\calM(\Sector_\theta)_{\Aopu}$ the extension technique allows us to define
\begin{align*}
f(A):=e(A)^{-1} (ef)(A),
\end{align*}
with $e\calE(\Sector_\theta)$ being an arbitrary regularizer of $f$.
For $\theta_2>\theta_1>\ome$ we have the subalgebras $\calE(\Sector_{\theta_1})\subset\calE(\Sector_{\theta_2})$ and $\calM(\Sector_{\theta_1})\subset\calM(\Sector_{\theta_2})$, and note that the respective functional calculi are consistent.
Thus we set $\calE[\Sector_\ome]:=\bigcup_{\theta\in(\ome,\pi)} \calE(\Sector_\theta)$ and $\calM[\Sector_\ome]:=\bigcup_{\theta\in(\ome,\pi)} \calM(\Sector_\theta)$ call the associated calculus the \textbf{natural functional calculus} for $\Aopu$ as a sectorial operator on $\Sector_\ome$.
We further introduce the \textbf{domain} $\calM[\Sector_\ome]_{\Aopu}$ of this calculus as
\begin{align*}
\calM[\Sector_\ome]_{\Aopu}:=\bigcup_{\theta\in(\ome,\pi)} \calM(\Sector_\theta)_{\Aopu},\qquad
H(A):=\{f\in \calM[\Sector_\ome]_{\Aopu} \colon f(A)\in\BLO(X)\}.
\end{align*}
If $\Aopu\in\Sect(\ome)$ is invertible and $\delta\in (0,1/\|\Aopu^{-1}\|)$, then we consider
\begin{align*}
\calE_\infty[\Sector_\ome]:=\bigcup_{\theta\in(\ome,\pi)} \calE_\infty(\Sector_\theta)
\end{align*}
and a functional calculus being consistent with the previous one can be defined via
\begin{align*}
f(A):=\frac{1}{2\pi i} \int_{\Gamma_{\ome',\del,+}} f(z) (z-A)^{-1} \,\dd z.
\end{align*}


\paragraph{Fractional powers}

For	$\alpha>0$ the function $f(z)=z^\alpha$ is holomorphic in $\Sector_\theta$, has a polynomial limit $0$ at $0$ and a polynomial limit $\infty$ at $\infty$.
Hence $f(A)$ is well-defined by the natural functional calculus for sectorial operators.
In particular, for $n\in\bbN$ with $n>\alpha$ it holds that $z^\alpha(1+z)^{-n}\in H^\infty_0(\Sector_\theta)$ and
\begin{align*}
\Aopu^\alpha:=(z^\alpha)(A)=(1+A)^n\bigg(\frac{z^\alpha}{(1+z)^n}\bigg) (A),
\end{align*}
which we call the \textbf{fractional power} with exponent $\alpha$ of $\Aopu$.
It holds that \cite[Prop.~3.1.1]{Haase06book}
\begin{align*}
\Aopu^{\alpha+\beta}&=\Aopu^\alpha\Aopu^\beta \forall\alpha,\beta>0;\\
\opdom(A^\gamma)&\subset\opdom(A^\alpha) \text{ for }0<\alpha<\gamma;\\
\text{if }&T\in\BLO(X) \text{ commutes with }\Aopu, \text{ then it also commutes with }\Aopu^\alpha;\\
\text{if }&\Aopu\text{ is densely defined and } \alpha\in(0,n),n\in\bbN, \text{ then the space } \opdom(\Aopu^n) \text{ is a core for } \Aopu^\alpha;\\
\spec(\Aopu^\alpha)&=\{z^\alpha\colon z\in\spec(\Aopu)\}.
\end{align*}
Furthermore \cite[Prop.~3.1.2]{Haase06book},
\begin{align*}
\Aopu^\alpha \text{ is sectorial with } \ome_{\Aopu^\alpha}=\alpha\ome_{\Aopu} \text{ for all }\alpha\in(0,\pi/\ome_{\Aopu}), \Aopu\in\Sect(\ome).
\end{align*}
and \cite[Cor.~3.1.5]{Haase06book}
\begin{align*}
(\Aopu^\alpha)^\beta=\Aopu^{\alpha\beta} \text{for all } \Aopu\in\Sect(\ome), \alpha\in(0,\pi/\ome), \beta>0.
\end{align*}
and also \cite[Prop~3.1.12]{Haase06book}
\begin{align*}
A^\alpha\afu&=\frac{\sin \alpha\pi}{\pi} \int_0^\infty t^{\alpha-1} (t+\Aopu)^{-1}\Aopu\afu\,\dd t.
\end{align*}
If $\Aopu$ is invertible, then $\Aopu^\alpha:=(z^\alpha)(\Aopu)$ is well-defined for arbitrary $\alpha\in\bbC$ and \cite[Prop~3.2.1]{Haase06book}
\begin{subequations}
\label{eq:frac-power-neg}
\begin{align}
(\Aopu^{-1})^\alpha&=(A^\alpha)^{-1};\\
A^{-\alpha}\afu&=\frac{\sin \alpha\pi}{\pi} \int_0^\infty t^{-\alpha} (t+\Aopu)^{-1}\afu\,\dd t \quad\text{for }\alpha\in(0,1).
\end{align}
\end{subequations}

\subsubsection{Semi-groups}

A \textbf{semigroup} is a strongly continuous map $T\colon(0,\infty)\to\BLO(X)$ that possesses the \textbf{semigroup property}
\begin{align*}
T(t)T(s)=T(t+s) \quad\forall t,s>0.
\end{align*}
The semigroup us called \textbf{bounded} if $\sup_{t\in(0,\infty)} \|T(t)\|<\infty$, and called \textbf{exponentially bounded} if there exist constants $M\geq1,\ome\in\bbR$ such that $\|T(t)\|\leq Me^{\ome t}$ for all $t>0$.
The number
\begin{align*}
\ome_0(T):=\inf \{\ome\in\bbR\colon \exists M>0 \text{ such that } \|T(t)\|\leq Me^{\ome t} \text{ for all }t>0\}
\end{align*}
is called the \textbf{growth bound} of $T$, and $T$ is called \textbf{exponentially stable} if $\ome_0(T)<0$.
An exponentially stable semigroup with constants $M=1$, $\ome=0$ is called a \textbf{contraction semigroup}.
The semigroup $T$ is called \textbf{quasi contractive} if there exists $\ome\geq0$ such that $t\mapsto e^{\ome t}T(t)$ is contraction semigroup.
The space $\{\afu\in X\colon \lim_{t\to0+} T(t)\afu=\afu\}$ is called the \textbf{space of strong continuity} of $T$.
If it coincides with $X$, then the semigroup is called \textbf{strongly continuous} or a \textbf{$C_0$-semigroup}.\\

For an exponentially stable semigroup $T$ with constants $(M,\ome)$ the Laplace transform of $T$ exists at least in the half-plane $\{\lam\in\bbC\colon \Re\lam>\ome\}$, i.e.,
\begin{align*}
\hat T(t)\afu:=\int_0^\infty e^{-\lam t}T(t)\afu \,\dd t, \quad\afu\in X
\end{align*}
defines a bounded operator on $X$ for each $\lam$ with $\Re\lam>\ome$,
and there exists a unique operator $\Aopu\colon\opdom(\Aopu)\subset X\to X$ such that
\begin{align*}
(\lam-A)^{-1} = \int_0^\infty e^{-\lam t}T(t) \,\dd t.
\end{align*}
The operator $A$ is called the \textbf{generator} of the semigroup $T$.
The semigroup $T$ is uniquely determined by its generator.
The space of strong continuity of $T$ coincides with $\overline{\opdom(\Aopu)}$ \cite[Cor.~A.8.3]{Haase06book}.
If $\lam\in\bbC$, then $A+\lam$ generates the semigroup $t\mapsto e^{\lam t}T(t)$.\\

Let $\theta\in(0,\pi/2]$.
A map $T\colon \Sector_\theta\to\BLO(X)$ is a called an \textbf{exponentially bounded holomorphic semigroup} of angle $\theta$, if
\begin{subequations}
\begin{align}
T(\lam)T(\mu)=T(\lam+\mu) \text{ for all }\lam,\mu\in\Sector_\theta;\\
T\colon \Sector_\theta\to\BLO(X) \text{ is holomorphic};\\
\forall\phi\in(0,\theta)\exists M_\phi\geq,w_\phi\geq0 \text{ such that } \|T(\lam)\|\leq M_\phi e^{w_\phi\Re\lam} \text{ for all }\lam\in\Sector_\phi.
\end{align}
\end{subequations}
In particular, for such $T$ the restriction $T|_{0,\infty}$ is an exponentially bounded semigroup.
If above it is possible to choose $w_\phi=0$, then $T$ is called a \textbf{bounded holomorphic semigroup} of angle $\theta$.\\

Let $\Aopu\in\Sect(\ome)$ with $\ome\in[0,\pi/2)$.
By means of the functional calculus we can define
\begin{align*}
e^{-\lam\Aopu}:=(e^{-\lam z})(\Aopu) \in \BLO(X).
\end{align*}
The family $(e^{-\lam\Aopu})_{\lam\in\Sector_{\pi/2-\ome}}$ is a bounded holomorphic semigroup \cite[Prop.~3.4.1]{Haase06book}, called the \textbf{holomorphic semigroup} generated by $-\Aopu$.
There is a one-to-one correspondence between sectorial operators and bounded holomorphic semigroups \cite[Prop.~3.4.4]{Haase06book}.

\subsection{Operators in Hilbert spaces}

Let $H$ be a Hilbert space over $\bbC$ with scalar product $\spl\plh,\plh\spr$ and associated norm $\|\plh\|$.

\subsubsection{M-Accretive operators}

We denote the \textbf{numerical range} of an operator $\Aopu\colon\opdom(\Aopu)\subset H\to H$ as
\begin{align*}
\NRW(\Aopu):=\{ \spl \Aopu \afu,\afu \spr\colon \afu\in \opdom(\Aopu), \|\afu\|=1\}.
\end{align*}
An operator $\Aopu$ is called \textbf{accretive}, if it's numerical range is contained in the closed right half-plane, i.e., $\NRW(\Aopu)\subset\overline{\Sector_{\pi/2}}$.
Note that an accretive operator does not have to be closed, which motivates the following stronger definition.
An accretive operator $\Aopu$ is called \textbf{m-accretive}, if in addition its resolvent set contains the point $-1$, i.e., $-1\in\rho(\Aopu)$.
Two equivalent characterization of m-accretive operators are
\begin{align*}
&\{z\colon\Re z<0\}\subset\rho(\Aopu) \qquad\text{and}\qquad \|(\lam+\Aopu)^{-1}\|\leq (\Re \lam)^{-1} \text{ for all } \Re\lam>0;\\
&-\Aopu \text{ generates a } C_0\text{-semigroup } (T(t))_{t\geq0} \text{ with }\|T(t)\|\leq1 \text{ for all } T\geq0.
\end{align*}
Furthermore, for $\theta\in[0,\pi/2]$ an operator $\Aopu$ is called \textbf{$\theta$-accretive}, if $\NRW(\Aopu)\subset\overline{\Sector_\theta}$, i.e.,
\begin{align*}
|\Im \spl \Aopu\afu,\afu \spr| \leq (\tan\theta) \Re \spl \Aopu\afu,\afu \spr
\quad\text{for all } \afu\in\opdom(\Aopu).
\end{align*}
A $\theta$-accretive operator $\Aopu$ is called \textbf{m-$\theta$-accretive}, if in addition $\ran(\Aopu+1)$ is dense in the Hilbert space.
Again $\Aopu$ is called \textbf{quasi (m-)($\theta$-)accretive}, if there exists $\lam\in\bbR$ such that $\Aopu+\lam$ is so.
We have the following.
\begin{proposition}[Prop.~7.1.1 of \cite{Haase06book}]
Let $\Aopu\colon\opdom(\Aopu)\subset H\to H$ be a linear operator, $\ome\in[0,\pi/2)$ and $\theta:=\pi/2-\ome$.
The following assertions are equivalent:
\begin{enumerate}
\item The operator $\Aopu$ is m-$\ome$-accretive.
\item The operators $e^{\pm i\theta}\Aopu$ are m-accretive.
\item The operator $-\Aopu$ generates a holomorphic $C_0$-semigroup $(T(z))_{z\in\Sector_\theta}$ on $\Sector_\theta$ such that $\|T(z)\|\leq1$ for all $z\in\Sector_\theta$.
\end{enumerate}
If one of these equivalent conditions is satisfied, then $\Aopu\in\Sect(\ome)$.
\end{proposition}

\subsubsection{Elliptic sesquilinear forms}

The \textbf{adjoint form} $\overline{a}(\plh,\plh)$ of a sesquilinear form $a(\plh,\plh)$ is defined as
\begin{align*}
\overline{a}(\afu,\afu\tf):=\overline{a(\afu\tf,\afu)}.
\end{align*}
The \textbf{real and imaginary parts} of $a(\plh,\plh)$ are defined as
\begin{align*}
\Re a:=\frac12(a+\overline{a}) \quad\text{and}\quad\Im a:=\frac{1}{2i}(a-\overline{a}).
\end{align*}
A form $a(\plh,\plh)$ is called \textbf{symmetric} if $\Im a=0$.
The \textbf{quadratic form} of a sesquilinear form $a(\plh,\plh)$ is defined as
\begin{align*}
a(\afu):=a(\afu,\afu),
\end{align*}
for which it holds that $(\Re a)(\afu)=\Re(a(\afu))$, $(\Im a)(\afu)=\Im(a(\afu))$.
The \textbf{numerical range} of a sesquilinear form $a(\plh,\plh)$ with domain $V\subset H$ is defined as
\begin{align*}
\NRW(a):=\{a(\afu)\colon \afu\in V, \|\afu\|=1\}.
\end{align*}
A form $a(\plh,\plh)$ is called \textbf{form sectorial} if $\NRW(a)$ is contained in an acute sector $\Sector_\theta$, $\theta\in[0,\pi/2)$.
A form $a(\plh,\plh)$ is called \textbf{quasi form sectorial} if there exists $\lam\in\bbR$ such that the form $a(\plh,\plh)+\lam\spl\plh,\plh\spr$ is form sectorial.%
\footnote{Kato's \cite{Kato95} definition of sectorial equals our definition of quasi sectorial.}\\

Let $V\subset H$ be a Hilbert space with scalar product $\spl\plh,\plh\spr_V$ and associated norm $\|\plh\|_V$ such that the inclusion $V\hookrightarrow H$ is continuous with dense range.
A sesquilinear form $a(\plh,\plh)$ with domain $V$ is called \textbf{elliptic}, if there exist $\lam_0\in\bbR$, $\del>0$ such that
\begin{subequations}
\begin{align}
\label{eq:def_elliptic_garding}
(\Re a)(\afu)+\lam_0\|\afu\|^2 &\geq \del \|\afu\|_V^2 \quad\text{for all } \afu\in V;\\
|a(\afu,\afu\tf)| &\leq M \|\afu\|_V \|\afu\tf\|_V \quad\text{for some } M>0 \text{ and all } \afu,\afu\tf\in V,
\end{align}
\end{subequations}
i.e., satisfies a \Garding inequality on $V$ and $a(\plh,\plh)$ is continuous on $V$.
We note that an elliptic form is quasi form sectorial.
We have the following result, which roughly equals Kato's \emph{first representation theorem} \cite[Thm.~VI.2.1]{Kato95}.
\begin{proposition}[Prop.~7.3.4 of \cite{Haase06book}]
\label{prop:repr_thm_first}
Let $a(\plh,\plh)$ be an elliptic form (with domain $V$) and $\lam_0$ as in \eqref{eq:def_elliptic_garding}.
Then, there exist $\Aopu\colon\opdom(\Aopu)\subset V\to H$ and $\ome\in[0,\pi/2)$ such that $\Aopu+\lam_=$ is an invertible m-$\ome$-accretive operator which satisfies $\spl\Aopu\afu,\afu\tf\spr=a(\afu,\afu\tf)$ for all $\afu\in\opdom(\Aopu)$, $\afu\tf\in V$.\\
Conversely, let $\Aopu\colon\opdom(\Aopu)\subset H\to H$ be an m-$\ome$-accretive operator for some $\ome\in[0,\pi/2)$.
Then, there exist an $H$-dense subspace $V\subset \opdom(\Aopu)$ and an elliptic form $a(\plh,\plh)$ such that $\spl\Aopu\afu,\afu\tf\spr=a(\afu,\afu\tf)$ for all $\afu\in\opdom(\Aopu)$, $\afu\tf\in V$;
The form $\Re a$ is positive.
\end{proposition}
An operator $\Aopu\colon\opdom(\Aopu)\subset H\to H$ is called \textbf{Kato-sectorial} if it is generated by an elliptic form (in the sense of \Cref{prop:repr_thm_first}).
We have the following corollary for Kato-sectorial operators.
\begin{corollary}[Cor.~7.3.5 of \cite{Haase06book}]
\label{cor:Kato_sectorial}
An operator $\Aopu\colon\opdom(\Aopu)\subset H\to H$ is Kato-sectorial if and only if there is $\lam\in\bbR$ such that $\Aopu+\lam$ is m-$\ome$-accretive for some $\ome\in[0,\pi/2)$ if and only if $-\Aopu$ generates a holomorphic semigroup $T$ on some sector $\Sector_\theta$ such that $\|T(z)\|\leq e^{\lam\Re z}$ for all $z\in\Sector_\theta$ for some $\lam\in\bbR$.
\end{corollary}
We also recall Kato's \emph{second representation theorem}.
\begin{proposition}[Thm.~VI.2.1 of \cite{Kato95}]
\label{prop:repr_thm_second}
Let $a(\plh,\plh)$ be an elliptic form (with domain $V$), which is symmetric and non-negative.
Let $\Aopu$ be the associated operator (by means of \Cref{prop:repr_thm_first}).
Then, $\opdom(\Aopu^{1/2})=V$ and $a(\afu,\afu\tf)=\spl\Aopu^{1/2}\afu,\Aopu^{1/2}\afu\tf\spr$ for all $\afu,\afu\tf\in V$.
\end{proposition}

\subsubsection{Perturbations on the form level}\label{subsubsec:perturbed_form}

Let $a(\plh,\plh)$ be an elliptic form (with domain $V$), which is symmetric and positive.
Let $b(\plh,\plh)$ be a sesquilinear form with domain $V$, which is form subordinate to $a(\plh,\plh)$, i.e., there exist $C>0$ and $p\in[0,1)$ such that
\begin{align*}
|b(\afu)|\leq C a(\afu)^p \|\afu\|^{2(1-p)} \quad\text{for all }\afu\in V.
\end{align*}
Then $b(\plh,\plh)$ is relatively bounded with respect to $a(\plh,\plh)$ with the bound $0$ in particular, i.e.,
\begin{align*}
\forall\delta>0 \,\,\exists C_\delta\geq0 \quad\text{such that}\quad
|b(\afu)| \leq \delta a(\afu)+C_\delta \|\afu\|^2 \quad\forall\afu\in V.
\end{align*}
It follows that the form $\qses(\plh,\plh):=\ases(\plh,\plh)+\bses(\plh,\plh)$ is elliptic.
Let $\Aopu\colon\opdom(\Aopu)\subset H\to H$ and $\Qopu\colon\opdom(\Qopu)\subset H\to H$ be the operators associated to $a(\plh,\plh)$ and $q(\plh,\plh)$ (by means of \Cref{prop:repr_thm_first,prop:repr_thm_second}).
For $\lam\notin\spec(\Aopu)$ let $(\lam-\Aopu)^{-1/2}$ be defined by means of $\sqrt{z}:=\sqrt{|z|}e^{\frac{i\arg z}{2}}$, $\arg z\in(-\pi,\pi]$, $z\in\Cx$, and
$\Bop(\lam)\in L(H)$ be the Riesz representation of the $H$-bounded sesquilinear form 
\begin{align*}
b((\lam-\Aopu)^{-1/2}\afu,(\lam-\Aopu)^{-*/2}\afu\tf)=\spl \Bop(\lam)\afu,\afu\tf\spr_{H} \quad \forall \afu,\afu\tf\in H.
\end{align*}
Due to \cite[Lem.~1]{Agranovich94} we have that for $\lam\notin\spec(\Aopu)$
\begin{align*}
\lam-\Qopu=(\lam-\Aopu)^{1/2} (I-\Bop(\lam)) (\lam-\Aopu)^{1/2}
\end{align*}
and if $I-\Bop(\lam)$ is invertible also that
\begin{align*}
(\lam-\Qopu)^{-1}=(\lam-\Aopu)^{-1/2} (I-\Bop(\lam))^{-1} (\lam-\Aopu)^{-1/2}.
\end{align*}

\printbibliography[heading=bibintoc]

\end{document}

%% file: header-paper.tex
\usepackage[utf8]{inputenc}
\usepackage[ngerman, english]{babel}
\usepackage{amsmath,amsthm,amssymb,amsfonts}
\usepackage{mathrsfs}
\usepackage{ifthen}
\usepackage{enumerate}
\usepackage{comment}
\usepackage{subcaption}
\usepackage{hyphsubst}
\usepackage{cleveref}
\usepackage[dvipsnames]{xcolor}
\usepackage[skins]{tcolorbox}
\usepackage{calc}
\usepackage{graphicx}
\usepackage{tabularx}
\usepackage[toc,page]{appendix}
\usepackage[bbgreekl]{mathbbol}
\DeclareSymbolFontAlphabet{\mathbbm}{bbold}
\DeclareSymbolFontAlphabet{\mathbb}{AMSb}

\usepackage{calc}
\usepackage{leftindex}

\usepackage{pgfplots}
\pgfplotsset{compat=1.5.1}

\usepackage{a4wide}
\usepackage{fancyhdr}
\usepackage[backend=biber,
	giveninits=true,
	maxbibnames=8,
	date=year,
	doi=true,
	url=false,
	isbn=false,
	eprint=false]{biblatex}
\AtEveryBibitem{\clearlist{language}} 
\renewbibmacro{in:}{%
  \ifboolexpr{%
     test {\ifentrytype{article}}%
     or
     test {\ifentrytype{inproceedings}}%
  }{}{\printtext{\bibstring{in}\intitlepunct}}%
}

\makeatletter
\def\@hspace#1{\begingroup\setlength\dimen@{#1}\hskip\dimen@\endgroup}
\makeatother

\newtheorem{theorem}{Theorem}[section]
\newtheorem{lemma}[theorem]{Lemma}
\newtheorem{definition}[theorem]{Definition}

\newtheorem{remark}[theorem]{Remark}
\newtheorem{corollary}[theorem]{Corollary}
\newtheorem{assumption}[theorem]{Assumption}
\newtheorem{proposition}[theorem]{Proposition}
\newtheorem{problem}[theorem]{Problem}
\newtheorem{evproblem}[theorem]{Eigenvalue Problem}

%% file: macros-general.tex
\definecolor{lightgray}{gray}{0.75}
\definecolor{prussianblue}{rgb}{0.0, 0.19, 0.33}
\definecolor{palecopper}{rgb}{0.85, 0.54, 0.4}
\definecolor{redbg}{rgb}{1.0,0.2,0.2}
\definecolor{brickred}{rgb}{0.8, 0.25, 0.33}
\definecolor{bostonuniversityred}{rgb}{0.8, 0.0, 0.0}
\definecolor{cornellred}{rgb}{0.7, 0.11, 0.11}
\definecolor{corn}{rgb}{0.98, 0.93, 0.36}
\definecolor{schoolbusyellow}{rgb}{1.0, 0.85, 0.0}
\definecolor{TUblue}{rgb}{0,102,153}
\colorlet{TUbluelight}{TUblue!30!white}

\newcommand\tf{^\dagger}
\newcommand\dd{\mathrm{d}}
\newcommand{\plh}{\cdot}
\def\spl{\langle}
\def\spr{\rangle}
\def\bpm{\begin{pmatrix}}
\def\epm{\end{pmatrix}}

\newcommand\ol[1]{\overline{#1}}
\newcommand\Nz{\mathbb{N}_0}

\newcommand\Rp{\mathbb{R}^+}
\newcommand\Rpz{\mathbb{R}^+_0}

\newcommand\Rx{\mathbb{R}_\times}
\newcommand\Cx{\mathbb{C}_\times}
\newcommand{\eps}{\epsilon}
\newcommand{\del}{\delta}
\newcommand{\lam}{\lambda}
\newcommand{\bflam}{\boldsymbol{\lambda}}
\newcommand{\gam}{\gamma}
\newcommand{\ome}{\omega}
\newcommand{\muinv}{\mu^{-1}}

\def\opdom{D}

\renewcommand{\Re}{\operatorname{Re}}
\renewcommand{\Im}{\operatorname{Im}}

\DeclareMathOperator{\curl}{curl}
\DeclareMathOperator{\Curl}{Curl}
\DeclareMathOperator{\diag}{diag}

\DeclareMathOperator{\dist}{dist}
\renewcommand{\div}{\operatorname{div}}

\DeclareMathOperator{\EtM}{EtM}

\DeclareMathOperator{\ran}{ran}
\DeclareMathOperator{\sgn}{sgn}
\DeclareMathOperator{\Span}{span}
\DeclareMathOperator{\spec}{spec}
\DeclareMathOperator{\supp}{supp}
\DeclareMathOperator{\tr}{tr}
\def\trtan{\tr_{T}}

\newtcbox{\mybox}{on line,
  colframe=white!10!white,colback=Goldenrod,
  boxrule=0.5pt,arc=4pt,boxsep=0pt,left=6pt,right=6pt,top=6pt,bottom=6pt}
\newtcbox{\myboxii}{on line,
  colframe=white!10!white,colback=Goldenrod,
  boxrule=0.5pt,arc=4pt,boxsep=0pt,left=6pt,right=6pt,top=5pt,bottom=1pt}
\newtcbox{\myboxiii}{on line,
  colframe=white!10!white,colback=Goldenrod,
  boxrule=0.5pt,arc=4pt,boxsep=0pt,left=2pt,right=2pt,top=2pt,bottom=2pt}
\newtcbox{\myredbox}{on line,fontupper=\tiny,
  colframe=white,colback=redbg!70!white,
  boxrule=0.0pt,arc=0pt,boxsep=2pt,left=0pt,right=0pt,top=0pt,bottom=0pt}
\newtcbox{\myboxnofill}{on line,
  colframe=NavyBlue,colback=white,
  boxrule=0.5pt,arc=4pt,boxsep=0pt,left=6pt,right=6pt,top=6pt,bottom=6pt}

\def\Calderon{Calder\'{o}n }

\def\Garding{G\aa rding }

\newcommand{\bbC}{\mathbb{C}}

\newcommand{\bbN}{\mathbb{N}}

\newcommand{\bbR}{\mathbb{R}}

\newcommand{\bfn}{\mathbf{n}}

\newcommand{\bfq}{\mathbf{q}}

\newcommand{\bft}{\mathbf{t}}
\newcommand{\bfu}{\mathbf{u}}
\newcommand{\bfv}{\mathbf{v}}

\newcommand{\bfx}{\mathbf{x}}

\newcommand{\bfC}{\mathbf{C}}

\newcommand{\bfE}{\mathbf{E}}

\newcommand{\bfH}{\mathbf{H}}

\newcommand{\bfL}{\mathbf{L}}

\newcommand{\bfV}{\mathbf{V}}

\newcommand{\calE}{\mathcal{E}}

\newcommand{\calH}{\mathcal{H}}

\newcommand{\calK}{\mathcal{K}}

\newcommand{\calM}{\mathcal{M}}

\newcommand{\calO}{\mathcal{O}}

\newcommand{\calV}{\mathcal{V}}
\newcommand{\calW}{\mathcal{W}}

\newcommand{\fraks}{\mathfrak{s}}


%% file: macros-emwg.tex
\makeatletter
\DeclareRobustCommand{\Hcrossbeam}{%
  \hmode@bgroup
  \vphantom{H}%
  \sbox\z@{H}%
  \ooalign{%
    $H$\cr
    \hidewidth
	\hspace{0.25\ht\z@}
    \vrule
      height \dimexpr 0.7\ht\z@+0.1ex\relax
      depth  -0.7\ht\z@ 
      width  0.4\wd\z@ 
    \hidewidth\cr
  }%
  \egroup
}
\newcommand{\Hvar}{\Hcrossbeam}

\newcommand{\Npr}{N_\mathrm{pr}} 
\newcommand{\tNpr}{\tilde N_\mathrm{pr}} 
\newcommand{\afu}{u} 
\def\geo{\Omega}
\def\geocs{\Gamma}
\def\geoint{\geo_\mathrm{i}}
\def\geoext{\geo_\mathrm{e}}

\newcommand{\geobddir}{\geocs_D}
\newcommand{\geobdtrans}{\geocs_T}
\newcommand{\bci}{\geobddir} 
\newcommand{\geocsbd}{\Sigma}
\newcommand{\Ngeocsbd}{N_{\geocsbd}}
\def\nvf{\bfn}
\def\nvt{\hat\bfn}
\def\tv{\hat\bft}

\def\trtanR{\tr_{RT}}
\def\rotm{R}
\def\xf{\bfx}
\def\xt{\hat\bfx}

\def\epsmax{\epsilon_\mathrm{max}}
\def\epsmin{\epsilon_\mathrm{min}}
\def\mumax{\mu_\mathrm{max}}
\def\mumin{\mu_\mathrm{min}}
\newcommand{\epscs}{\eps_{\geocs}}
\newcommand{\mucs}{\mu_{\geocs}}
\newcommand{\mucsinv}{\mucs^{-1}}
\def\gradF{\nabla_\geo}
\def\curlF{\curl_\geo}
\def\divF{\div_\geo}
\def\gradg{\nabla_{\geocs}}
\def\divg{\div_{\geocs}}
\def\curlg{\curl_{\geocs}}
\def\Curlg{\Curl_{\geocs}}

\def\Eff{\mathcal{E}}
\def\Hff{\mathcal{H}}
\def\Uff{\mathcal{U}}
\newcommand{\Fff}{\mathcal{F}}
\def\Ef{\mathbf{E}}
\def\Hf{\mathbf{H}}
\def\Uf{\mathbf{U}}
\def\Et{\mathbf{E}_T}
\def\Ht{\mathbf{H}_T}
\def\Ut{\mathbf{U}_T}
\def\En{E_N}
\def\Hn{H_N}
\def\Un{U_N}

\def\Ux{U_x}
\def\Uy{U_y}
\def\Uz{U_z}

\def\bv{\mathbf{v}}
\newcommand{\cf}{\phi} 
\def\BLO{\mathcal{L}} 
\newcommand{\tloc}{\widetilde{\mathrm{loc}}}
\def\Lspace{\bfL} 
\def\Hspace{\bfH} 
\def\Xspace{X}
\def\Xspaceu{\underline{\Xspace}}

\def\Vspace{\mathbf{V}}
\def\Wspace{W}

\newcommand{\calop}{\operatorname{\mathfrak{C}}}
\newcommand{\Tscbd}{T_{\geobdtrans}}
\newcommand{\Tscvol}{T_{\geoint}}
\newcommand{\Vint}{\bfV_{\geoint}}
\newcommand{\Eint}{\calE}
\newcommand{\vint}{\calV}
\newcommand{\wint}{\calW}
\newcommand{\asesint}{a_{\geoint}}
\newcommand{\Aopint}{A_{\geoint}}
\def\Meps{M_\epsilon}
\def\Mmui{M_{\mu^{-1}}}
\def\MepsVec{\mathbf{M}_\epsilon}
\def\MmuiVec{\mathbf{M}_{\mu^{-1}}}
\def\embL{E_{L^2}}
\def\embLVec{\mathbf{E}_{L^2}}
\def\DGrad{\mathbf{D}_g} 
\def\Dcurl{D_c} 
\def\Aop{A}
\def\Bop{B}
\def\Kop{K}
\newcommand{\Qop}{Q}
\newcommand{\Gop}{G}

\newcommand{\Aopb}{\Aop} 
\newcommand{\Bopb}{\Bop} 
\newcommand{\Kopb}{\Kop} 
\newcommand{\Sop}{S} 
\def\ases{a}
\def\bses{b}
\def\qses{q}
\newcommand{\Aopu}{\underline{\Aop}}

\newcommand{\Qopu}{\underline{\Qop}}
\newcommand{\Gopu}{\underline{\Gop}}

\def\tmp{\#}
\def\Aoptmp{\Aop_\tmp} 
\def\Koptmp{\Kop_\tmp} 
\newcommand{\Sector}{S}
\DeclareMathOperator{\Sect}{Sect} 
\newcommand{\Mod}{M} 
\newcommand{\NRW}{W} 

